\documentclass[12pt,reqno]{amsart}
\usepackage[margin=1in]{geometry}
\usepackage{amsmath,amssymb,amsthm,graphicx,amsxtra, setspace}
\usepackage[utf8]{inputenc}
\usepackage{mathrsfs}
\usepackage{hyperref}
\usepackage{upgreek}
\usepackage{mathtools}
\usepackage[dvipsnames]{xcolor}
\usepackage[mathcal]{euscript}
\usepackage{amsmath,units}
\usepackage{pgfplots}
\usepackage{tikz}
\usepackage{verbatim}
\allowdisplaybreaks
\usepackage{dsfont}
\usepackage{fontenc}
\usepackage{textcomp}
\usepackage{marvosym}
\usepackage{eurosym}
\usepackage{upgreek}
\usepackage[pagewise]{lineno}

\DeclareMathAlphabet{\mathpzc}{OT1}{pzc}{m}{it}

\usepackage[cyr]{aeguill}

\colorlet{darkblue}{blue!50!black}

\hypersetup{
	colorlinks,%
	citecolor=blue,%
	filecolor=red,%
	linkcolor=red,%
	urlcolor=blue,%
	pdfnewwindow=true,%
	pdfstartview={FitH}
}

\newtheorem{theorem}{Theorem}[section]
\newtheorem{lemma}[theorem]{Lemma}

\newtheorem{corollary}[theorem]{Corollary}
\newtheorem{definition}[theorem]{Definition}
\newtheorem{example}[theorem]{Example}
\newtheorem{remark}[theorem]{Remark}

\let\originalleft\left
\let\originalright\right
\renewcommand{\left}{\mathopen{}\mathclose\bgroup\originalleft}
\renewcommand{\right}{\aftergroup\egroup\originalright}

\newcommand{\Tr}{\mathop{\mathrm{Tr}}}
\renewcommand{\d}{\/\mathrm{d}\/}

\def\w{\textbf{W}^{\varepsilon}_{{\theta}^{\varepsilon}}}

\def\e{\varepsilon}

\def\S{\mathrm{S}}

\def\L{\mathbb{L}}
\def\A{\mathrm{A}}
\def\I{\mathrm{I}}
\def\F{\mathrm{F}}
\def\C{\mathrm{C}}
\def\f{\boldsymbol{f}}

\def\B{\mathrm{B}}
\def\D{\mathrm{D}}
\def\y{\boldsymbol{y}}

\def\Y{\mathrm{Y}}
\def\Z{\boldsymbol{Z}}

\def\E{\mathbb{E}}
\def\X{\boldsymbol{X}}
\def\Y{\boldsymbol{Y}}
\def\x{\boldsymbol{x}}

\def\g{\boldsymbol{g}}

\def\h{\boldsymbol{h}}
\def\z{\boldsymbol{z}}
\def\v{\boldsymbol{v}}
\def\V{\mathbb{v}}
\def\w{\boldsymbol{w}}
\def\W{\mathrm{W}}
\def\G{\mathrm{G}}
\def\Q{\mathrm{Q}}

\def\N{\mathbb{N}}

\def\no{\nonumber}
\def\V{\mathbb{V}}
\def\wi{\widetilde}
\def\Q{\mathrm{Q}}
\def\u{\mathrm{U}}
\def\P{\mathrm{P}}
\def\u{\boldsymbol{u}}
\def\H{\mathbb{H}}

\newcommand{\eps}{\varepsilon}

\newcommand{\R}{\mathbb{R}}

\renewcommand{\d}{\/\mathrm{d}\/}

\newcommand{\Addresses}{{
		\footnote{
			
			\noindent \textsuperscript{1,2}Department of Mathematics, Indian Institute of Technology Roorkee-IIT Roorkee,
			Haridwar Highway, Roorkee, Uttarakhand 247667, INDIA.\par\nopagebreak
			\noindent  \textit{e-mail:} \texttt{Manil T. Mohan: maniltmohan@ma.iitr.ac.in, maniltmohan@gmail.com.}
			
			\textit{e-mail:} \texttt{Sagar Gautam: sagar\_g@ma.iitr.ac.in.}
			
			\noindent \textsuperscript{*}Corresponding author.
			
			\textit{Key words:}  Convective Brinkman-Forchheimer equations,  invariant measure, Kolmogorov equation, carr\'e du champs identity, infinite horizon  problems, optimal stopping time  problem. 
			
			Mathematics Subject Classification (2020): Primary 93B05, 35R15; Secondary 60H30, 37L40, 76D03.
			
}}}

\begin{document}
	
	
	\title[Kolmogorov equations for 2D SCBF equations]{Kolmogorov equations for 2D stochastic convective Brinkman-Forchheimer equations: Analysis and Applications
			\Addresses}
		\author[S. Gautam and M. T. Mohan]
		{Sagar Gautam\textsuperscript{1} and Manil T. Mohan\textsuperscript{2*}}

	\maketitle
	\begin{abstract}
		In this work, we consider the following 2D stochastic convective Brinkman-Forchheimer (SCBF) equations  in a bounded smooth domain $\mathcal{O}$: 
		\begin{align*}
			\mathrm{d}\boldsymbol{u}+\left[-\mu \Delta\boldsymbol{u}+(\boldsymbol{u}\cdot\nabla)\boldsymbol{u}+\alpha\boldsymbol{u}+\beta|\boldsymbol{u}|^{r-1}\boldsymbol{u}+\nabla p\right]\mathrm{d}t=\sqrt{\mathrm{Q}}\mathrm{W}, \ \nabla\cdot\boldsymbol{u}=0, 
		\end{align*}
		where $\mu,\alpha,\beta>0$, $r\in\{1,2,3\}$, $\mathrm{Q}$ is a non-negative operator of trace class, $\mathrm{W}$ is a cylindrical Wiener process in a Hilbert space $\mathbb{H}$.  Under the following  assumption on  the viscosity co-efficient $\mu$  and the Darcy co-efficient $\alpha$: 	for some positive constant $\gamma_1$, 
		\begin{equation*} 
			\mu(\mu+\alpha)^2>\gamma_1\max\{4\Tr(\Q),\Tr(\A^{2\delta}\Q)\},
			\end{equation*}
	where $\A$ is the Stokes operator and $\delta\in(0,\frac{1}{2})$, our primary goal is to solve the corresponding Kolmogorov equation in the space	$\mathbb{L}^2(\H;\eta),$ where $\eta$ is the unique invariant measure associated with 2D SCBF equations. Then, we  establish the well-known ``carr\'e du champs'' identity.  Some sharp estimates on the derivatives of the solution constitute the key component of the proofs.  We take into consideration two control problems from the application point of view. The first is an infinite horizon control problem for which we establish the existence of a solution for the Hamilton-Jacobi-Bellman equation associated with it. Finally, by exploiting  $m$-accretive theory, we demonstrate the existence of a unique solution for an obstacle problem associated with the Kolmogorov operator corresponding to the stopping-time problem for 2D SCBF equations.
	\end{abstract}

	\section{Introduction}\label{sec1}\setcounter{equation}{0}
	This article studies the Kolmogorov equation associated with the two-dimensional stochastic convective Brinkman-Forchheimer (SCBF) and their control theoretic applications. 
 	\subsection{The model}
	The convective Brinkman-Forchheimer (CBF) equations describe the motion of incompressible fluid flows in a saturated porous medium. These equations are applicable when the fluid flow rate is sufficiently high and the porosity is not too small. Let us first provide a mathematical formulation of CBF equations. Let $\mathcal{O}\subset\R^2$ be a bounded domain with a smooth boundary $\partial\mathcal{O}$. We denote the velocity field by $\Y(t,\xi)\in\R^2$ at time $t\in[0,T]$, $T<\infty$, and position $\xi\in\mathcal{O}$ and the pressure field by $p(t,\xi)\in\R$. The SCBF  equations driven by an additive Gaussian noise  are given by
   \begin{equation}\label{1}
   \left\{
	\begin{aligned}
	\frac{\partial \Y(t,\xi)}{\partial t}-\mu \Delta\Y(t,\xi)+(\Y(t,\xi)\cdot\nabla)&\Y(t,\xi)+\alpha\Y(t,\xi)+\beta|\Y(t,\xi)|^{r-1}\Y(t,\xi)\\+\nabla p(t,\xi)&=\sqrt{\Q}\d\W(t), \ \text{ in } \ (0,T)\times\mathcal{O}, \\ \nabla\cdot\Y(t,\xi)&=0, \ \text{ in } \ [0,T]\times\mathcal{O}, \\
	\Y(t,\xi)&=\boldsymbol{0},\ \text{ on } \ [0,T]\times\partial\mathcal{O}, \\
	\Y(0,\xi)&=\y(\xi) \ \text{ in } \ \mathcal{O},
	\end{aligned}
	\right.
	\end{equation}
where $\{\W(t)\}_{t\geq0}$ is an $\H$-valued cylindrical Wiener process defined on a filtered probability space  $(\Omega,\mathscr{F},\{\mathscr{F}_t\}_{t\geq0},\mathbb{P})$  and $\Q$ is a trace class operator. For the uniqueness of the pressure, one can impose the condition $
\int_{\mathcal{O}}p(t,\xi)\d\xi=0,  \text{ in }  [0,T].$ The constant $\mu>0$ denotes the \emph{Brinkman coefficient} (effective viscosity) and the constants $\alpha$ and $\beta$ are due to Darcy-Forchheimer law which are termed as \emph{Darcy} (permeability of the porous medium) and \emph{Forchheimer} (proportional to the porosity of the material) coefficients, respectively.For $\alpha=\beta=0$, one obtains the classical 2D stochastic Navier-Stoke equations (SNSE).   The parameter $r\in[1,\infty)$ is known as the \emph{absorption exponent}  and the case  $r=3$ is referred as the critical exponent (\cite{KWH}). 
For $\y\in\H$,   by exploiting a monotonicity property of the linear and nonlinear operators as well as a stochastic generalization of the Minty-Browder technique, the author in \cite{kkmtm,MTM8} established the existence and uniqueness of a global strong solution $$\Y\in\C([0,T];\H)\cap\mathrm{L}^2(0,T;\V)\cap\mathrm{L}^{r+1}(0,T;\wi\L^{r+1}),\ \mathbb{P}\text{-a.s.},$$  satisfying the energy equality (It\^o's formula) for SCBF equations (in bounded and torus) driven by multiplicative Gaussian noise. 

 \subsubsection{Kolmogorov equations} \setcounter{equation}{0}
 It is well-known that the solution of an It\^o equation in finite dimensions is a diffusion process and the expectation of a smooth function of such a solution satisfies a diffusion equation in $\R^d$, known as \emph{Kolmogorov equations} (see \cite[Chapters 7,8]{oks}). Therefore, exploring such a relationship for stochastic partial differential equations (SPDE) is natural. A. N. Kolmogorov examined these equations for the first time in 1931 when he examined transition probabilities that satisfy \emph{Chapman-Kolmogorov equations} (see \cite{ANK1}) in the context of the Markov processes. These are commonly called the \emph{Kolmogorov backward (or forward) equations} in the literature. 
  In this article, we shall mainly concentrate on infinite-dimensional generalization of the (backward) Kolmogorov equations associated with SCBF equations driven an  additive Gaussian noise. Furthermore, our work focuses on the elliptic counterpart of Kolmogorov equations.
 The infinite-dimensional Kolmogorov equations have been intensively studied in the literature for the past three decades (cf.  \cite{SC1,gdp7,mr13,mr14,jz13}). 
 Regularity analysis for  Kolmogorov equations is one of the very important problems. The authors in \cite{MH1} recently discussed the loss of regularity phenomena for Kolmogorov PDEs in finite dimensions. 
Whereas, the authors in \cite{Aand} discussed regularity properties for solutions to infinite-dimensional Kolmogorov equations. They demonstrated that if the nonlinear drift and nonlinear diffusion coefficients and the initial function of the corresponding Kolmogorov equations are $n$-times continuously Fr\'echet differentiable, then so does their generalized solution at every positive time (see also \cite{ceb}).
The authors in \cite{VBGA} studied the essential $m$-dissipativity for 2D SNSE driven by an additive Gaussian  noise in the periodic domain. They later expanded this conclusion to a channel with periodic boundary conditions by incorporating space-time white noise (see \cite{VbGdp}). Moreover, for 2D stochastic Navier-Stokes-Coriolis equations in the periodic domain, the author in \cite{WSt} demonstrated the same conclusion in $\mathrm{L}^p$-spaces as opposed to Hilbert space.

\subsubsection{Applications} 
An important application of Kolmogorov equation revolves around the analysis of a family of nonlinear elliptic and parabolic equations on Hilbert spaces. They arise in relation to the control of infinite-dimensional stochastic systems. One of the crucial problems in this direction is finding solutions to an infinite-dimensional (second order) \emph{Hamilton-Jacobi-Bellman} (HJB) equation for the value function associated with a stochastic optimal control problem of distributed parameter systems. Usually, the HJB equations are solved by using transition semigroup and  fixed point arguments (cf. \cite{gdp7}).
This methodology has already been implemented in the case of HJB equations with Lipschitz (locally and globally) Hamiltonian; see the works of \cite{pcgd,pcgdp1}. In both of these works, the authors showed the existence and uniqueness of a global solution of class $\C^1$. Proceeding further, in the paper \cite{gozzi2}, the author showed the existence, uniqueness, and $\C^2$-regularity of a local solution for the second-order infinite-dimensional HJB equation with a global Lipschitz Hamiltonian. 
Furthermore, the authors in \cite{gPaD} solved the HJB equations for 2D SNSE  defined in a bounded domain via dynamic programming approach. 
Furthermore, for a stationary HJB equation, the authors in \cite{gozzi3} established the existence and uniqueness of a mild solution via maximal monotone operator theory.

Another interconnected application of Kolmogorov equations is \emph{optimal stopping problems}, which is highly applicable in stochastic analysis, finance, and control theory. It is the problem of finding the optimal stopping time to stop a given (Markov) process so that the cost associated with the process is minimum. 
The problem of characterizing the optimal stopping time can often be reduced to an equivalent problem in PDEs, called \emph{variational inequality}. The application of variational parabolic or elliptic inequalities to stopping time problems was first introduced by the authors in \cite{WHF, TT2}. The application of variational and quasi-varitional inequalities in optimal stopping problems was studied in the work \cite{AvF}. The authors in \cite{VBSSO} discussed the optimal stopping problem for 2D SNSE using infinite-dimensional analysis and examined the existence results for the corresponding infinite-dimensional variational inequality. 
The existence theory for parabolic variational inequalities related to finite- and infinite-dimensional diffusion processes in weighted $\mathrm{L}^2$-spaces is established by the author in \cite{VBCM} with respect to excessive measures associated with a transition semigroup. They relaxed the usual non-degeneracy assumptions on the diffusion coefficient using the theory of maximal monotone operators in Hilbert spaces. 
The finite horizon optimal stopping problem for an infinite-dimensional diffusion, driven by an SDE on a Hilbert space, is considered in the work \cite{MBCTD}. 
An intense treatment and further analysis on optimal stopping problems can be found in \cite{ABJL2}.

\subsection{Difficulties, strategies and novelties}
Let us briefly explain the main challenges  and difficulties encountered in this  work. Consider the following infinite-dimensional Kolmogorov equations (parabolic) associated with the SCBF equations \eqref{32}:
\begin{equation}\label{kolmp}
	\left\{
	\begin{aligned}
		\frac{\partial\y(t,\x)}{\partial t} &=\frac{1}{2}\Tr\left[\Q\D_{\x}^2\y(t,\x)\right]-(\mu\A\x+\alpha\x+\B(\x)+\beta\mathcal{C}(\x),\D_{\x}\y(t,\x)), \\
		\y(0,\x)&=\varphi(\x), \ \ \x\in\H,
	\end{aligned}
	\right.
\end{equation}
where $\varphi(\cdot):\H\to\R$ belongs to some suitable function space, $\y(\cdot,\cdot):[0,T]\times\H\to\R$ is the unknown, $\D_{\x}$ represents the derivative with respect to $\x$ and $\Tr$ stands for the trace. Besides the parabolic equations \eqref{kolmp}, we shall also consider the following elliptic Kolmogorov equations:
\begin{align}\label{kolme}
	\lambda\widetilde{\y}(\x)-\frac{1}{2}\Tr\left[\Q\D_{\x}^2\widetilde{\y}(\x)\right]+(\mu\A\x+\alpha\x+\B(\x)+\beta\mathcal{C}(\x),\D_{\x}\widetilde{\y}(\x))=\mathfrak{f}(\x),
\end{align}
where $\lambda>0$ and $\mathfrak{f}(\cdot):\H\to\R$, are given, and $\widetilde{\y}(\cdot):\H\to\R$ is unknown. It is well-known in the literature that the solution $\X(\cdot,\cdot)=\mathrm{P}_{\H}\Y(\cdot,\cdot)$ of the system \eqref{32} ($\mathrm{P}_{\H}:\L^2(\mathcal{O})\to\H$ is the Helmholtz-Hodge orthogonal projection) and the solutions of Kolmogorov equations \eqref{kolmp}-\eqref{kolme}, are closely related via following formal identities:
\begin{align*}
	\y(t,\x)=\E[\varphi(\X(t,\x))], \ t\geq0, \ \x\in\H,
\end{align*}
and 
\begin{align*}
	\widetilde{\y}(\x)=\int_0^{\infty} e^{-\lambda t} \E[\mathfrak{f}(\X(t,\x))]\d t, \ \x\in\H,
\end{align*}
respectively (cf. \cite{EANK}). However, showing that $\y(t,\x)=\E[\varphi(\X(t,\x))]$ is a solution (the precise sense will be made later) to \eqref{kolmp} is one of the most challenging tasks of this work. The main difficulty is as follows: First (by a direct computation), we have to justify the following derivatives formulae for regular $\varphi$:
\begin{align*}
	(\D_{\x}\y(t,\x),\h)&=\E\left[\big(\D_{\x}\varphi(\X(t,\x)),\D_{\x}\X(t,\x)\h\big)\right],\\
	\D_{\x}^2\y(t,\x)(\h,\h)&=\E\left[\D_{\x}^2\varphi(\X(t,\x))
	\big(\D_{\x}\X(t,\x)\h,\D_{\x}\X(t,\x)\h\big)\right]\nonumber\\&\quad+
	\E\left[\big(\D_{\x}\varphi(\X(t,\x)),\D_{\x}^2\X(t,\x)(\h,\h)\big) \right],
\end{align*}
for $t\geq 0$ and $\x,\h\in\H$. Then, we use It\^o's formula to show that $\y(t,\x)=\E[\varphi(\X(t,\x))]$ is a solution of \eqref{kolmp}. But this is not easy, as the above derivative formulae involve the derivatives 
$\D_{\x}\X(t,\x)$ and $\D_{\x}^2\X(t,\x)$ of the solution of \eqref{32}, which are, in general, we do not know whether they are integrable or not. This is true only when the coefficients $\mathcal{B}(\cdot)$ and $\mathcal{C}(\cdot)$ are regular and $\Q$ is of trace class (see \cite[Theorem 7.5.1, Chapter 7]{gdp7}). 

To overcome this difficulty, we analyze the Kolmogorov equation \eqref{kolme} in the space $\mathrm{L}^2(\H,\eta)$, where $\eta$ is the invariant measure for the transition semigroup $\P_t$. There are plenty of results available in the literature regarding the existence and uniqueness of invariant measures (see \cite{gdp2,AD,MHJC} and references therein). The existence of invariant measures is ensured by energy estimates and the Krylov-Bogoliubov theorem (see \cite{gdp2}).  The uniqueness of the invariant measure requires more effort. The most common approach to establish the uniqueness is to show that the transition semigroup $\P_t$ is \emph{irreducibile} and \emph{strong Feller} (\cite{gdp2}). The authors in \cite{akmtm} used this method to demonstrate the uniqueness of invariant measure for 2D SCBF equations in the space $\D(\A^{\zeta})\subset\H$ for arbitrary $\zeta\in[\frac{1}{4},\frac{1}{2})$ for $r=1,2$ and $\zeta\in[\frac{1}{3},\frac{1}{2})$ for $r=3$. 
	In our work, we are getting the restriction on the viscosity $\mu$ (and Darcy coefficient $\alpha$, cf. \eqref{439}) while proving the uniqueness of the invariant measure in the $\H$-space by using exponential estimates on the first derivative $\boldsymbol{\xi}^{\boldsymbol{h}}(t,\x)=\D_{\x}\X(t,\x)\boldsymbol{h},$  for all $ \x, \boldsymbol{h}\in\H$. 
	
	Note that, unlike the whole space and periodic domain,  the major difficulty in working with bounded domains is that $\mathrm{P}_{\H}(|\Y|^{r-1}\Y)$  need not be zero on the boundary, and $\mathrm{P}_{\H}$ and $-\Delta$ are not necessarily commuting (see \cite{KT2}). Therefore, the equality 
\begin{align}\label{3}
	&\int_{\mathcal{O}}(-\Delta\Y(\xi))\cdot|\Y(\xi)|^{r-1}\Y(\xi)\d \xi \nonumber\\&=\int_{\mathcal{O}}|\nabla\Y(\xi)|^2|\Y(\xi)|^{r-1}\d \xi+ \frac{r-1}{4}\int_{\mathcal{O}}|\Y(\xi)|^{r-3}|\nabla|\Y(\xi)|^2|^2\d \xi,
\end{align}
may not be useful in the context of bounded domains. Therefore, we only restrict ourselves to 2 dimensions with $r\in\{1,2,3\}$. The fact that $(\B(\Y),\A\Y)=0$ in the torus was utilised by the authors in \cite{VBGA} (see \cite[Chapter 6]{gdp1}) to derive the essential $m$-dissipativity of the Kolmogorov operator for 2D SNSE  driven by an additive Gaussian noise. But in our case, even if we work on the torus, most of the calculations will remain the same as the one presented in this paper, since we still need to compute $(\mathcal{C}(\Y),\A\Y)$. However, as we are working on a bounded domain, so we must explicitly compute $(\B(\Y),\A\Y)$ as well as $(\mathcal{C}(\Y),\A\Y)$ due to the presence of the nonlinear damping term $|\Y|^{r-1}\Y$ in \eqref{1}. Even though in the work of \cite{VBGD}, the authors assumed  $\Tr(\A^{\delta}\Q)$ for $\delta>\frac{2}{3}$ to prove essential $m$-dissipativity, let us highlight the significance of the condition $\Tr(\A\Q)<\infty$ in our work. In the proof of the main result (Theorem \ref{thm4.7}), the inequality \eqref{traq} holds true only when  $\int_{\H}\|\A\x\|_{\H}^2\eta(\d\x)\leq C$ (see Lemma \ref{lem5.2}). 

Motivated from the work \cite{VBGD}, we first derive estimates of exponential moments. We mention here that the  authors in \cite{VBGD} assumed large viscosity relative to the operator norm $\|\Q\|_{\mathscr{L}(\H)}$ of the trace class operator $\Q$ to obtain a unique invariant measure. In their work (see Lemma 2.1-2.3 of \cite{VBGD}), the quantity $\|\Q^{1/2}\x\|_{\H}$ is further calculated as $\|\Q^{1/2}\x\|_{\H}^2\leq\|\Q\|_{\mathcal{L}(\H)}\|\x\|_{\H}^2$. In our case, we compute this expression by using the fact that $\|\Q^{\frac{1}{2}}\|_{\mathcal{L}(\H)}^2\leq\Tr(\Q)$. This approach allows us to derive the restriction on $\mu$ (and also on $\alpha$) in terms of $\Tr(\Q)$ only (see the calculations \eqref{trop}-\eqref{trop1}). Moreover, the presence of linear damping term $\alpha\Y$ in the model \eqref{1} helps us to prevent several restrictions on $\mu$.  Due to this, we simplify several calculations and we modify conditions on $\mu$ (up to a certain extent), which are much simpler as compared to the work of \cite{VBGD} (see Sections \ref{sec3}-\ref{sec4}). 
  
Then taking into account the invariance of $\eta$, one can uniquely extend the transition semigroup $\P_t$ to strongly continuous semigroup of contractions on $\mathrm{L}^2(\H,\eta)$ (see \eqref{extd}). We denote by $\mathcal{N}_2$, its infinitesimal generator. The primary problem is to relate the concrete differential operator or Kolmogorov operator, $\mathcal{N}_0$ (see \eqref{4p5} for description) with the abstract operator, $\mathcal{N}_2$. One of the main results of this article is to demonstrate that the Kolmogorov operator $\mathcal{N}_0$ has a closure in $\mathrm{L}^2(\H,\eta)$, which is $\mathcal{N}_2$. Then, we say that $\mathcal{N}_0$ is essentially $m$-dissipative, and this guarantees the existence and uniqueness of a strong solution (in the sense of Friedrichs) to Kolmogorov equations \eqref{kolme} (and \eqref{kolmp} as well) (that is, the limit of strict solution; see \cite[Chapter 7]{gdp7} for the notion of strict solution). Once we get essential $m$-dissipativity, the most crucial property of $\mathcal{N}_2$, which is the \emph{identit\'e du carr\'e du champ} (or integration by parts formula)  follows easily (see \cite{gdp1,gdp7}). 

The integration by parts formula gives the existence of $\Q^{\frac{1}{2}}\D_{\x}\varphi$ for any $\varphi\in\D(\mathcal{N}_2)$ and we derive the following estimate:
\begin{align*}
 \|\Q^{\frac{1}{2}}\D_{\x}(\lambda\I-\mathcal{N}_2)^{-1}\varphi\|_{\mathbb{L}^2(\H,\eta;\H)}\leq \sqrt{\frac{2}{\lambda}}\|\f\|_{\mathbb{L}^2(\H,\eta)},
\end{align*}
for any $\f\in\mathrm{L}^2(\H,\eta)$ (see Section \ref{disscore}).
It gives a well-defined meaning of \emph{perturbed Kolmogorov equations}, which is precisely the Kolmogorov equations with a perturbation term of the form $(\mathfrak{F}(\x),\Q^{\frac{1}{2}}\D_{\x}\varphi(\x))$ for any $\mathfrak{F}\in\mathscr{B}_b(\H)$ and $\x\in\H$. 

These perturbed Kolmogorov equations naturally arise in the infinite (or finite) horizon stochastic optimal control problems where the value function satisfies  \emph{nonlinear staionary (or non-stationary) HJB equations} with a Hamiltonian of the form $g(\|\Q^{\frac{1}{2}}\D_{\x}\varphi(\cdot)\|_{\H})$ for some $g\in\mathscr{B}_b(\H)$ (see \eqref{55} in Section \ref{sec5}).
 Note that the linear part of the nonlinear HJB equation is the Kolmogorov equation associated with the SCBF equations \eqref{51} whose solution is given by the transition semigroup $\P_t$. While the nonlinear part consists of $g(\|\Q^{\frac{1}{2}}\D_{\x}\varphi(\cdot)\|_{\H})$ (see \eqref{55}).
We provide brief details  for the existence of optimal controls by taking a global Lipschitz Hamiltonian (see Subsection \ref{existopt}). Moreover, it turns out that the optimal feedback control, say $\mathrm{U}^*(\cdot)$, can be explicitly given in terms of $\Q^{\frac{1}{2}}\D_{\x}\varphi(\cdot)$ (cf. \cite{gozzi, MoSS2}).  

 Another essential part of this work is analyzing the optimal stopping problem associated with an infinite-dimensional Kolmogorov equation using variational techniques. Unlike the optimal control problem, the value function of the optimal stopping problem satisfies some variational inequality, which turns out to be an obstacle problem (see \eqref{4.3}). Following the work of \cite{VBSSO}, we study the obstacle problem associated with the infinite-dimensional Kolmogorov operator of optimal stopping problem associated with the SCBF equations \eqref{51}. Infinite-dimensional obstacle problems are less explored in the literature. Formally, we show (see Subsection \ref{derivation}) how variational inequality (parabolic) can be obtained by using the value function of the optimal stopping problem. Moreover, the variational inequality of the optimal stopping problem can be viewed as a nonlinear Cauchy problem perturbed with the maximal monotone operator, $N_{K}$, the normal cone onto a closed convex subset $K$. Since the Kolmogorov operator is $m$-dissipative, we can obtain the solvability of our obstacle (or parabolic variational inequality) by applying the standard results for nonlinear Cauchy problems of accretive type (see \cite{VB1,VB2}), and we solve our optimal stopping problem associated with the SCBF equations \eqref{51}.

\subsection{Organization of the paper}
The remaining sections are arranged as follows: In the upcoming section, we provide the necessary function spaces needed for the analysis of Kolmogorov equations associated with SCBF equations and we give definitions of linear, bilinear, and nonlinear operators. Next, we discuss stochastic setup of CBF system in Section \ref{sec3}, where we also provide the relevant details about the noise, trace-class operator, abstract formulation (see \eqref{32}) and the probabilistic notion of a strong solution (see the Definition \ref{def3.1}). Section \ref{sec4} is devoted to the existence and uniqueness of invariant measure for the transition semigroup $\mathrm{P}_t$ associated with the system \eqref{32}. The existence of an invariant measure $\eta$ follows by using the classical argument of Krylov-Bogoliubov theorem (see Subsection \ref{sec4.2}). Moreover, we define the transition semigroup in $\mathrm{L}^2(\H;\eta)$ and introduce the Kolmogorov differential operator $\mathcal{N}_0$ (see \eqref{4p5}). We then establish various exponential moments of the invariant measure (see Lemma \ref{lem4.1}) and provide estimates for bounds on the derivative of the solution to the SCBF equations \eqref{32} (see Lemma \ref{lem4.4}), under the assumption  that the viscosity $\mu$ and Darcy coefficient $\alpha$ are sufficiently large (see \eqref{419}). These results ensure the uniqueness of invariant measure $\eta$ (see Theorem \ref{uniqinv}). In section \ref{sec5}, we examine the approximated system \eqref{4.33} of the SCBF equations \eqref{32} and derive some auxiliary estimates concerning the bound of interpolation estimate  $\int_{\H}\|\A^{\delta}\x\|_{\H}^{2m-2}\|\A^{\delta+\frac{1}{2}}\x\|_{\H}^2\eta(\d\x),$ for any $0<\delta<\frac{1}{2}$ and for any $m\in\N$ (see Lemma \ref{lem4.5}). In Section \ref{disscore}, we present one of the central contributions of this article, the essential $m$-dissipativity of the Kolmogorov operator $\mathcal{N}_0$ (see Theorem \ref{thm4.7}). Additionally, we outline some of its significant implications, such as the integration by parts formula and the perturbation of the Kolmogorov operator (see Lemma \ref{carredu}-\ref{pertubkol}). In the last two sections, we discuss  applications of Kolmogorov equations in control theory, namely infinite horizon problems and optimal stopping time problems. In Section \ref{sec5}, by using a fixed point argument, we solve the stationary HJB equations \eqref{55} of the optimal control problem associated with \eqref{51}. Whereas in Section \ref{sec6}, we study the optimal stopping problem associated with the SCBF equations \eqref{4p2}, and we deduce that it is equivalent to solving the variational inequality (free boundary value problem) (see \eqref{4.3}) satisfied by the value function of the optimal stopping problem (see Theorem \ref{main-thm}). 

\section{Mathematical Formulation} \label{sec2}\setcounter{equation}{0}
This section provides the necessary function spaces and operators needed to obtain the global solvability results of the system \eqref{1}. 

\subsection{Function spaces} Let $\C_\mathrm{cpt}^{\infty}(\mathcal{O};\R^2)$ denotes the space of all infinitely differentiable functions  ($\R^2$-valued) with compact support in $\mathcal{O}\subset\R^2$. We define 
\begin{align*} 
\mathcal{V}&:=\{\x\in\C_\mathrm{cpt}^{\infty}(\mathcal{O},\R^2):\nabla\cdot\x=0\},\\
\mathbb{H}&:=\text{the closure of }\ \mathcal{V} \ \text{ in the Lebesgue space } \L^2(\mathcal{O})=\mathrm{L}^2(\mathcal{O};\R^2),\\
\mathbb{V}&:=\text{the closure of }\ \mathcal{V} \ \text{ in the Sobolev space } \H_0^1(\mathcal{O})=\mathrm{H}_0^1(\mathcal{O};\R^2),\\
\widetilde{\L}^{p}&:=\text{the closure of }\ \mathcal{V} \ \text{ in the Lebesgue space } \L^p(\mathcal{O})=\mathrm{L}^p(\mathcal{O};\R^2),
\end{align*}
for $p\in(2,\infty)$. Then under some smoothness assumptions on the boundary, we characterize the spaces $\H$, $\V$ and $\widetilde{\L}^p$ as 
$
\H=\{\x\in\L^2(\mathcal{O}):\nabla\cdot\x=0,\x\cdot\boldsymbol{n}\big|_{\partial\mathcal{O}}=0\}$,  with norm  $\|\x\|_{\H}^2:=\int_{\mathcal{O}}|\x(\boldsymbol{\xi})|^2\d\boldsymbol{\xi},
$
where $\boldsymbol{n}$ is the outward normal to $\partial\mathcal{O}$,
$
\V=\{\x\in\H_0^1(\mathcal{O}):\nabla\cdot\x=0\},$  with norm $ \|\x\|_{\V}^2:=\int_{\mathcal{O}}|\nabla\x(\boldsymbol{\xi})|^2\d\boldsymbol{\xi},
$ and $\widetilde{\L}^p=\{\x\in\L^p(\mathcal{O}):\nabla\cdot\x=0, \x\cdot\boldsymbol{n}\big|_{\partial\mathcal{O}}\},$ with norm $\|\x\|_{\widetilde{\L}^p}^p=\int_{\mathcal{O}}|\x(\boldsymbol{\xi})|^p\d\boldsymbol{\xi}$, respectively.
Let $(\cdot,\cdot)$ denotes the inner product in the Hilbert space $\H$ and $\langle \cdot,\cdot\rangle $ denotes the induced duality between the spaces $\V$  and its dual $\V'$ as well as $\widetilde{\L}^p$ and its dual $\widetilde{\L}^{p'}$, where $\frac{1}{p}+\frac{1}{p'}=1$. Note that $\H$ can be identified with its dual $\H'$. 

\subsubsection{$\mathrm{L}^2$-Hilbert spaces} 

Let $\eta$ be an invariant measure.
We denote by $\mathrm{L}^2(\H,\eta)$, the equivalence class of all Borel square integrable functions $\varphi:\H\to\R$, endowed with the inner product
\begin{align*}
 (\varphi,\psi)_{\mathrm{L}^2(\H,\eta)}:=\int_{\H} \varphi(\x)\psi(\x)\eta(\d\x), \  \varphi,\psi\in\mathrm{L}^2(\H,\eta),
\end{align*}
and norm 
\begin{align*}
	\|\varphi\|_{\mathrm{L}^2(\H,\eta)}:=\left(\int_{\H} |\varphi(\x)|^2\eta(\d\x)\right)^{\frac{1}{2}}, \ \ \varphi\in\mathrm{L}^2(\H,\eta).
\end{align*}
Let us now consider the space $\L^2(\H,\eta;\H)$, of all equivalence classes of Borel square integrable functions $F:\H\to\H$, such that
\begin{align*}
	\|F\|_{\L^2(\H,\eta;\H)}:=\left(\int_{\H}\|F(\x)\|_{\H}^2\eta(\d\x)\right)^{\frac{1}{2}} <\infty.
\end{align*}
The space $\L^2(\H,\eta;\H)$, endowed with the inner product 
\begin{align*}
	(F,G)_{\L^2(\H,\eta;\H)}:=\int_{\H} (F(\x),G(\x))\eta(\d\x), \ \ \ F,G\in\L^2(\H,\eta;\H),
\end{align*} 
 is a Hilbert space. The elements of $\L^2(\H,\eta;\H)$ are called as $\mathrm{L}^2$-\emph{vector fields} (cf. \cite{gdp7}). For any $F\in\L^2(\H,\eta;\H)$,  we can write $F(\x)=\sum\limits_{k=1}^{\infty}(F(\x),\boldsymbol{e}_k)\boldsymbol{e}_k$, $\eta$-a.e., where $\{\boldsymbol{e}_k\}_{k=1}^{\infty}$ is a complete orthonormal basis in $\H$.

\subsection{Linear operator}
Let $\mathrm{P}_{\H} : \L^2(\mathcal{O}) \to\H$ denotes the \emph{Helmholtz-Hodge orthogonal projection} (see \cite{OAL,AC}). We define
\begin{equation*}
\left\{
\begin{aligned}
\A\u:&=-\mathrm{P}_{\H}\Delta\u,\;\u\in\D(\A),\\ \D(\A):&=\V\cap\H^{2}(\mathcal{O}).
\end{aligned}
\right.
\end{equation*}
It can be easily seen that the operator $\A$ is a non-negative self-adjoint operator in $\H$ with $\V=\D(\A^{1/2})$ and \begin{align}\label{2.7a}\langle \A\u,\u\rangle =\|\u\|_{\V}^2,\ \textrm{ for all }\ \u\in\V, \ \text{ so that }\ \|\A\u\|_{\V'}\leq \|\u\|_{\V}.\end{align}
For a bounded domain $\mathcal{O}$, the operator $\A$ is invertible and its inverse $\A^{-1}$ is bounded, self-adjoint and compact in $\H$. Thus, using spectral theorem, the spectrum of $\A$ consists of an infinite sequence $0< \uplambda_1\leq \uplambda_2\leq\ldots\leq \uplambda_k\leq \ldots,$ with $\uplambda_k\to\infty$ as $k\to\infty$ of eigenvalues. 
Moreover, there exists an orthogonal basis $\{\boldsymbol{e}_k\}_{k=1}^{\infty} $ of $\H$ consisting of eigenvectors of $\A$ such that $\A \boldsymbol{e}_k =\uplambda_k\boldsymbol{e}_k$,  for all $ k\in\mathbb{N}$.  We know that $\u$ can be expressed as $\u=\sum_{k=1}^{\infty}\langle\u,\boldsymbol{e}_k\rangle \boldsymbol{e}_k$ and $\A\u=\sum_{k=1}^{\infty}\uplambda_k\langle\u,\boldsymbol{e}_k\rangle \boldsymbol{e}_k$. Thus, it is immediate that 
\begin{align}\label{poin}
\|\nabla\u\|_{\mathbb{H}}^2=\langle \A\u,\u\rangle =\sum_{k=1}^{\infty}\uplambda_k|\langle \u,\boldsymbol{e}_k\rangle|^2\geq \uplambda_1\sum_{k=1}^{\infty}|\langle\u,\boldsymbol{e}_k\rangle|^2=\uplambda_1\|\u\|_{\mathbb{H}}^2,
\end{align}
which is the Poincar\'e inequality.  In this work, we also need the fractional powers of $\A$.  For $\u\in \H$ and  $\alpha>0,$ we define
$\A^\alpha \u=\sum_{k=1}^\infty \uplambda_k^\alpha (\u,\boldsymbol{e}_k) \boldsymbol{e}_k,  \ \u\in\D(\A^\alpha), $ where $\D(\A^\alpha)=\left\{\u\in \H:\sum_{k=1}^\infty \uplambda_k^{2\alpha}|(\u,\boldsymbol{e}_k)|^2<+\infty\right\}.$ 
Here  $\D(\A^\alpha)$ is equipped with the norm 
\begin{equation} \label{fn}
\|\A^\alpha \u\|_{\H}=\left(\sum_{k=1}^\infty \uplambda_k^{2\alpha}|(\u,\boldsymbol{e}_k)|^2\right)^{1/2}.
\end{equation}
It can be easily seen that $\D(\A^0)=\H,$  $\D(\A^{1/2})=\V$. We set $\V_\alpha= \D(\A^{\alpha/2})$ with $\|\u\|_{\V_{\alpha}} =\|\A^{\alpha/2} \u\|_{\H}.$   Using Rellich-Kondrachov compactness embedding theorem, we know that for any $0\leq s_1<s_2,$ the embedding $\D(\A^{s_2})\subset \D(\A^{s_1})$ is also compact. 

From the characterization \eqref{fn}, it can be easily seen that 
\begin{align*}
	\|\A^{\delta+\frac{1}{2}}\x\|_{\H}^2=\sum_{k=1}^\infty \uplambda_k^{2\left(\delta+\frac{1}{2}\right)}|(\x,\boldsymbol{e}_k)|^2&=\sum_{k=1}^\infty \uplambda_k^{2\delta}\uplambda_k|(\x,\boldsymbol{e}_k)|^2=\sum_{k=1}^\infty \uplambda_k^{2\delta}|(\A^{\frac{1}{2}}\x,\boldsymbol{e}_k)|^2\nonumber\\&\geq\uplambda_1^{2\delta}
	\sum_{k=1}^\infty |(\A^{\frac{1}{2}}\x,\boldsymbol{e}_k)|^2= \uplambda_1^{2\delta}\|\A^{\frac{1}{2}}\x\|_{\H}^2,
\end{align*}
which gives 
\begin{align}\label{fracA}
	\|\A^{\delta+\frac{1}{2}}\x\|_{\H}\geq\uplambda_1^{\delta}\|\x\|_{\V} \ \text{ for all } \ \x\in\D(\A^{\delta+\frac{1}{2}}). 
\end{align}
Similarly, we can calculate
\begin{align}
\|\A^{\delta}\x\|_{\H}&\geq\uplambda_1^{\delta}\|\x\|_{\H} \ \text{ for all } \ \x\in\D(\A^{\delta}), \label{fracA1}
\\
\|\A^{\delta+\frac{1}{2}}\x\|_{\H}&\geq\sqrt{\uplambda_1}\|\A^{\delta}\x\|_{\H} \ \text{ for all } \ \x\in\D(\A^{\delta+\frac{1}{2}}),\label{fracA11}\\
\|\A^{\frac{1}{2}}\x\|_{\H}&\geq {\uplambda_1^{\delta-\frac{1}{2}}}  \|\A^{\delta}\x\|_{\H}
\ \text{ for all } \ \x\in\D(\A^{\frac{1}{2}}),\label{fracA2}
\end{align}
provided $0<\delta<\frac{1}{2}$. 

\subsection{Bilinear operator}
Let us define the \emph{trilinear form} $b(\cdot,\cdot,\cdot):\V\times\V\times\V\to\R$ by $$b(\u,\v,\w)=\int_{\mathcal{O}}(\u(x)\cdot\nabla)\v(x)\cdot\w(x)\d x=\sum_{i,j=1}^n\int_{\mathcal{O}}\u_i(x)\frac{\partial \v_j(x)}{\partial x_i}\w_j(x)\d x.$$ If $\u, \v$ are such that the linear map $b(\u, \v, \cdot) $ is continuous on $\V$, the corresponding element of $\V'$ is denoted by $\B(\u, \v)$. We also denote (with an abuse of notation) $\B(\u) = \B(\u, \u)=\mathrm{P}_{\H}(\u\cdot\nabla)\u$.
An integration by parts gives 
\begin{equation}\label{b0}
\left\{
\begin{aligned}
b(\u,\v,\v) &= 0,\text{ for all }\u,\v \in\V,\\
b(\u,\v,\w) &=  -b(\u,\w,\v),\text{ for all }\u,\v,\w\in \V.
\end{aligned}
\right.
\end{equation}

\subsection{Nonlinear operator}
Let us now consider the operator $\mathcal{C}(\u):=\P_{\H}(|\u|^{r-1}\u)$. It is immediate that $\langle\mathcal{C}(\u),\u\rangle =\|\u\|_{\widetilde{\L}^{r+1}}^{r+1}$ and the map $\mathcal{C}(\cdot):\widetilde{\L}^{r+1}\to\widetilde{\L}^{\frac{r+1}{r}}$ is Gateaux differentiable with Gateaux derivative given by (\cite{MT2})
\begin{align}\label{2133}
	\mathcal{C}'(\y)\z&=\left\{\begin{array}{cl}\mathcal{P}(\z),&\text{ for }r=1,\\ \left\{\begin{array}{cc}\mathcal{P}(|\y|^{r-1}\z)+(r-1)\mathcal{P}\left(\frac{\y}{|\y|^{3-r}}(\y\cdot\z)\right),&\text{ if }\y\neq \mathbf{0},\\\mathbf{0},&\text{ if }\y=\mathbf{0},\end{array}\right.&\text{ for } 1<r<3,\\ \mathcal{P}(|\y|^{r-1}\z)+(r-1)\mathcal{P}(\y|\y|^{r-3}(\y\cdot\z)), &\text{ for }r\geq 3,\end{array}\right.
\end{align}
for all $\z\in\widetilde{\L}^{r+1}$. 

\section{Stochastic convective Brinkman-Forchheimer equations}\label{sec3} \setcounter{equation}{0}

Let $(\Omega,\mathscr{F},\mathbb{P})$ be a complete probability space equipped with an increasing family of sub-sigma fields $\{\mathscr{F}_t\}_{t\geq0}$ of $\mathscr{F}$ satisfying the usual conditions (that is, a normal filtration). The noise term on the  stochastic basis $(\Omega,\mathscr{F},\{\mathscr{F}_t\}_{t\geq0},\mathbb{P})$ is described by  a cylindrical Wiener process $\{\W(t)\}_{t\geq 0}$ on $\H$,  and a covariance operator $\Q$. Let $\mathcal{L}(\H,\H)$ be the space of all bounded linear operators on $\H$. Let the covariance operator  $\mathrm{Q}\in\mathcal{L}(\H,\H)$ be  such that $\mathrm{Q}$ is positive, symmetric and trace class operator with ker $\mathrm{Q}=\{\boldsymbol{0}\}$. We assume that there exists a complete orthonormal system $\{\boldsymbol{e}_k\}_{k\in\mathbb{N}}$ in $\mathbb{H}$ of the covariance operator $\Q$ and a bounded sequence $\{\mu_k\}_{k\in\mathbb{N}}$ of positive real numbers such that $\Q \boldsymbol{e}_k=\mu_k \boldsymbol{e}_k,\ k\in\N$. Here $\mu_k$ is an eigenvalue corresponding to the eigenfunction $\boldsymbol{e}_k$ such that following holds:
\begin{align*}
	\Tr\Q=\sum\limits_{k=1}^{\infty} \mu_k<\infty \ \text{ and }\  \sqrt{\Q}\y=\sum\limits_{k=1}^{\infty}\sqrt{\mu_k}(\y,\boldsymbol{e}_k)\boldsymbol{e}_k,  \ \text{ for }\  \y\in\H.
\end{align*}
We know that the stochastic process $\{\W(t)\}_{t\geq0}$ is an $\H$-valued cylindrical Wiener process with respect to a filtration $\{\mathscr{F}_t\}_{t\geq0}$ if and only if for any $t\geq0,$ the process $\{\W(t)\}_{t\geq 0}$ can be expressed as $\W(t)=\sum\limits_{k=1}^{\infty} \beta_k(t)\boldsymbol{e}_k$, where $\{\beta_k(\cdot), \ k\in\N\}$ is a family of real-valued independent Brownian motions on the probability space $\left(\Omega,\mathscr{F},\{\mathscr{F}_t\}_{t\geq 0},\P\right)$ (cf. \cite{gdp}).

\subsection{Abstract formulation of the stochastic system}\label{sec2.4}
Let us set $\X(t,\x):=\mathrm{P}_{\H}\Y(t,\y)$, $\mathrm{P}_{\H}\y:=\x$ and $\W(t):=\mathrm{P}_{\H}\W(t)$. On projecting the first equation in \eqref{1}, we obtain 
\begin{equation}\label{32}
\left\{
\begin{aligned}
\d\X(t)+[\mu \A\X(t)+\alpha\X(t)+\B(\X(t))+\beta\mathcal{C}(\X(t))]\d t&=\sqrt{\Q}\d\W(t), \ t\in(0,T),\\
\X(0)&=\x,
\end{aligned}
\right.
\end{equation}
where $\x\in\H$, the linear operator $\Q\in\mathcal{L}(\H)$ is non-negative, symmetric and trace class and $\W$ is an $\H$-valued Wiener process defined on a stochastic basis $(\Omega,\mathscr{F},\mathbb{P},\{\mathscr{F}_t\}_{t\geq0})$. 

Let us now provide the definition of a unique global strong solution in the probabilistic sense to the system (\ref{32}).
\begin{definition}[Global strong solution]\label{def3.1}
	Let $\x\in\H$ be given. An $\H$-valued $(\mathscr{F}_t)_{t\geq 0}$-adapted stochastic process $\X(\cdot)$ is called a \emph{strong solution} to the system (\ref{32}) if the following conditions are satisfied: 
	\begin{enumerate}
		\item [(i)] the process $\X\in\mathrm{L}^4(\Omega;\mathrm{L}^{\infty}(0,T;\H)\cap\mathrm{L}^2(0,T;\V))$ and $\X(\cdot)$ has a $\V$-valued  modification, which is progressively measurable with continuous paths in $\H$ and $\X\in\C([0,T];\H)\cap\mathrm{L}^2(0,T;\V)$, $\mathbb{P}$-a.s.,
		\item [(ii)] the following equality holds for every $t\in [0, T ]$, as an element of $\V',$ $\mathbb{P}$-a.s.
		\begin{align}\label{4.4}
		\X(t)&=\X_0-\int_0^t\left[\mu \A\X(s)+\alpha\X(s)+\B(\X(s))+\beta\mathcal{C}(\X(s))\right]\d s+\int_0^t\sqrt{\Q}\d\W(s),
		\end{align}
		\item [(iii)] the following It\^o formula holds true: 	for all $t\in[0,T]$, $\mathbb{P}$-a.s.
	\begin{align}\label{a34}
	&	\|\X(t)\|_{\H}^2+2\mu \int_0^t\|\X(s)\|_{\V}^2\d s+2\alpha\int_0^t\|\X(s)\|_{\H}^2\d s+2\beta\int_0^t\|\X(s)\|_{\widetilde{\L}^{r+1}}^{r+1}\d s\nonumber\\&=\|\x\|_{\H}^2+t\Tr(\Q)+2\int_0^t(\sqrt{\Q}\d\W(s),\X(s)).
	\end{align}

	\end{enumerate}
\end{definition}
An alternative version of condition (\ref{4.4}) is to require that for any  $\v\in\V$:
\begin{align}\label{4.5}
(\X(t),\v)&=(\X_0,\v)-\int_0^t\langle\mu \A\X(s)+\alpha\X(s)+\B(\X(s))+\beta\mathcal{C}(\X(s)),\v\rangle\d s\no\\&\quad+\int_0^t\left(\sqrt{\Q}\d \W(s),\v\right),\ \mathbb{P}\text{-a.s.}
\end{align}	
\begin{definition}
	A strong solution $\X(\cdot)$ to (\ref{32}) is called a
	\emph{pathwise  unique strong solution} if
	$\widetilde{\X}(\cdot)$ is an another strong
	solution, then $$\mathbb{P}\big\{\omega\in\Omega:\X(t)=\widetilde{\X}(t),\ \text{ for all }\ t\in[0,T]\big\}=1.$$ 
\end{definition}

\begin{theorem}[Theorem 3.7, \cite{MTM8}]\label{exis2}
	Let $\x\in\H$ be given and $\mathrm{Tr}(\Q)<\infty$. Then, for $r\in\{1,2,3\}$, there exists a \emph{pathwise unique strong solution}
	$\X(\cdot)$ to the system (\ref{32}) in the sense of Definition \ref{def3.1} such that 
	\begin{align}
		\E\bigg[\sup_{t\in[0,T]}\|\X(t)\|_{\H}^4+\int_0^T\|\X(s)\|_{\H}^2\|\X(s)\|_{\V}^2\d s\bigg]\leq C(\|\x\|_{\H},\mu,\mathrm{Tr}(\Q),T).
	\end{align}
\end{theorem}
\begin{remark}
	Even though one obtains $\X\in\mathrm{L}^{r+1}(\Omega;\mathrm{L}^{r+1}(0,T;\widetilde{\L}^{r+1}))$ from the It\^o formula \eqref{a34}, for $d=2$ and $r\in\{1,2,3\}$, it can be shown that $\mathrm{L}^4(\Omega;\mathrm{L}^{\infty}(0,T;\H)\cap\mathrm{L}^2(0,T;\V))\subset \mathrm{L}^{r+1}(\Omega;\mathrm{L}^{r+1}(0,T;\widetilde{\L}^{r+1}))$ by using Gagliardo-Nirenberg's and H\"older's inequalities, that is,
	\begin{align*}
		\E\bigg[\int_0^T\|\X(t)\|_{\wi\L^{r+1}}^{r+1}\d t\bigg]\leq CT^{\frac{3-r}{2}}\bigg\{\E\bigg[\sup_{t\in[0,T]}\|\X(t)\|_{\H}^4\bigg]\bigg\}^{1/2}\bigg\{\E\bigg[\bigg(\int_0^T\|\X(t)\|_{\V}^2\d t\bigg)^{r-1}\bigg]\bigg\}^{1/2}.
	\end{align*} 
\end{remark}

	\section{Existence and uniqueness of invariant measures}\label{sec4} \setcounter{equation}{0}
	We know that the system  \eqref{32}
	 has a unique strong solution (cf. \cite{akmtm,MTM8}) $$\X\in\mathrm{L}^4(\Omega;\mathrm{L}^{\infty}(0,T;\H)\cap\mathrm{L}^2(0,T;\V))$$ having a  modification with paths in $\C([0,T];\H)\cap\mathrm{L}^2(0,T;\V)$, $\mathbb{P}$-a.s. Let us now introduce the \emph{transition semigroup} associated to the  problem \eqref{32}. 
	 For any $\psi\in\mathscr{B}_b(\H)$ and $t\geq0$, we define the family of operators $\{\P_t\}_{t\geq0}$ as 
	 \begin{align}\label{tight1}
	 	(\P_t\psi)(\x)=\E\left[\psi(\X(t,\x))\right], \ \x\in\H,
	 \end{align} 
 where $\X=\X(t,\x)$ is the unique strong solution of the SCBF system \eqref{32}. It is well-known from \cite[Corollary 23]{MO} that the transition function is jointly measurable, that is, for any Borel subset $\Gamma$ of $\H$, the map $\H\times[0,\infty)\ni(\x,t)\mapsto\mathbb{P}\{\X(t,\x)\in\Gamma\}\in\R$ is measurable. Thus, $\P_t\psi$ is also measurable for any $\psi\in\mathscr{B}_b(\H)$ and hence $\P_t$ maps $\mathscr{B}_b(\H)$ into itself for every $t\geq0$. Since the unique solution of \eqref{32} is an $\H$-valued continuous process, therefore from \cite[Theorem 27]{MO}, it is also a Markov process and hence we deduce that the family of operators $\{\P_t\}_{t\geq0}$ is a Markov semigroup, namely $\P_{t+s}=\P_t \P_s$ for any $t,s\geq0$. 

   Furthermore, for every $\psi\in\C_b(\H)$, the function $\H\ni\x\mapsto\E\left[\psi(\X(t,\x))\right]\in\R$ is continuous and hence the semigroup $\{\P_t\}_{t\geq0}$ is Feller (see \cite[Chapter 4]{EANK}). But in general, the restriction of $\{\P_t\}_{t\geq0}$ to $\C_b(\H)$ is not strongly continuous semigroup (see \cite{SC2} for counter examples). In fact one can prove that, under certain hypothesis, $\{\P_t\}_{t\geq0}$ defines a \emph{weakly continuous semigroup} on $\C_b(\H)$ (see \cite[Appendix B, pp. 305]{SC1} for the definition and properties).
Following \cite{SC2}, the infinitesimal generator of the weakly continuous semigroup $\{\P_t\}_{t\geq0}$ is the unique closed linear operator $\mathcal{N}:\D(\mathcal{N})\subset\C_b(\H)\to\C_b(\H)$ such that for any $\lambda>0,$ the following resolvent formula holds:
 \begin{align*}
 	R(\lambda,\mathcal{N})\psi(\x)=\int_0^{\infty} e^{-\lambda t} \P_t \psi(\x)\d t, 
 \end{align*}
 for all $\psi\in\C_b(\H)$ and $\x\in\H$. 
 
 \subsection{Existence of an invariant measure}\label{sec4.2}
 It is well known that the system \eqref{32} has a unique strong solution $\X(\cdot)$ with paths in $\C([0,T];\H)\cap\mathrm{L}^2(0,T;\V)\cap\mathrm{L}^{r+1}(0,T;\wi\L^{r+1})$, $\mathbb{P}$-a.s. Applying infinite-dimensional It\^o's formula to the process $\|\X(\cdot)\|_{\H}^2$, we obtain 
	\begin{align}\label{4.26}
	&	\E\left[\|\X(t)\|_{\H}^2+2\mu\int_0^t\|\X(s)\|_{\V}^2\d s+2\alpha\int_0^t\|\X(s)\|_{\H}^2\d s+2\beta\int_0^t\|\X(s)\|_{\wi\L^{r+1}}^{r+1}\d s\right]\nonumber\\&=\|\x\|_{\H}^2+\Tr(\Q)t,
	\end{align}
	for all $t\in[0,T]$. Consequently, we have 
	\begin{align*}
			\frac{2\mu}{t}\E\left[\int_0^t\|\X(s)\|_{\V}^2\d s\right]\leq\frac{1}{t_0}\|\x\|_{\H}^2+\Tr(\Q), \text{ for all } t>t_0.
	\end{align*}
   By applying Markov's inequality, we obtain
   \begin{align}\label{tight}
   	\lim\limits_{r\to\infty}\sup\limits_{t>t_0}\frac{1}{t}\int_0^t \P\{\|\X(s)\|_{\V}>r\}\d s\leq\lim \limits_{r\to\infty}\sup\limits_{t>t_0}\frac{1}{r^2}\E\left[\frac{1}{t}\int_0^t \|\X(s)\|_{\V}^2\d s\right]=0.
   \end{align}
Let us set
\begin{align*}
	\zeta_{t,\x_0}(\cdot)=\frac{1}{t}\int_0^t \lambda_{s,\x_0}(\cdot)\d s,
\end{align*}
where, $\lambda_{t,\x_0}(\Lambda)=\P\{\X(t,\x_0)\in\Lambda\}, \ \Lambda\in\mathscr{B}(\H)$, is the law of $\X(t,\x_0)$ for each $\x_0\in\H$. Hence, along with the estimate in \eqref{tight}, it is clear that the sequence of probability measures $\{\zeta_{t,\x_0}\}_{t>0}$ is tight.
	Then, it follows by the Krylov-Bogoliubov theorem (cf. \cite[Theorem 3.1.1, Chapter 3]{gdp2}) that there is an invariant measure $\eta$ for the transition semigroup $(\P_t)_{t\geq0}$, that is, 
	\begin{align*}
	\int_{\H}(\P_t\psi)(\x)\eta(\d \x)=\int_{\H}\psi(\x)\eta(\d \x), \ \text{ for all }\ \psi\in\C_b(\H). 
	\end{align*}

  \subsection{Transition semigroup in $\L^2(\H,\eta)$} Let $\psi\in\C_b(\H)$. By using H\"older's inequality and thanks to the invariance of $\eta$, we obtain following:
\begin{align}\label{extd}
	\int_{\H} (\P_t\psi(\x))^2\eta(\d\x)\leq\int_{\H} \psi^2(\x)\eta(\d\x).
\end{align}
Since $\C_b(\H)$ is dense in $\L^2(\H,\eta)$ (\cite{gdp1}), therefore by using the Lebesgue dominated convergence theorem, $(\P_t)_{t\geq0}$ can be uniquely extendable to a strongly continuous semigroup of contractions (still denoted by $\P_t$) in $\L^2(\H,\eta)$(see \cite{gdp1,gdp7}). We shall denote by $\mathcal{N}_2: \D(\mathcal{N}_2)\subset\L^2(\H;\eta)\to\L^2(\H;\eta)$, the infinitesimal generator of $\mathrm{P}_t$. 

Let us introduce an algebra of \emph{exponential functions}: 
\begin{align*}
	\mathscr{E}_{\A}(\H):=\mathrm{linspan}\{\psi_h: h\in\D(\A)\},
\end{align*}
where $\psi_h:\H\ni\x\mapsto\psi_h(\x)=e^{i(h,\x)}$, where $i=\sqrt{-1}$. From \cite[Proposition 1.2, Chapter 1, pp. 7]{gdp1} and dominated convergence theorem, it is easy to see that $\mathscr{E}_{\A}(\H)$ is dense in $\L^2(\H;\eta)$. Let us define on $\mathscr{E}_{\A}(\H)$, the following Kolmogorov differential operator:
\begin{align}\label{4p5}
	(\mathcal{N}_0\psi)(\x)=\frac{1}{2}\Tr\left[\Q\D_{\x}^2\psi(\x)\right]-(\mu\A\x+\alpha\x+\B(\x)+\beta\mathcal{C}(\x),\D_{\x}\psi(\x)), \ \text{ for all }\ \psi\in\mathscr{E}_{\A}(\H),
\end{align}
where $\Tr$ represents the trace and $\D_{\x}^k$, $k=1,2$ denotes the $k^{\mathrm{th}}$  Fr\'echet derivative with respect to the initial data $\x$. One of the interesting problem  of this work is to understand the relationship between the abstract operator $\mathcal{N}_2$ and the Kolmogorov differential operator $\mathcal{N}_0$.

\subsection{Uniqueness of invariant measures} Under some assumptions on $\mu,\alpha$ and $\Q$, let us  show that the invariant measure obtained in Section \ref{sec4.2} is unique. Let us first obtain some exponential moments. 
	\begin{lemma}[Exponential moments]\label{lem4.1}
Assume that
 \begin{align}\label{improve}
	0\leq\upsigma\leq\frac{2\alpha+\mu\uplambda_1}{2\Tr(\Q)}.
\end{align}
 Then, we have 
		\begin{align}\label{4p6}
			\int_{\H}e^{\upsigma\|\x\|_{\H}^2}\eta(\d\x)\leq 2 e^{\frac{2\upsigma \Tr(\Q)}{\mu\uplambda_1}}.
		\end{align}
		Moreover, we get 
		\begin{align}\label{4p7}
			\int_{\H}\|\x\|_{\V}^2 e^{\upsigma \|\x\|_{\H}^2}\eta(\d\x)\leq \frac{2\Tr(\Q)}{\mu}e^{\frac{2\upsigma \Tr(\Q)}{\mu\uplambda_1} }\ \text{ and }\ \int_{\H}\|\x\|_{\wi\L^{r+1}}^{r+1} e^{\upsigma \|\x\|_{\H}^2}\eta(\d\x)\leq\frac{\Tr(\Q)}{\beta}e^{\frac{2\upsigma \Tr(\Q)}{\mu\uplambda_1}}. 
		\end{align}
	\end{lemma}

\begin{proof}
	Let us consider the Kolmogorov operator $\mathcal{N}_0$ defined in \eqref{4p5} and the function $\psi(\x)=e^{\upsigma\|\x\|_{\H}^2}$. Therefore, we have 
	\begin{align*}
		\D_{\x}\psi(\x)=2\upsigma  e^{\upsigma \|\x\|_{\H}^2}\x , \ \D_{\x}^2\psi(\x)=\upsigma  e^{\upsigma \|\x\|_{\H}^2}(4\upsigma (\x\otimes\x)+2), \ \x\in\H. 
	\end{align*}
	Approximating  $\psi$ by a family of smooth functions, we may assume that $\psi$ belongs to the domain $\D(\mathcal{N}_2)$ of the infinitesimal generator $\mathcal{N}_2$ of the semigroup $\P_t$ in $\L^2(\H,\eta)$, we obtain 
	\begin{align}\label{4.6}
	(\mathcal{N}_2\psi)(\x)=e^{\upsigma \|\x\|_{\H}^2}\left[\upsigma \Tr(\Q)+2\upsigma ^2\|\Q^{1/2}\x\|_{\H}^2-2\mu\upsigma \|\x\|_{\V}^2-2\alpha\upsigma \|\x\|_{\H}^2-2\beta\upsigma \|\x\|_{\wi\L^{r+1}}^{r+1}\right].
	\end{align}
	Integrating the above equality over $\H$ with respect to $\eta$ and using the fact that $\eta$ is invariant (\cite[Definition 5.1, Remark 5.2]{mr13}, see Step 4 of Theorem \ref{thm4.7} below), that is, 
	$$\int_{\H}{\mathcal{N}_2}\psi (\x) \d\eta(\x)=0,$$
	we get 
	\begin{align}\label{4.7}
	\int_{\H}e^{\upsigma \|\x\|_{\H}^2}\left(\mu\|\x\|_{\V}^2+\alpha\|\x\|_{\H}^2+\beta\|\x\|_{\wi\L^{r+1}}^{r+1}-\upsigma \|\Q^{1/2}\x\|_{\H}^2\right)\d\eta(\x)=\frac{1}{2}\Tr(\Q)\int_{\H}e^{\upsigma \|\x\|_{\H}^2}\d\eta(\x).
	\end{align}
Note that for $\x\in\H$, we can write $\Q^{\frac{1}{2}}\x=\sum\limits_{j=1}^{\infty} \sqrt{\uplambda_j} (\x,\boldsymbol{e}_j)\boldsymbol{e}_j$. Then, we calculate 
\begin{align*}
	\|\Q^{\frac{1}{2}}\x\|_{\H}^2&=(\Q^{\frac{1}{2}}\x,\Q^{\frac{1}{2}}\x)=\bigg(\sum\limits_{j=1}^{\infty} \sqrt{\uplambda_j} (\x,\boldsymbol{e}_j)\boldsymbol{e}_j,\sum\limits_{k=1}^{\infty} \sqrt{\uplambda_k} (\x,\boldsymbol{e}_k)\boldsymbol{e}_k\bigg)\nonumber\\&=
	\sum\limits_{j=1}^{\infty}\sum\limits_{k=1}^{\infty} \sqrt{\uplambda_j \uplambda_k} (\x,\boldsymbol{e}_j)(\x,\boldsymbol{e}_k)(\boldsymbol{e}_j,\boldsymbol{e}_k)=\sum\limits_{j=1}^{\infty} \uplambda_j (\x,\boldsymbol{e}_j)^2\leq\sum\limits_{j=1}^{\infty} \uplambda_j\|\x\|_{\H}^2= \Tr(\Q)\|\x\|_{\H}^2,
\end{align*} 
where we have used the definition of the trace class operator $\Q$, that is, $\Tr(\Q)=\sum\limits_{j=1}^{\infty} \uplambda_j$. This yields that 
\begin{align}\label{trop}
	\|\Q^{\frac{1}{2}}\|_{\mathcal{L}(\H)}^2\leq\Tr(\Q).
\end{align}
Then, by using the definition of operator norm and \eqref{trop}, we calculate
\begin{align}\label{trop1}
\|\Q^{\frac{1}{2}}\x\|_{\H}^2\leq\|\Q^{\frac{1}{2}}\|_{\mathcal{L}(\H)}^2\|\x\|_{\H}^2\leq\Tr(\Q)\|\x\|_{\H}^2.
\end{align}
Thus, for $\upsigma\leq\frac{2\alpha+\mu\uplambda_1}{2\Tr(\Q)}$, it is immediate that  
	\begin{align}\label{4.8}
&\mu\|\x\|_{\V}^2+\alpha\|\x\|_{\H}^2+\beta\|\x\|_{\wi\L^{r+1}}^{r+1}
-\upsigma \|\Q^{1/2}\x\|_{\H}^2\nonumber\\&\geq\mu\|\x\|_{\V}^2+\alpha\|\x\|_{\H}^2+\beta\|\x\|_{\wi\L^{r+1}}^{r+1}-\upsigma \Tr(\Q)\|\x\|_{\H}^2\nonumber\\&\geq\mu\|\x\|_{\V}^2+\alpha\|\x\|_{\H}^2+\beta\|\x\|_{\wi\L^{r+1}}^{r+1}-\left(\alpha+\frac{\mu\uplambda_1}{2}\right)\|\x\|_{\H}^2
\nonumber\\&\geq\frac{\mu}{2}\|\x\|_{\V}^2+\beta\|\x\|_{\wi\L^{r+1}}^{r+1},
	\end{align}
	where we have employed the Poincar\'e inequality. From \eqref{4.7}, we infer that 
	\begin{align}\label{4.9}
		\int_{\H}e^{\upsigma \|\x\|_{\H}^2}\left(\frac{\mu}{2}\|\x\|_{\V}^2+\beta\|\x\|_{\wi\L^{r+1}}^{r+1}\right)\d\eta(\x)\leq \frac{1}{2}\Tr(\Q)\int_{\H}e^{\upsigma \|\x\|_{\H}^2}\d\eta(\x).
	\end{align}
Using \eqref{4.9}, we find 
\begin{align}\label{412}
	\int_{\H}e^{\upsigma \|\x\|_{\H}^2}\d\eta(\x)&=	\int_{\{\x\in\H:\|\x\|_{\H}\leq R\}}e^{\upsigma \|\x\|_{\H}^2}\d\eta(\x)+	\int_{\{\x\in\H:\|\x\|_{\H}> R\}}e^{\upsigma \|\x\|_{\H}^2}\d\eta(\x)\nonumber\\&\leq e^{\upsigma  R^2}+\frac{1}{R^2}	\int_{\H}\|\x\|_{\H}^2e^{\upsigma \|\x\|_{\H}^2}\d\eta(\x)\leq e^{\upsigma  R^2}+\frac{1}{\uplambda _1R^2}	\int_{\H}\|\x\|_{\V}^2e^{\upsigma \|\x\|_{\H}^2}\d\eta(\x)\nonumber\\&\leq e^{\upsigma  R^2}+ \frac{\Tr(\Q)}{\mu\uplambda_1R^2}\int_{\H}e^{\upsigma \|\x\|_{\H}^2}\d\eta(\x).
\end{align}
Let us now choose an $R>0$ such that $$R^2=\frac{2\Tr(\Q)}{\mu\uplambda_1},$$  so that the estimate \eqref{4p7} follows easily from \eqref{4.9}. 
\end{proof}

The following result is available in \cite{akmtm}, and for completeness, we provide a proof here. 
\begin{lemma}\label{lem4.3}
Under the condition \eqref{improve}, for $\x\in\H$, we have 
	\begin{align}\label{4.14}
	\E\left\{\exp\left[\upsigma\left(\|\X(t)\|_{\H}^2+\mu\int_0^t\|\X(s)\|_{\V}^2\d s+\beta\int_0^t\|\X(s)\|_{\wi\L^{r+1}}^{r+1}\d s\right)\right]\right\}\leq e^{\upsigma(\|\x\|_{\H}^2+\Tr(\Q)t)},
	\end{align}
	for all $t\in[0,T]$. 
\end{lemma}
\begin{proof}
	Let us set 
	$\Z(t):=\|\X(t)\|_{\H}^2+\mu\int_0^t\|\X(s)\|_{\V}^2\d s+\beta\int_0^t\|\X(s)\|_{\wi\L^{r+1}}^{r+1}\d s.$ Then $\Z(\cdot)$ satisfies the following stochastic differential: for $t\in(0,T),$ $\mathbb{P}$-a.s. 
	\begin{align}\label{414}
	\d\Z(t)=-(\mu\|\X(t)\|_{\V}^2+2\alpha\|\X(t)\|_{\H}^2+\beta\|\X(t)\|_{\wi\L^{r+1}}^{r+1}+\Tr(\Q))\d t+2(\sqrt{\Q}\d\W(t),\X(t)).
	\end{align}
	We apply the It\^o formula to the process $\Phi(\cdot)=e^{\upsigma \Z(\cdot)}$ to obtain 
	\begin{align}\label{4.15}
	\d\Phi(t)&=\upsigma\Phi(t)\bigg(\left[-\mu\|\X(t)\|_{\V}^2-2\alpha\|\X(t)\|_{\H}^2-\beta\|\X(t)\|_{\wi\L^{r+1}}^{r+1}+\Tr(\Q)\right]\d t+2(\sqrt{\Q}\d\W(t),\X(t))\nonumber\\&\qquad+2\upsigma\|\Q^{1/2}\X(t)\|_{\H}^2\d t\bigg).
	\end{align}
	Using \eqref{improve}, we get 
	\begin{align*}
	2\upsigma\|\Q^{1/2}\X\|_{\H}^2\leq 2\upsigma\Tr(\Q)\|\X\|_{\H}^2\leq{(\mu\uplambda_1+2\alpha)}\|\X\|_{\H}^2. 
	\end{align*}
	Thus, using the Poincar\'e inequality,  it is immediate from \eqref{4.15} that 
	\begin{align}\label{416}
	\Phi(t)\leq e^{\upsigma\|\x\|_{\H}^2}+\upsigma\Tr(\Q)\int_0^t\Phi(s)\d s+2\upsigma\int_0^t\Phi(s)(\sqrt{\Q}\d\W(s),\X(s)),
	\end{align}
for all $t\in[0,T]$, $\mathbb{P}$-a.s.	Hence, taking expectation in \eqref{416} and then using the fact that the final term appearing the right hand side of the inequality \eqref{416} is a martingale, we find 
	\begin{align}\label{4p17}
	\E\left[\Phi(t)\right]\leq e^{\upsigma\|\x\|_{\H}^2}+\upsigma\Tr(\Q)\int_0^t\E\left[\Phi(s)\right]\d s,
	\end{align}
	for all $t\in[0,T]$. An application of Gr\"onwall's inequality in \eqref{4p17} yields the estimate \eqref{4.14}.  
\end{proof}
Let us now prove an estimate for $\E\left[\|\boldsymbol{\xi}^{\boldsymbol{h}}(t,\x)\|_{\H}^2\right]$, where $\boldsymbol{\xi}^{\boldsymbol{h}}(t,\x)=\D_{\x}\X(t,\x)\boldsymbol{h},$  for all $ \x, \boldsymbol{h}\in\H$. 
\begin{lemma}[Estimates for bounds of derivatives]\label{lem4.4}
	Let us assume that 
	\begin{align}\label{419}
	\mu^2(\mu\uplambda_1+2\alpha)>4\Tr(\Q).
	\end{align}
	Then, we have 
	\begin{align}\label{4p20}
\E\left[\|\boldsymbol{\xi}^{\boldsymbol{h}}(t,\x)\|_{\H}^2\right]\leq \|\boldsymbol{h}\|_{\H}^2e^{\frac{2}{\mu^2}\|\x\|_{\H}^2}e^{-\left(\mu\uplambda_1+2\alpha-\frac{2}{\mu^2}\Tr(\Q)\right)t},
	\end{align}
	for all $t\in[0,T]$. 
\end{lemma}
\begin{proof}
	We write $\boldsymbol{\xi}^{\boldsymbol{h}}(t)$ instead of $\boldsymbol{\xi}^{\boldsymbol{h}}(t,\x),$ for simplicity. Then $\boldsymbol{\xi}^{\boldsymbol{h}}(\cdot)$ satisfies 	for a.e. $t\in[0,T]$,  $\mathbb{P}$-a.s.: 
	\begin{equation}\label{4.21}
	\left\{
	\begin{aligned}
	\frac{\d\boldsymbol{\xi}^{\boldsymbol{h}}(t)}{\d t}+\mu\A\boldsymbol{\xi}^{\boldsymbol{h}}(t)+\alpha\boldsymbol{\xi}^{\boldsymbol{h}}(t)+\B'(\X(t))\boldsymbol{\xi}^{\boldsymbol{h}}(t)+\beta\mathcal{C}'(\X(t))\boldsymbol{\xi}^{\boldsymbol{h}}(t)&=\boldsymbol{0},\\
	\boldsymbol{\xi}^{\boldsymbol{h}}(0)&=\h, 
	\end{aligned}\right. 
	\end{equation}
where $\B'(\X)\boldsymbol{\xi}^{\boldsymbol{h}}=\B(\X,\boldsymbol{\xi}^{\boldsymbol{h}})+\B(\boldsymbol{\xi}^{\boldsymbol{h}},\X)$ and $\mathcal{C}'(\X)$ is defined in \eqref{2133}. Taking the inner product with $\boldsymbol{\xi}^{\boldsymbol{h}}(\cdot)$ to the first equation in \eqref{4.21}, we find 
	\begin{align}\label{4p22}
&	\frac{1}{2}\frac{\d}{\d t}\|\boldsymbol{\xi}^{\boldsymbol{h}}(t)\|_{\H}^2+\mu\|\boldsymbol{\xi}^{\boldsymbol{h}}(t)\|_{\V}^2+\alpha\|\boldsymbol{\xi}^{\boldsymbol{h}}(t)\|_{\H}^2+\beta \||\X(t)|^{\frac{r-1}{2}}\boldsymbol{\xi}^{\boldsymbol{h}}(t)\|_{\H}^2\nonumber\\&\quad+\beta(r-1) \||\X(t)|^{\frac{r-3}{2}}(\X(t)\cdot\boldsymbol{\xi}^{\boldsymbol{h}}(t))\|_{\H}^2\nonumber\\&=-\langle\B(\boldsymbol{\xi}^{\boldsymbol{h}}(t),\X(t)),\boldsymbol{\xi}^{\boldsymbol{h}}(t)\rangle \leq\|\boldsymbol{\xi}^{\boldsymbol{h}}(t)\|_{\wi\L^4}^2\|\X(t)\|_{\V}\leq\sqrt{2}\|\boldsymbol{\xi}^{\boldsymbol{h}}(t)\|_{\H}\|\boldsymbol{\xi}^{\boldsymbol{h}}(t)\|_{\V}\|\X(t)\|_{\V}\nonumber\\&\leq\frac{\mu}{2}\|\boldsymbol{\xi}^{\boldsymbol{h}}(t)\|_{\V}^2+\frac{1}{\mu}\|\X(t)\|_{\V}^2\|\boldsymbol{\xi}^{\boldsymbol{h}}(t)\|_{\H}^2, 
\end{align}
where we have used Ladyzhenskaya's, H\"older's and Young's inequalities. Thus, from \eqref{4p22}, we deduce that 
\begin{align}\label{423}
&	\frac{\d}{\d t}\|\boldsymbol{\xi}^{\boldsymbol{h}}(t)\|_{\H}^2\leq \left[-\left(\mu\uplambda_1+2\alpha\right)+\frac{2}{\mu}\|\X(t)\|_{\V}^2\right]\|\boldsymbol{\xi}^{\boldsymbol{h}}(t)\|_{\H}^2,
\end{align}
for a.e. $t\in[0,T]$. An application of Gr\"onwall's inequality in \eqref{423} gives
\begin{align}\label{424}
\|\boldsymbol{\xi}^{\boldsymbol{h}}(t)\|_{\H}^2\leq \|\boldsymbol{h}\|_{\H}^2e^{-\left(\mu\uplambda_1+2\alpha\right)t}\exp\left(\frac{2}{\mu}\int_0^t\|\X(s)\|_{\V}^2\d s\right),
\end{align}
for all $t\in[0,T]$. Taking expectation in \eqref{424}, we get 
\begin{align}\label{425}
\E\left[\|\boldsymbol{\xi}^{\boldsymbol{h}}(t)\|_{\H}^2\right]\leq \|\boldsymbol{h}\|_{\H}^2e^{-\left(\mu\uplambda_1+2\alpha\right)t}\E\left[\exp\left(\frac{2}{\mu}\int_0^t\|\X(s)\|_{\V}^2\d s\right)\right].
\end{align}
Choosing $\upsigma=\frac{2}{\mu^2}$ in \eqref{improve}, we further have 
\begin{align}\label{4p24}
\E\left[\|\boldsymbol{\xi}^{\boldsymbol{h}}(t)\|_{\H}^2\right]\leq \|\boldsymbol{h}\|_{\H}^2e^{\frac{2}{\mu^2}\|\x\|_{\H}^2}e^{-\left(\mu\uplambda_1+2\alpha-\frac{2}{\mu^2}\Tr(\Q)\right)t},
\end{align}
which converges exponentially fast to zero provided $\mu^3\uplambda_1+2\alpha\mu^2>2\Tr(\Q)$.
\end{proof}

\begin{remark}
	The computation of the estimate \eqref{trop} differs from Lemma 2.1 in the work \cite{VBGD}, where it is obtained using the operator norm definition as $\|\Q^{1/2}\x\|_{\H}^2\leq\|\Q\|_{\mathcal{L}(\H)}\|\x\|_{\H}^2$. In comparison with \cite[Lemma 2.1]{VBGD}, this simplifies several computations in this paper and we obtained the condition \eqref{419} of $\mu$ and $\alpha$  in terms of $\mathrm{Tr}(\Q)$. 
\end{remark}

Let us now prove the existence and uniqueness of invariant measure for the semigroup $\{\P_t\}_{t\geq 0}$.

\begin{theorem}\label{uniqinv}
Under the condition \eqref{419}, the invariant measure $\eta$ for $\{\P_t\}_{t\geq 0}$ is unique. 
\end{theorem}

\begin{proof}
Let us assume that $\eta$ and $\tilde\eta$ are two invariant measures for $\P_t$ and the estimates obtained in Lemmas \ref{lem4.3} and \ref{lem4.4} hold true. Using the invariance, we observe that 
	\begin{align}\label{4.29}
	\left|(\P_t\psi)(\x)-\int_{\H}\psi(\x_1)\tilde\eta(\d\x_1)\right|\leq\int_{\H}|(\P_t\psi)(\x)-(\P_t\psi)(\x_1)|
	\tilde\eta(\d\x_1). 
	\end{align}
	Furthermore, by using Lemma \ref{lem4.4},  for all $\psi\in \C_b^1(\H)$, we have 
	\begin{align}\label{4.30}
	|(\P_t\psi)(\x)-(\P_t\psi)(\x_1)|&\leq\int_0^1|\D_{\x}(\P_t\psi)(\theta\x+(1-\theta)\x_1)\cdot(\x-\x_1)|\d\theta\nonumber\\&\leq\int_0^1\|\psi\|_{1}\E\left(\|\boldsymbol{\xi}^{\x-\x_1}(t,\theta\x+(1-\theta)\x_1)\|_{\H}\right)\d\theta\nonumber\\&\leq\|\psi\|_1 \int_0^1\left\{\E\left(\|\boldsymbol{\xi}^{\x-\x_1}(t,\theta\x+(1-\theta)\x_1)\|_{\H}^2\right)\right\}^{1/2}\d\theta\nonumber\\&\leq\|\psi\|_1 \left(\|\x\|_{\H}+\|\x_1\|_{\H}\right)e^{\frac{2}{\mu^2}(\|\x\|_{\H}^2+\|\x_1\|_{\H}^2)}e^{-\frac{1}{2}\left(\mu\uplambda_1+2\alpha-\frac{2}{\mu^2}\Tr(\Q)\right)t},
	\end{align}
	where we used H\"older's inequlaity and \eqref{4p24}. Using \eqref{4.30} in \eqref{4.29}, we obtain  
	\begin{align}\label{431}
	&	\left|(\P_t\psi)(\x)-\int_{\H}\psi(\x_1)\tilde\eta(\d\x_1)\right|\nonumber\\&\leq \|\psi\|_1e^{-\frac{1}{2}\left(\mu\uplambda_1+2\alpha-\frac{2}{\mu^2}\Tr(\Q)\right)t}e^{\frac{2}{\mu^2}\|\x\|_{\H}^2} \int_{\H}\left(\|\x\|_{\H}+\|\x_1\|_{\H}\right) e^{\frac{2}{\mu^2}\|\x_1\|_{\H}^2}\tilde\eta(\d\x_1)
	\nonumber\\&\leq\|\psi\|_1e^{-\frac{1}{2}\left(\mu\uplambda_1+2\alpha-\frac{2}{\mu^2}\Tr(\Q)\right)t}e^{\frac{2}{\mu^2}\|\x\|_{\H}^2}\bigg[\|\x\|_{\H}\int_{\H}e^{\frac{2}{\mu^2}\|\x_1\|_{\H}^2}\tilde\eta(\d\x_1)+\frac{1}{2}\int_{\H}\|\x_1\|_{\V}^2 e^{\frac{2}{\mu^2}\|\x_1\|_{\H}^2}\eta(\d\x_1)\nonumber\\&\quad+\frac{1}{2\uplambda_1}\int_{\H}e^{\frac{2}{\mu^2}\|\x_1\|_{\H}^2}\eta(\d\x_1)\bigg],
	\end{align}
where  the last term is obtained by using Poincar\'e's, H\"older's and Young's inequalities. Finally, using  Lemma \ref{lem4.1}, we conclude from \eqref{431} that
\begin{align*}
	\left|(\P_t\psi)(\x)-\int_{\H}\psi(\x_1)\tilde\eta(\d\x_1)\right|&\leq\|\psi\|_1e^{-\frac{1}{2}\left(\mu\uplambda_1+2\alpha-\frac{2}{\mu^2}\Tr(\Q)\right)t} e^{\frac{2}{\mu^2}\|\x\|_{\H}^2}\nonumber\\&\quad\times \left(2\|\x\|_{\H}+\frac{1}{\uplambda_1}+\frac{\Tr(\Q)}{\mu}\right)e^{\frac{4\Tr(\Q)}{\mu^3\uplambda_1}},
\end{align*}
	provided the condition \eqref{419} is satisfied. Taking $t\to\infty$ in \eqref{431}, we get 
	\begin{align}
	\lim_{t\to\infty}(\P_t\psi)(\x)=\int_{\H}\psi(\x_1)\tilde\eta(\d\x_1), \ \text{ for all }\ \x\in\H. 
	\end{align}
	Consequently taking $t\to\infty$ in the identity $\int_{\H}(\P_t\psi)(\x)\eta(\d\x)=\int_{\H}\psi(\x)\eta(\d\x),$ yields  $$\int_{\H}\psi(\x_1)\tilde\eta(\d\x_1)=\int_{\H}\psi(\x_1)\eta(\d\x_1)\ \text{ for all }\ \psi\in\C_b^1(\H),$$ and hence we obtain $\eta=\tilde\eta$. 
\end{proof}

\section{Some Auxiliary Results}\label{secappx}\setcounter{equation}{0}
One of the main aims of the upcoming sections is to examine the infinitesimal generator $\mathcal{N}_2$ of $\{\P_t\}_{t\geq 0}$. Once again, we  consider  the Kolmogorov operator \eqref{4p5}. Applying It\^o's formula, it follows that $\mathcal{N}_2\psi=\mathcal{N}_0\psi$, for all $\psi\in\mathscr{E}_{\A}(\H)$ (see Step-2 of Theorem \ref{thm4.7}). Our main goal is to show that $\mathcal{N}_2$ is the closure of $\mathcal{N}_0$ or, equivalently, that $\mathscr{E}_{\A}(\H)$ is the core of $\mathcal{N}_2$.  Then, we say that  $\mathcal{N}_0$ is \emph{essentially $m$-dissipative} (see \cite{VBGA}). The important consequence of this result is the existence and uniqueness of a strong solution (in the sense of Friedrichs) of the elliptic problem
\begin{align*}
	\lambda\psi-\mathcal{N}_0\psi=\g,
\end{align*}
where $\lambda>0$ and $\g\in\L^2(\H;\eta)$ are given. That is, for any $\lambda>0$ and any $\g\in\L^2(\H;\eta)$, there exists a sequence $\{\psi_n\}_{n\in\N}\subset\mathscr{E}_{\A}(\H)$ such that (see \cite[pp. 75]{EANK})
\begin{align*}
	\psi_n\to\psi, \  \lambda\psi_n-\mathcal{N}_0\psi_n\to\g, \ \text{ as } \ n\to\infty \  \text{ in } \ \L^2(\H;\eta),
\end{align*} 
 In order to do this, we need to find a bound for the  following interpolation estimate: $$\int_{\H}\|\A^{\delta}\x\|_{\H}^{2m-2}\|\A^{\delta+\frac{1}{2}}\x\|_{\H}^2\eta(\d\x),$$ where $\delta>0$ and $m\in\mathbb{N}$, which is obtained by considering an approximated problem. 
 \subsection{An approximated problem}
 We first approximate the system \eqref{32} by the following regular system: 
\begin{equation}\label{4.33}
\left\{
\begin{aligned}
\d\X_{\e}(t)+[\mu \A\X_{\e}(t)+\alpha\X_{\e}(t)+\B_{\e}(\X_{\e}(t))+\beta\mathcal{C}_{\e}(\X_{\e}(t))]\d t&=\sqrt{\Q}\d\W(t), \ t>0,\\
\X_{\e}(0)&=\x\in\H,
\end{aligned}
\right.
\end{equation}
where 
\begin{align}\label{modiB}
\B_{\e}(\x)=\left\{\begin{array}{cc}\B(\x)&\text{ if }\ \|\x\|_{\V}\leq\e^{-1},\\
\e^{-2}\|\x\|_{\V}^{-2}\B(\x)&\text{ if }\ \|\x\|_{\V}>\e^{-1},\end{array}\right.
\end{align}
and 
\begin{align}\label{modiC}
	\mathcal{C}_{\e}(\x)=\left\{\begin{array}{cc}\mathcal{C}(\x)&\text{ if }\ \|\x\|_{\V}\leq\e^{-1},\\
		\e^{-(r+1)}\|\x\|_{\V}^{-(r+1)}\mathcal{C}(\x)&\text{ if }\ \|\x\|_{\V}>\e^{-1}.\end{array}\right.
\end{align}
Since, for $\e>0$, $\B_{\e}(\cdot)$ and $\mathcal{C}_{\e}(\cdot)$ are regular and bounded, the problem \eqref{4.33} possesses a unique strong solution, which  shall be denoted by $\X_{\e}(t,\x)$. For $\e>0$, let $\{\P_t^{\e}\}_{t\geq0}$ be the transition semigroup associated with \eqref{4.33}, which is given by
\begin{align*}
	\P_t^\e\psi(\x)=\E[\psi(\X_{\e}(t,\x))], \ \psi\in\C_b(\H).
\end{align*}  It should be noted that the estimate \eqref{tight} remains true for the system \eqref{4.33} also. This implies that there exists an invariant measure $\eta_{\e}$ of $\{\P_t^{\e}\}_{t\geq 0}$ and the sequence $\{\eta_{\e}\}_{\e>0}$ is tight. 
 Hence, by the uniqueness, the invariant measure $\eta$ for $\{\P_t\}_{t\geq0}$ is the weak limit of $\{\eta_{\e}\}_{\e>0}$. Furthermore, we note that Lemma \ref{lem4.1} remains valid for the semigroup $\{\P_t^{\e}\}_{\e>0}$ with the constants independent of $\e$. In the sequel, we need the following interpolation estimates for $\delta\in(0,\frac{1}{2})$ (cf. \cite{VBGD}): 
\begin{align}
\|\x\|_{\wi\L^{\infty}}&\leq C_{a}\|\A^{\delta}\x\|_{\H}^{2\delta}\|\A^{\delta+\frac{1}{2}}\x\|_{\H}^{1-2\delta}, \ \x\in\D(\A^{\delta+\frac{1}{2}}), \label{435}\\
\|\x\|_{\V}&\leq \|\A^{\delta}\x\|_{\H}^{2\delta}\|\A^{\delta+\frac{1}{2}}\x\|_{\H}^{1-2\delta}, \ \x\in\D(\A^{\delta+\frac{1}{2}}), \label{436}\\
\|\A^{2\delta}\x\|_{\H}&\leq \|\A^{\delta}\x\|_{\H}^{1-2\delta}\|\A^{\delta+\frac{1}{2}}\x\|_{\H}^{2\delta}, \ \x\in\D(\A^{\delta+\frac{1}{2}}),\label{437}
\end{align}
The first inequality follows from the fact that $\D(\A^{\alpha})$ is a closed subspace of $\H^{2\alpha}(\mathcal{O})$ and Agmon's inequality. The last two inequalities follow from the application of interpolation inequality.

\begin{lemma}\label{lem4.5}
	Let us assume that
	 \begin{align}\label{Asmp}
	 	\Tr(\A^{2\delta}\Q)<+\infty,
	 \end{align} 
	 for some $\delta\in(0,\frac{1}{2})$. Then, there are some positive constants $\gamma_i$, for $i=1,2$ depending on $m$ such that if 
	\begin{align}\label{439}
	\mu(\mu+\alpha)^2>\gamma_1\max\{4\Tr(\Q),\Tr(\A^{2\delta}\Q)\},
	\end{align}
then the following estimate holds for all $\eps>0$:
\begin{align*}
	&\int_{\H}e^{\lambda\|\x\|_{\H}^2}\|\A^{\delta}\x\|_{\H}^{2m}\left(\|\x\|_{\V}^2+\|\x\|_{\H}^2+\|\x\|_{\wi\L^{r+1}}^{r+1}\right)\eta_{\e}(\d\x)\nonumber\\&+ \int_{\H}e^{\lambda\|\x\|_{\H}^2}\|\A^{\delta}\x\|_{\H}^{2(m-1)} \|\A^{\delta+\frac{1}{2}}\x\|_{\H}^2\eta_{\e}(\d\x)\leq\gamma_2,
\end{align*}
for all $m\in\N$ and 
\begin{align*}
	0\leq\lambda<\gamma_3\max\{4\Tr(\Q),\Tr(\A^{2\delta}\Q)\}. 
\end{align*}.
\end{lemma}

\begin{proof}
	Let us set 
	\begin{align}
	\psi(\x)=e^{\upnu\|\x\|_{\H}^2}\|\A^{\delta}\x\|_{\H}^{2m}, \ \text{ for all }\ \x\in\D(\A^{\delta}). 
	\end{align}
Then, it can be easily seen that (see \cite{VBGD} also) 
\begin{align}
	\D_{\x}\psi(\x)&=e^{\upnu\|\x\|_{\H}^2}[2\upnu\|\A^{\delta}\x\|_{\H}^{2m}\x+2m\|\A^{\delta}\x\|_{\H}^{2(m-1)}\A^{2\delta}\x]\label{est1},\\
	\D_{\x}^2\psi(\x)&=e^{\upnu\|\x\|_{\H}^2}[4\upnu^2\|\A^{\delta}\x\|_{\H}^{2m}\x\otimes\x+2\upnu\|\A^{\delta}\x\|_{\H}^{2m}\nonumber\\&\quad+4\upnu m\|\A^{\delta}\x\|_{\H}^{2(m-1)}(\A^{2\delta}\x\otimes\x+\x\otimes\A^{2\delta}\x)\nonumber\\&\quad+4m(m-1)\|\A^{\delta}\x\|_{\H}^{2(m-2)}\A^{2\delta}\x\otimes\A^{2\delta}\x+2m\|\A^{\delta}\x\|_{\H}^{2(m-1)}\A^{2\delta}]\label{est2},
\end{align}
and 
\begin{align*}
	\Tr[\Q\D_{\x}^2\psi(\x)]&=e^{\upnu\|\x\|_{\H}^2}[4\upnu^2\|\A^{\delta}\x\|_{\H}^{2m}\|\Q^{1/2}\x\|_{\H}^2++2\upnu\|\A^{\delta}\x\|_{\H}^{2m}\Tr(\Q)\nonumber\\&\quad+8\upnu m\|\A^{\delta}\x\|_{\H}^{2(m-1)}(\A^{2\delta}\x,\Q\x)+4m(m-1)\|\A^{\delta}\x\|_{\H}^{2(m-2)}\|\Q^{1/2}\A^{2\delta}\x\|_{\H}^2\nonumber\\&\quad+2m\|\A^{\delta}\x\|_{\H}^{2(m-1)}\Tr(\A^{2\delta}\Q)]. 
\end{align*}
Thus, it follows from the above equality and \eqref{trop} that 
	\begin{align}\label{est3}
	\Tr[\Q\D_{\x}^2\psi(\x)]&\leq e^{\upnu\|\x\|_{\H}^2} \Big[4\upnu^2\Tr(\Q)\|\A^{\delta}\x\|_{\H}^{2m}\|\x\|_{\H}^2+2\upnu\|\A^{\delta}\x\|_{\H}^{2m}\Tr(\Q)\nonumber\\&\quad+
	8\upnu m\Tr(\Q)\|\A^{\delta}\x\|_{\H}^{2(m-1)} \|\A^{2\delta}\x\|_{\H}\|\x\|_{\H}\nonumber\\&\quad+4m(m-1)\Tr(\Q)\|\A^{\delta}\x\|_{\H}^{2(m-2)}\|\A^{2\delta}\x\|_{\H}^2+2m\|\A^{\delta}\x\|_{\H}^{2(m-1)}\Tr(\A^{2\delta}\Q)\Big]\nonumber\\&\leq2\max\{4\Tr(\Q),\Tr(\A^{2\delta}\Q)\}e^{\upnu\|\x\|_{\H}^2}\|\A^{\delta}\x\|_{\H}^{2m-4}
	\bigg[\upnu^2\|\A^{\delta}\x\|_{\H}^4\|\x\|_{\H}^2+\upnu\|\A^{\delta}\x\|_{\H}^4\nonumber\\&\quad+\upnu m \|\A^{\delta}\x\|_{\H}^2\|\A^{2\delta}\x\|_{\H}\|\x\|_{\H}+m(m-1)\|\A^{2\delta}\x\|_{\H}^2+m\|\A^{\delta}\x\|_{\H}^2\bigg].
	\end{align}
	Using \eqref{est1}, we find
	\begin{align}\label{est4}
&(\mu\A\x+\alpha\x+\B(\x)+\beta\mathcal{C}(\x),\D_{\x}\psi(\x))\nonumber\\&
 =2\upnu e^{\upnu\|\x\|_{\H}^2}\|\A^{\delta}\x\|_{\H}^{2m} (\mu\|\x\|_{\V}^2+\alpha\|\x\|_{\H}^2+\beta\|\x\|_{\wi\L^{r+1}}^{r+1})\nonumber\\&\quad+2m e^{\upnu\|\x\|_{\H}^2}\|\A^{\delta}\x\|_{\H}^{2(m-1)}(\mu\|\A^{\delta+\frac{1}{2}}\x\|_{\H}^2+\alpha\|\A^{\delta}\x\|_{\H}^2)\nonumber\\&\quad+ 2m e^{\upnu\|\x\|_{\H}^2}\|\A^{\delta}\x\|_{\H}^{2(m-1)}\left[\underbrace{(\B(\x),\A^{2\delta}\x)}_{=:I_1}+\beta\underbrace{(\mathcal{C}(\x),\A^{2\delta}\x)}_{=:I_2}\right].
	\end{align}
Let  $\mathcal{N}_2^{\e}$ be the infinitesimal generator of the transition semigroup $\{\P_t^{\e}\}_{t\geq 0}$ in $\L^2(\H,\eta_\e)$.	Substituting \eqref{est3} and \eqref{est4} into the equation $$\int_{\H} \mathcal{N}_2^{\e}\psi(\x)\eta_{\e}(\d\x)=0,$$  we obtain 
	\begin{align}\label{442}
	&2\upnu\int_{\H}e^{\upnu\|\x\|_{\H}^2}\left(\|\A^{\delta}\x\|_{\H}^{2m} (\mu\|\x\|_{\V}^2+\alpha\|\x\|_{\H}^2+\beta\|\x\|_{\wi\L^{r+1}}^{r+1})\right)\eta_{\e}(\d\x)\nonumber\\&\quad +2m	\int_{\H}e^{\upnu\|\x\|_{\H}^2}\left(\|\A^{\delta}\x\|_{\H}^{2(m-1)}(\mu\|\A^{\delta+\frac{1}{2}}\x\|_{\H}^2+\alpha\|\A^{\delta}\x\|_{\H}^2)\right)\eta_{\e}(\d\x)\nonumber\\&=-2m \int_{\H}e^{\upnu\|\x\|_{\H}^2}\|\A^{\delta}\x\|_{\H}^{2(m-1)}
	\left[I_1+\beta I_2\right] \eta_{\e}(\d\x)+ \frac{1}{2}\int_{\H}\Tr\left[\Q\D_{\x}^2\psi(\x)\right]\eta_{\e}(\d\x)\nonumber\\&\leq -2m \int_{\H}e^{\upnu\|\x\|_{\H}^2}\|\A^{\delta}\x\|_{\H}^{2(m-1)}
	\left[I_1+\beta I_2\right] \eta_{\e}(\d\x)\nonumber\\&\quad+
	\max\{4\Tr(\Q),\Tr(\A^{2\delta}\Q)\}\int_{\H}e^{\upnu\|\x\|_{\H}^2}\|\A^{\delta}\x\|_{\H}^{2m-4}\left(\upnu^2\|\A^{\delta}\x\|_{\H}^{4}\|\x\|_{\H}^2+\upnu\|\A^{\delta}\x\|_{\H}^{4}\right.\nonumber\\&\qquad\left.+\upnu m\|\A^{\delta}\x\|_{\H}^2 \|\A^{2\delta}\x\|_{\H}\|\x\|_{\H}+m(m-1)\|\A^{2\delta}\x\|_{\H}^2+m\|\A^{\delta}\x\|_{\H}^{2}\right)\eta_{\e}(\d\x). 
	\end{align}
	Let us denote  
	\begin{align}\label{sigma}
		\sigma(\Q,\delta):=\max\{4\Tr(\Q),\Tr(\A^{2\delta}\Q)\}.
		\end{align}
	Using \eqref{435}, \eqref{437}, H\"older's and Young's inequalities, we estimate $I_1$ as 
	\begin{align}\label{443}
|I_1|\leq \|\x\|_{\wi\L^{\infty}}\|\x\|_{\V}\|\A^{2\delta}\x\|_{\H}&\leq
C_a\|\x\|_{\V}\|\A^{\delta}\x\|_{\H}\|\A^{\delta+\frac{1}{2}}\x\|_{\H}\nonumber\\&\leq\frac{\upnu\mu}{2m}\|\x\|_{\V}^2\|\A^{\delta}\x\|_{\H}^2+\frac{C_a^2m}{2\mu\upnu}\|\A^{\delta+\frac{1}{2}}\x\|_{\H}^2. 
	\end{align}
In order to estimate $(\mathcal{C}(\x),\A^{2\delta}\x)$, we consider the cases $r=2,3$ separately. 	For $r=2$, we estimate $(\mathcal{C}(\x),\A^{2\delta}\x)$, using \eqref{436}-\eqref{437}, Ladyzhenskaya's, H\"older's and Young's inequalities as 
	\begin{align}\label{444}
	|I_2|\leq\|\x\|_{\wi\L^4}^2\|\A^{2\delta}\x\|_{\H}\leq\sqrt{2}\|\x\|_{\H}\|\x\|_{\V}\|\A^{2\delta}\x\|_{\H}&\leq \sqrt{2}\|\x\|_{\H}\|\A^{\delta}\x\|_{\H}\|\A^{\delta+\frac{1}{2}}\x\|_{\H}\nonumber\\&\leq\frac{\upnu\alpha}{2m\beta}
	\|\x\|_{\H}^2\|\A^{\delta}\x\|_{\H}^2+\frac{m\beta}{2\mu\upnu} \|\A^{\delta+\frac{1}{2}}\x\|_{\H}^2.   
	\end{align}
	For $r=3$, the term $(\mathcal{C}(\x),\A^{2\delta}\x)$ can be estimated using  \eqref{435}, \eqref{437},  interpolation, H\"older's and Young's inequalities as  
	\begin{align}\label{445}
   |I_2|\leq\|\x\|_{\wi\L^{\infty}}\|\x\|_{\wi\L^4}^2\|\A^{2\delta}\x\|_{\H}
   &\leq C_a\|\x\|_{\wi\L^4}^2\|\A^{\delta}\x\|_{\H}\|\A^{\delta+\frac{1}{2}}\x\|_{\H}\nonumber\\&\leq\frac{\upnu}{2m}\|\x\|_{\wi\L^4}^4\|\A^{\delta}\x\|_{\H}^2+\frac{C_a^2m}{\upnu}\|\A^{\delta+\frac{1}{2}}\x\|_{\H}^2. 
	\end{align}
	Thus, for $r=2$, combining \eqref{443} and \eqref{444},  and then substituting it in \eqref{442}, we obtain 
	\begin{align}\label{et0}
	&\upnu\int_{\H}e^{\upnu\|\x\|_{\H}^2}\|\A^{\delta}\x\|_{\H}^{2m}\left(\mu\|\x\|_{\V}^2+\alpha\|\x\|_{\H}^2+2\beta\|\x\|_{\wi\L^{r+1}}^{r+1}\right)\eta_{\e}(\d\x) \nonumber\\&\quad+2m	\int_{\H}e^{\upnu\|\x\|_{\H}^2}\|\A^{\delta}\x\|_{\H}^{2(m-1)}\left(\mu\|\A^{\delta+\frac{1}{2}}\x\|_{\H}^2+\alpha\|\A^{\delta}\x\|_{\H}^2\right)\eta_{\e}(\d\x)\nonumber\\&\leq \frac{m^2}{\mu\upnu}(C_a^2+\beta^2)\int_{\H}e^{\upnu\|\x\|_{\H}^2}\|\A^{\delta}\x\|_{\H}^{2m-2}\|\A^{\delta+\frac{1}{2}}\x\|_{\H}^2\eta_{\e}(\d\x) \nonumber\\&\quad+\sigma(\Q,\delta)\int_{\H}e^{\upnu\|\x\|_{\H}^2}\|\A^{\delta}\x\|_{\H}^{2m-4}\left(\upnu^2\|\A^{\delta}\x\|_{\H}^{4}\|\x\|_{\H}^2+\upnu\|\A^{\delta}\x\|_{\H}^{4}\right.\nonumber\\&\qquad\left.+\upnu m\|\A^{\delta}\x\|_{\H}^2\|\A^{2\delta}\x\|_{\H}\|\x\|_{\H}+m(m-1)\|\A^{2\delta}\x\|_{\H}^2+m\|\A^{\delta}\x\|_{\H}^{2}\right)\eta_{\e}(\d\x)
	\nonumber\\&\leq\frac{m^2}{\mu\upnu}(C_a^2+\beta^2)
	\int_{\H}e^{\upnu\|\x\|_{\H}^2}\|\A^{\delta}\x\|_{\H}^{2m-2}
	\|\A^{\delta+\frac{1}{2}}\x\|_{\H}^2\eta_{\e}(\d\x)\nonumber\\&\quad+
	\upnu^2\sigma(\Q,\delta)\underbrace{\int_{\H}e^{\upnu\|\x\|_{\H}^2}\|\A^{\delta}\x\|_{\H}^{2m}\|\x\|_{\H}^2\eta_{\e}(\d\x)}_{=:I_{22}}+\upnu\sigma(\Q,\delta)\int_{\H}e^{\upnu\|\x\|_{\H}^2}\|\A^{\delta}\x\|_{\H}^{2m}\eta_{\e}(\d\x)\nonumber\\&\quad
	+m\sigma (\Q,\delta)\bigg[\upnu \underbrace{\int_{\H}e^{\upnu\|\x\|_{\H}^2}\|\A^{\delta}\x\|_{\H}^{2m-2}\|\A^{2\delta}\x\|_{\H}\|\x\|_{\H}\eta_{\e}(\d\x)}_{=:I_3}
	\nonumber\\&\quad+m
	\underbrace{\int_{\H}e^{\upnu\|\x\|_{\H}^2}\|\A^{\delta}\x\|_{\H}^{2m-4}\|\A^{2\delta}\x\|_{\H}^2\eta_{\e}(\d\x)}_{=:I_4}+
	\underbrace{\int_{\H}e^{\upnu\|\x\|_{\H}^2}\|\A^{\delta}\x\|_{\H}^{2m-2}\eta_{\e}(\d\x)}_{=:I_5}\bigg],
	\end{align} 
   where $\sigma (\Q,\delta)$  is defined in \eqref{sigma}. Let us first calculate $I_{22}$, by using \eqref{fracA2} and Young's inequality, with exponent $m$ and $\frac{m}{m-1}$, as follows:
    \begin{align}\label{et1y}
         I_{22}&\leq\frac{1}{\uplambda_1^{2\delta-1}}
    	\int_{\H}e^{\upnu\|\x\|_{\H}^2} \|\A^{\delta}\x\|_{\H}^{2m-2}
    	\|\x\|_{\H}^2\|\x\|_{\V}^2\eta_{\e}(\d\x)
    	\nonumber\\&\leq\frac{1}{\uplambda_1^{2\delta-1}}
    	\int_{\H}e^{\upnu\|\x\|_{\H}^2}\left(\kappa_3\|\A^{\delta}\x\|_{\H}^{2m}+c(\kappa_3)\|\x\|_{\H}^{2m}\right)\|\x\|_{\V}^2\eta_{\e}(\d\x).
    \end{align} 
    By using the interpolation estimate \eqref{437} and Young's inequality, we compute
	\begin{align}\label{et1}
	|I_3|&\leq\int_{\H}e^{\upnu\|\x\|_{\H}^2}\|\A^{\delta}\x\|_{\H}^{2m-2}\|\A^{\delta}\x\|_{\H}^{1-2\delta}\|\A^{\delta+\frac{1}{2}}\x\|_{\H}^{2\delta}\|\x\|_{\H}\eta_{\e}(\d\x)\nonumber\\&\leq
	\int_{\H}e^{\upnu\|\x\|_{\H}^2}\left((1-2\delta)\|\A^{\delta}\x\|_{\H}^{m-1}\|\A^{\delta}\x\|_{\H}+2\delta\|\A^{\delta}\x\|_{\H}^{m-1}\|\A^{\delta+\frac{1}{2}}\x\|_{\H}^{2\delta}\right)\|\A^{\delta}\x\|_{\H}^{m-1}\|\x\|_{\H}\eta_{\e}(\d\x)\nonumber\\&\leq
	\frac{1}{\upnu\sigma(\delta,\Q)}\int_{\H}e^{\upnu\|\x\|_{\H}^2}
	\|\A^{\delta}\x\|_{\H}^{2(m-1)}\left(\mu\|\A^{\delta+\frac{1}{2}}\x\|_{\H}^2+\alpha\|\A^{\delta}\x\|_{\H}^2\right)\eta_{\e}(\d\x)
	\nonumber\\&\quad+\upnu\sigma(\delta,\Q)
	\left(\frac{(2\delta)^2}{4\mu}+\frac{(1-2\delta)^2}{4\alpha}\right)
    \int_{\H}e^{\upnu\|\x\|_{\H}^2}
    \|\A^{\delta}\x\|_{\H}^{2(m-1)}\|\x\|_{\H}^2\eta_{\e}(\d\x)
    \nonumber\\&<
    \frac{1}{\upnu\sigma(\delta,\Q)}\int_{\H}e^{\upnu\|\x\|_{\H}^2}
    \|\A^{\delta}\x\|_{\H}^{2(m-1)}\left(\mu\|\A^{\delta+\frac{1}{2}}\x\|_{\H}^2+\alpha\|\A^{\delta}\x\|_{\H}^2\right)\eta_{\e}(\d\x)
    \nonumber\\&\quad+\upnu\sigma(\delta,\Q)
    \left(\frac{1}{\mu}+\frac{1}{\alpha}\right)
    \underbrace{\int_{\H}e^{\upnu\|\x\|_{\H}^2}
    	\|\A^{\delta}\x\|_{\H}^{2(m-1)}\|\x\|_{\H}^2\eta_{\e}(\d\x)}_{\text{ Young's inequality with exponent }  m \text{ and } \frac{m}{m-1}}
    \nonumber\\&\leq
     \frac{1}{\upnu\sigma(\delta,\Q)}\int_{\H}e^{\upnu\|\x\|_{\H}^2}
    \|\A^{\delta}\x\|_{\H}^{2(m-1)}\left(\mu\|\A^{\delta+\frac{1}{2}}\x\|_{\H}^2+\alpha\|\A^{\delta}\x\|_{\H}^2\right)\eta_{\e}(\d\x)
     \nonumber\\&\quad+\upnu\sigma(\delta,\Q)
    \left(\frac{1}{\mu}+\frac{1}{\alpha}\right)\int_{\H}e^{\upnu\|\x\|_{\H}^2}\left(\kappa\|\A^{\delta}\x\|_{\H}^{2m}+c(\kappa)\|\x\|_{\H}^{2m}\right)\eta_{\e}(\d\x),
	\end{align}
	
where the constant $\kappa>0$ will be specified later. In a similar way, we find
\begin{align}\label{et2}
|I_4|&\leq\int_{\H}e^{\upnu\|\x\|_{\H}^2}\|\A^{\delta}\x\|_{\H}^{2m-4}\|\A^{\delta}\x\|_{\H}^{2-4\delta}\|\A^{\delta+\frac{1}{2}}\x\|_{\H}^{4\delta}\eta_{\e}(\d\x)\nonumber\\&=\int_{\H}e^{\upnu\|\x\|_{\H}^2}\|\A^{\delta}\x\|_{\H}^{4(m-1)\delta}\|\A^{\delta+\frac{1}{2}}\x\|_{\H}^{4\delta}\|\A^{\delta}\x\|_{\H}^{2(m-1)(1-2\delta)-4\delta}\eta_{\e}(\d\x)\nonumber\\&\leq\int_{\H}e^{\upnu\|\x\|_{\H}^2}\|\A^{\delta}\x\|_{\H}^{4(m-1)\delta}\|\A^{\delta+\frac{1}{2}}\x\|_{\H}^{4\delta}\eta_{\e}(\d\x)\nonumber\\&\quad+\int_{\H}e^{\upnu\|\x\|_{\H}^2}\|\A^{\delta}\x\|_{\H}^{4(m-1)\delta}\|\A^{\delta+\frac{1}{2}}\x\|_{\H}^{4\delta}\|\A^{\delta}\x\|_{\H}^{2(m-1)(1-2\delta)}\eta_{\e}(\d\x)\nonumber\\&\leq\kappa_1\int_{\H}e^{\upnu\|\x\|_{\H}^2} \|\A^{\delta}\x\|_{\H}^{2(m-1)}\|\A^{\delta+\frac{1}{2}}\x\|_{\H}^2\eta_{\e}(\d\x)+c(\kappa_1)\int_{\H}e^{\upnu\|\x\|_{\H}^2}\eta_{\e}(\d\x)\nonumber\\&\quad+\kappa_2\int_{\H}e^{\upnu\|\x\|_{\H}^2}\|\A^{\delta}\x\|_{\H}^{2(m-1)}\|\A^{\delta+\frac{1}{2}}\x\|_{\H}^2\eta_{\e}(\d\x)+c(\kappa_2)\int_{\H}e^{\upnu\|\x\|_{\H}^2}\|\A^{\delta}\x\|_{\H}^{2(m-1)}\eta_{\e}(\d\x),
\end{align}
where we have used the fact that $x\leq1+x^p$, for any $p\geq1$ and $x\geq0$, and the constants $\kappa_1>0$ and $\kappa_2>0$ will be specified later. In a similar way, one can compute 
\begin{align}\label{et4}
	 I_5\leq\int_{\H}e^{\upnu\|\x\|_{\H}^2}\eta_{\e}(\d\x)+ \int_{\H}e^{\upnu\|\x\|_{\H}^2}\|\A^{\delta}\x\|_{\H}^{2m}\eta_{\e}(\d\x).  
\end{align} 
Combining \eqref{et1}-\eqref{et4}, we obtain
\begin{align}\label{et9}
\upnu I_3+m I_4+I_5&\leq
 \frac{1}{\sigma(\delta,\Q)}\int_{\H}e^{\upnu\|\x\|_{\H}^2}
\|\A^{\delta}\x\|_{\H}^{2(m-1)}\left(\mu\|\A^{\delta+\frac{1}{2}}\x\|_{\H}^2+\alpha\|\A^{\delta}\x\|_{\H}^2\right)\eta_{\e}(\d\x)
\nonumber\\&\quad+\upnu^2\sigma(\delta,\Q)
\left(\frac{1}{\mu}+\frac{1}{\alpha}\right)\int_{\H}e^{\upnu\|\x\|_{\H}^2}\left(\kappa\|\A^{\delta}\x\|_{\H}^{2m}+c(\kappa)\|\x\|_{\H}^{2m}
\right)\eta_{\e}(\d\x)
\nonumber\\&\quad+\bigg(1+m c(\kappa_2)\bigg)\int_{\H} e^{\upnu\|\x\|_{\H}^2}\|\A^{\delta}\x\|_{\H}^{2m}\eta_{\e}(\d\x)
\nonumber\\&\quad+
\bigg(1+m c(\kappa_1)+m c(\kappa_2)\bigg)\int_{\H} e^{\upnu\|\x\|_{\H}^2}\eta_{\e}(\d\x)
\nonumber\\&\quad
+\bigg(m\kappa_1+m\kappa_2\bigg)\int_{\H} e^{\upnu\|\x\|_{\H}^2} \|\A^{\delta}\x\|_{\H}^{2(m-1)}\|\A^{\delta+\frac{1}{2}}\x\|_{\H}^2\eta_{\e}(\d\x).
\end{align}
Finally, by using \eqref{et1y} and \eqref{et9} in \eqref{et0}, we obtain
\begin{align}\label{et10}
	&\upnu\int_{\H}e^{\upnu\|\x\|_{\H}^2}\|\A^{\delta}\x\|_{\H}^{2m}\left(\mu\|\x\|_{\V}^2+\alpha\|\x\|_{\H}^2+2\beta\|\x\|_{\wi\L^{r+1}}^{r+1}\right)\eta_{\e}(\d\x) \nonumber\\&\quad+m	\int_{\H}e^{\upnu\|\x\|_{\H}^2}\|\A^{\delta}\x\|_{\H}^{2(m-1)}\left(
	\mu\|\A^{\delta+\frac{1}{2}}\x\|_{\H}^2+\alpha\|\A^{\delta}\x\|_{\H}^2\right)\eta_{\e}(\d\x)\nonumber\\&\leq \frac{m^2}{\mu\upnu}(C_a^2+\beta^2)\int_{\H}e^{\upnu\|\x\|_{\H}^2}\|\A^{\delta}\x\|_{\H}^{2m-2}\|\A^{\delta+\frac{1}{2}}\x\|_{\H}^2\eta_{\e}(\d\x)
	\nonumber\\&\quad+m\upnu^2\sigma^2(\Q,\delta)
	\left(\frac{1}{\mu}+\frac{1}{\alpha}\right)\int_{\H}e^{\upnu\|\x\|_{\H}^2}\left(\kappa\|\A^{\delta}\x\|_{\H}^{2m}+c(\kappa)\|\x\|_{\H}^{2m}
	\right)\eta_{\e}(\d\x)\nonumber\\&\quad+
	\frac{\upnu^2\sigma(\Q,\delta)}{\uplambda_1^{2\delta-1}}
	\int_{\H}e^{\upnu\|\x\|_{\H}^2}\left(\kappa_3\|\A^{\delta}\x\|_{\H}^{2m}+c(\kappa_3)\|\x\|_{\H}^{2m}\right)\|\x\|_{\V}^2\eta_{\e}(\d\x)\nonumber\\&\quad+\sigma(\Q,\delta)\bigg(m+\upnu+m^2 c(\kappa_2)\bigg)\int_{\H} e^{\upnu\|\x\|_{\H}^2}\|\A^{\delta}\x\|_{\H}^{2m}\eta_{\e}(\d\x)
	\nonumber\\&\quad+\sigma(\Q,\delta)\bigg(m+m^2 c(\kappa_1)+m^2 c(\kappa_2)\bigg)\int_{\H} e^{\upnu\|\x\|_{\H}^2}\eta_{\e}(\d\x)\nonumber\\&\quad
	+\sigma(\Q,\delta)\bigg(m^2\kappa_1+m^2\kappa_2\bigg)\int_{\H} e^{\upnu\|\x\|_{\H}^2}  \|\A^{\delta}\x\|_{\H}^{2(m-1)}\|\A^{\delta+\frac{1}{2}\x\|_{\H}^2}\eta_{\e}(\d\x).
\end{align}

We choose the constants $\kappa=\frac{1}{\upnu^2\sigma^2(\Q,\delta)
	\left(\frac{1}{\mu}+\frac{1}{\alpha}\right)}\frac{\alpha}{2},$ $\kappa_3=\frac{1}{\upnu\sigma(\Q,\delta)C_e}\frac{\mu}{2},$and $\kappa_1=\kappa_2=\frac{\mu}{4m\sigma(\Q,\delta)}$, so that one can find the constants 
\begin{align}\label{et6}
	c(\kappa)&=\left(\frac{m-1}{\kappa m}\right)^{m-1}\frac{1}{m}=
	\left(\upnu^2\sigma^2(\Q,\delta)\left(\frac{1}{\mu}+ \frac{1}{\alpha}\right)\frac{2}{\alpha}\right)^{m-1}
	\left(\frac{m-1}{ m}\right)^{m-1}\frac{1}{m},
	\nonumber\\
	c(\kappa_3)&=\left(\frac{m-1}{\kappa_3 m}\right)^{m-1}\frac{1}{m}=
	\left(\upnu\sigma(\Q,\delta)\frac{2}{\mu\uplambda_1^{2\delta-1}}
	\right)^{m-1}
	\left(\frac{m-1}{ m}\right)^{m-1}\frac{1}{m}\nonumber\\ c(\kappa_1)&=c(\kappa_2)=\frac{1-2\delta}{\left(\frac{\kappa_1}{2\delta}\right)^{\frac{2\delta}{1-2\delta}}}
	=\frac{1-2\delta}{\left(\frac{\mu}{8m\delta\sigma(\Q,\delta)}\right)^{\frac{2\delta}{1-2\delta}}}.
\end{align}
Substituting the values of $\kappa, \kappa_3, \kappa_1$ and $\kappa_2$ in \eqref{et10}, we obtain
\begin{align}\label{et11}
	&\upnu\int_{\H}e^{\upnu\|\x\|_{\H}^2}\|\A^{\delta}\x\|_{\H}^{2m}\left(\frac{\mu}{2}\|\x\|_{\V}^2+\alpha\|\x\|_{\H}^2+2\beta\|\x\|_{\wi\L^{r+1}}^{r+1}\right)\eta_{\e}(\d\x) \nonumber\\&\quad+\frac{m}{2}	\int_{\H}e^{\upnu\|\x\|_{\H}^2}\|\A^{\delta}\x\|_{\H}^{2m-2}\left(\mu\|\A^{\delta+\frac{1}{2}}\x\|_{\H}^2+\alpha\|\A^{\delta}\x\|_{\H}^2\right)\eta_{\e}(\d\x)\nonumber\\&\leq \frac{m^2}{\mu\upnu}(C_a^2+\beta^2)\int_{\H}e^{\upnu\|\x\|_{\H}^2}\|\A^{\delta}\x\|_{\H}^{2m-2}\|\A^{\delta+\frac{1}{2}}\x\|_{\H}^2\eta_{\e}(\d\x)\nonumber\\&\quad+m\upnu^2\sigma^2(\Q,\delta)
	\left(\frac{1}{\mu}+\frac{1}{\alpha}\right)c(\kappa)
		\int_{\H}e^{\upnu\|\x\|_{\H}^2}\|\x\|_{\H}^{2m}\eta_{\e}(\d\x)\nonumber\\&\quad+\frac{\upnu^2\sigma(\Q,\delta)}{\uplambda_1^{2\delta-1}}c(\kappa_3)
	\int_{\H}e^{\upnu\|\x\|_{\H}^2}\|\x\|_{\H}^{2m}\|\x\|_{\V}^2\eta_{\e}(\d\x)\nonumber\\&\quad+\sigma(\Q,\delta)\bigg(m+\upnu+m^2 c(\kappa_2)\bigg)\int_{\H} e^{\upnu\|\x\|_{\H}^2}\|\A^{\delta}\x\|_{\H}^{2m}\eta_{\e}(\d\x)
	\nonumber\\&\quad+\sigma(\Q,\delta)\bigg(m+m^2 c(\kappa_1)+m^2 c(\kappa_2)\bigg)\int_{\H} e^{\upnu\|\x\|_{\H}^2}\eta_{\e}(\d\x)
\end{align}
On rearranging the terms in \eqref{et11}, we finally obtain for $r=2$
\begin{align}\label{et7}
	&\upnu\int_{\H}e^{\upnu\|\x\|_{\H}^2}\|\A^{\delta}\x\|_{\H}^{2m}\left(\frac{\mu}{2}\|\x\|_{\V}^2+\alpha\|\x\|_{\H}^2+2\beta\|\x\|_{\wi\L^{r+1}}^{r+1}\right)\eta_{\e}(\d\x)+\nonumber\\&\quad+ \left(\frac{m\mu}{2}-\frac{m^2}{\mu\upnu}(C_a^2+\beta^2)\right)\int_{\H}e^{\upnu\|\x\|_{\H}^2}\|\A^{\delta}\x\|_{\H}^{2m-2} \|\A^{\delta+\frac{1}{2}}\x\|_{\H}^2\eta_{\e}(\d\x)\nonumber\\&\quad+ \left(\frac{m\alpha}{2}-\left[m+\upnu+m^2 c(\kappa_2)\right]\sigma(\Q,\delta)
	\right)\int_{\H}e^{\upnu\|\x\|_{\H}^2}\|\A^{\delta}\x\|_{\H}^{2m}
	\eta_{\e}(\d\x)\nonumber\\&\leq 
	m\upnu^2\sigma^2(\Q,\delta)
	\left(\frac{1}{\mu}+\frac{1}{\alpha}\right)c(\kappa)
		\int_{\H}e^{\upnu\|\x\|_{\H}^2}\|\x\|_{\H}^{2m}\eta_{\e}(\d\x)\nonumber\\&\quad+\frac{\upnu^2\sigma(\Q,\delta)}{\uplambda_1^{2\delta-1}}c(\kappa_3)
\int_{\H}e^{\upnu\|\x\|_{\H}^2}\|\x\|_{\H}^{2m}\|\x\|_{\V}^2\eta_{\e}(\d\x)\nonumber\\&\quad+
	\sigma(\Q,\delta)\bigg(m+m^2 c(\kappa_1)+m^2 c(\kappa_2)\bigg)\int_{\H} e^{\upnu\|\x\|_{\H}^2}\eta_{\e}(\d\x).
\end{align}

Proceeding in a similar fashion and using \eqref{445}, we deduce the following for $r=3$:
\begin{align}\label{et8}
&\upnu\int_{\H}e^{\upnu\|\x\|_{\H}^2}\|\A^{\delta}\x\|_{\H}^{2m}
	\left(\frac{\mu}{2}\|\x\|_{\V}^2+\alpha\|\x\|_{\H}^2+\beta\|\x\|_{\wi\L^{r+1}}^{r+1}\right)
	\eta_{\e}(\d\x)
	\nonumber\\&\quad+ \left(\frac{m\mu}{2}-\frac{C_a^2 m^2}{\mu\upnu}(1+2\beta\mu)\right) \int_{\H}e^{\upnu\|\x\|_{\H}^2}\|\A^{\delta}\x\|_{\H}^{2m-2} \|\A^{\delta+\frac{1}{2}}\x\|_{\H}^2\eta_{\e}(\d\x)\nonumber\\ 
	&\quad+\left(\frac{m\alpha}{2}-\left[m+\upnu+m^2 c(\kappa_2)\right]\sigma(\Q,\delta)
		\right)\int_{\H}e^{\upnu\|\x\|_{\H}^2}\|\A^{\delta}\x\|_{\H}^{2m}
	\eta_{\e}(\d\x)\nonumber\\&\leq
	m\upnu^2\sigma^2(\Q,\delta)
	\left(\frac{1}{\mu}+\frac{1}{\alpha}\right)c(\kappa)
		\int_{\H}e^{\upnu\|\x\|_{\H}^2}\|\x\|_{\H}^{2m}\eta_{\e}(\d\x)\nonumber\\&\quad+\upnu^2\sigma(\Q,\delta)c(\kappa_3)C_e
	\int_{\H}e^{\upnu\|\x\|_{\H}^2}\|\x\|_{\H}^{2m}\|\x\|_{\V}^2\eta_{\e}(\d\x)\nonumber\\&\quad+
	\sigma(\Q,\delta)\bigg(m+m^2 c(\kappa_1)+m^2 c(\kappa_2)\bigg)\int_{\H} e^{\upnu\|\x\|_{\H}^2}\eta_{\e}(\d\x).
\end{align}
From Lemma \ref{lem4.1}, we infer that the right-hand quantity is bounded. Moreover, by using \eqref{fracA11}, we write 
\begin{align*}
	\|\A^{\delta}\x\|_{\H}^{2m}=\|\A^{\delta}\x\|_{\H}^{2m-2}\|\A^{\delta}\x\|_{\H}^2\leq\frac{1}{\uplambda_1}\|\A^{\delta}\x\|_{\H}^{2m-2} \|\A^{\delta+\frac{1}{2}}\x\|_{\H}^2.
\end{align*}
 Therefore, from lemma \ref{lem4.1}, it follows that \eqref{et7}-\eqref{et8} holds provided
\begin{itemize}
	\item[\emph{(i)}] $\mu+\alpha>\frac{2C_a^2 m}{\mu\upnu} (1+2\beta\mu)+\frac{2}{\uplambda_1}\left[1+\upnu+m c(\kappa_2)\right]\sigma(\Q,\delta)$, 
	\item[\emph{(ii)}] $\mu+\alpha>\frac{2m}{\mu\upnu}(C_a^2+\beta^2)+\frac{2}{\uplambda_1}\left[1+\upnu+m c(\kappa_2)\right]\sigma(\Q,\delta)$,
	\item[\emph{(iii)}] $0<\upnu\leq\frac{2\alpha+\mu\upnu_1}{2\Tr(\Q)}$.
\end{itemize}
From conditions $(i)$ and $(ii)$, we get following lower bounds on $\upnu$:
\begin{align}\label{UBL}
	\upnu>\frac{2C_a^2 m}{\mu(\mu+\alpha)}(1+2\beta\mu), \  \upnu>\frac{2m}{\mu(\mu+\alpha)}(C_a^2+\beta^2).
\end{align}
Furthermore, we get following upper bound for $\upnu$:
\begin{align*}
	\upnu<\frac{(\mu+\alpha)\uplambda_1}{4 \sigma(\Q,\delta)}.
\end{align*}
Finally, we conclude that \eqref{et7}-\eqref{et8} hold if  
\begin{align*}
  \sup\left\{\frac{2C_a^2 m}{\mu(\mu+\alpha)}(1+2\beta\mu), \frac{2m}{\mu(\mu+\alpha)}(C_a^2+\beta^2)\right\}
  <\upnu<\inf\left\{\frac{(\mu+\alpha)\uplambda_1}{4 \sigma(\Q,\delta)} ,\frac{2\alpha+\mu\uplambda_1}{2\Tr(\Q)}\right\},
\end{align*}
     where $\mu$ and $\alpha$ satisfy the condition  \eqref{439}.
\end{proof}
Let us now take $\delta\in(\frac{1}{4},\frac{1}{2})$. Then, we have the following result: 
\begin{lemma}\label{lem5.2}
	Assume that $$\Tr(\A\Q)<+\infty.$$ Then, we have 
	\begin{align}\label{4p44}
		\int_{\H}\|\A\x\|_{\H}^2\eta(\d\x)\leq C,
	\end{align}
where the constant $C=\mathrm{const}\{\Tr(\Q),m,\beta,\mu\}$.
\end{lemma}
	
\begin{remark}
Note that for $\delta<\frac{1}{2}$, we have  $\D(\A^{\frac{1}{2}})\hookrightarrow\D(\A^{\delta})$. Thus, $\Tr(\A\Q)<\infty$ implies $\Tr(\A^{2\delta}\Q)<\infty$ and therefore, we can use the result of Lemma \ref{lem4.5}. 
\end{remark}
\begin{proof}[Proof of Lemma \ref{lem5.2}]
	Let us set $$\varphi(\x)=\|\A^{\frac{1}{2}}\x\|_{\H}^2\ \text{ for all }\ \x\in\D(\A^{\frac{1}{2}}). $$ Then, it is immediate that 
		\begin{align*}
			\D_{\x}\varphi(\x)=2\A\x,\ \text{ and }\ \D_{\x\x}\varphi(\x)=2\A. 
			\end{align*}
			Therefore, we deduce 
			\begin{align*}
			&\frac{1}{2}\Tr[\Q\D_{\x}^2\psi(\x)]-(\mu\A\x+\alpha\x+\B(\x)+\beta\mathcal{C}(\x),\D_{\x}\psi(\x))\nonumber\\&=\Tr(\Q\A)-2\mu\|\A\x\|_{\H}^2-2\alpha\|\A^{\frac{1}{2}}\x\|_{\H}^2-2(\B(\x),\A\x)-2\beta(\mathcal{C}(\x),\A\x)\nonumber\\&\leq
			  \Tr(\Q\A)-\mu\|\A\x\|_{\H}^2-2\alpha\|\A^{\frac{1}{2}}\x\|_{\H}^2+\frac{2}{\mu}\|\x\|_{\wi\L^{\infty}}^2\|\x\|_{\V}^2+\frac{2\beta^2}{\mu}\|\x\|_{\wi\L^{2r}}^{2r}. 
			\end{align*}
			Using the fact that $$\int_{\H}\mathcal{N}_2\varphi(\x)\eta(\d\x)=0,$$ we deduce from the above expression that 
			\begin{align}\label{4p45}
				&\mu\int_{\H}\|\A\x\|_{\H}^2\eta(\d\x)+2\alpha\int_{\H}\|\A^{\frac{1}{2}}\x\|_{\H}^2\d\eta(\x)\nonumber\\&\leq\Tr(\A\Q)+\frac{2}{\mu}\int_{\H}\|\x\|_{\wi\L^{\infty}}^2\|\x\|_{\V}^2\d\eta(\x)+\frac{2\beta^2}{\mu}\int_{\H}\|\x\|_{\wi\L^{2r}}^{2r}\d\eta(\x). 
			\end{align}
			Using the estimates \eqref{435}-\eqref{437}, we estimate using Lemma \ref{lem4.5} as 
			\begin{align}\label{wtm1}
			\int_{\H}\|\x\|_{\wi\L^{\infty}}^2\|\x\|_{\V}^2\eta(\d\x)&\leq C\int_{\H} \|\A^{\delta}\x\|_{\H}^{8\delta}\|\A^{\delta+\frac{1}{2}}\x\|_{\H}^{4-8\delta}\eta(\d\x)
			\nonumber\\&\leq C\left(\int_{\H}\|\A^{\delta}\x\|_{\H}^{\frac{8\delta}{2-4\delta}}
			\|\A^{\delta+\frac{1}{2}}\x\|_{\H}^2\eta(\d\x)\right)^{2-4\delta}\leq C,
			\end{align}
			by choosing $m=\frac{1}{1-2\delta}$. Since $\frac{1}{4}<\delta<\frac{1}{2}$, we get $2<m<\infty$. For $r=2$, Ladyzhenskaya's and H\"older's inequalities yield
			\begin{align}\label{wtm2}
			\int_{\H}\|\x\|_{\wi\L^{4}}^4\eta(\d\x)&\leq 2\int_{\H}\|\x\|_{\H}^2\|\x\|_{\V}^2\eta(\d\x)
			\nonumber\\&\leq 2\left(\int_{\H}\|\x\|_{\H}^4\eta(\d\x)\right)^{\frac{1}{2}} \left(\int_{\H}\|\x\|_{\V}^4\eta(\d\x)\right)^{\frac{1}{2}}\nonumber\\&\leq C
			\left(\int_{\H}\|\x\|_{\H}^4\eta(\d\x)\right)^{\frac{1}{2}} \left(\int_{\H} \|\A^{\delta}\x\|_{\H}^{8\delta}\|\A^{\delta+\frac{1}{2}}\x\|_{\H}^{4-8\delta}\eta(\d\x)\right)
			^{\frac{1}{2}}\nonumber\\&\leq C\left(\int_{\H}\|\x\|_{\H}^4\eta(\d\x)\right)^{\frac{1}{2}} \left(\int_{\H}\|\A^{\delta}\x\|_{\H}^{\frac{8\delta}{2-4\delta}}
			\|\A^{\delta+\frac{1}{2}}\x\|_{\H}^2\eta(\d\x)\right)^{1-2\delta}\leq C,
			\end{align}
		where we have used \eqref{4p7} and \eqref{wtm1}. Similarly, for $r=3$, an application of the Gagliardo-Nirenberg inequality, the estimates \eqref{436} and Lemma \ref{lem4.5} provide us 
			\begin{align}\label{wtm3}
			\int_{\H}\|\x\|_{\wi\L^{6}}^6\eta(\d\x)&\leq C\int_{\H}\|\x\|_{\V}^4\|\x\|_{\H}^2 \eta(\d\x) \nonumber\\&\leq C\int_{\H}\|\A^{\delta}\x\|_{\H}^{8\delta} \|\A^{\delta+\frac{1}{2}}\x\|_{\H}^{4-8\delta}\|\x\|_{\H}^2\eta(\d\x)\nonumber\\&\leq \left(\int_{\H}\|\x\|_{\H}^{\frac{4}{4\delta-1}}\eta(\d\x)\right)^{4\delta-1}\left(\int_{\H}\|\A^{\delta}\x\|_{\H}^{\frac{8\delta}{2-4\delta}}\|\A^{\delta+\frac{1}{2}}\x\|_{\H}^2\eta(\d\x)\right)^{2-4\delta}\leq C,
			\end{align}
			by using \eqref{4p7} and \eqref{wtm1}. Using the above estimates \eqref{wtm1}-\eqref{wtm3} in \eqref{4p45}, we finally obtain \eqref{4p44}. 
\end{proof}

\section{Eseential $m$-dissipativity of the Kolmogorov operator: Main result}\label{disscore}\setcounter{equation}{0}

We say that a linear operator $\mathscr{A}:\D(\mathscr{A})\subset\mathcal{H}\to\mathcal{H}$ in a Hilbert space $\mathcal{H}$  is \emph{dissipative} if
\begin{align}\label{dissp}
	\|\varphi\|_{\mathcal{H}}\leq \frac{1}{\lambda}\|\lambda\varphi-\mathscr{A}\varphi\|_{\mathcal{H}}\ \text{ for all }\ \varphi\in\D(\mathscr{A}),\ \lambda>0. 
\end{align}
Any dissipative operator is \emph{closable}, that is, it has a closed extension \cite[Chapter 3, Lemma 3.4.4]{ware}. The dissipative operator $\mathscr{A}$  is called \emph{$m$-dissipative} if the range of $\lambda\I-\mathscr{A}$  coincides with $\mathcal{H}$ for some (and consequently for any) $\lambda>0$. An operator $\mathscr{A}$ with dense domain is $m$-dissipative if and only if it is the infinitesimal generator of a strongly continuous semigroup of contractions in $\mathcal{H}$. 

\begin{theorem}[Lumer-Phillips, {\cite[Theorem 3.20]{gdp1}}]
Let $\mathscr{A}:\D(\mathscr{A})\subset\mathcal{H}\to\mathcal{H}$  be a dissipative operator in the Hilbert space $\mathcal{H}$ such that $\D(\mathscr{A})$ is dense in $\mathcal{H}$.  Assume that for some $\lambda>0$,  the range of $\lambda\I-\mathscr{A}$ is dense in $\mathcal{H}$. Then the closure of $\mathscr{A}$ is $m$-dissipative. 
\end{theorem}

\begin{theorem}\label{thm4.7}
	Assume that the condition \eqref{439} holds true and that $\Tr(\A\Q)<+\infty$. Then $\mathcal{N}_0$ is dissipative in $\L^2(\H;\eta)$ and its closure $\overline{\mathcal{N}_0}$ in $\L^2(\H;\eta)$ coincides with the infinitesimal generator $\mathcal{N}_2$ of $\{\P_t\}_{t\geq 0}$ in $\L^2(\H;\eta)$. 
\end{theorem}
\begin{proof}
	The proof of the theorem is divided into the following steps:
	\vskip 0.1 cm
	\noindent
	\textbf{Step 1:}
	Applying the infinite-dimensional  It\^o formula to the process $\varphi(\X(t,\x)=\varphi(\X(t)),$ we	have
	\begin{align}\label{450}
		\d\varphi(\X(t))&=\langle \D_{\x}\varphi(\X(t)),\d\X(t)\rangle+\frac{1}{2}\Tr(\Q\D^2_{\x}\varphi(\x)).
	\end{align}
	Consequntly, 
	\begin{align}\label{451}
		\P_t\varphi(\x)&=\E\left[\varphi(\X(t,\x))\right]\nonumber\\&=\varphi(\x)-\E\left[\int_0^t\langle[\mu \A\X(s)+\alpha\X(s)+\B(\X(s))+\beta\mathcal{C}(\X(s))],\D_{\x}\varphi(\X(s))\rangle\d s\right]\nonumber\\&\quad+\frac{1}{2}\E\left[\int_0^t\Tr(\Q\D^2_{\x}\varphi(\X(s)))\d s\right]\nonumber\\&=\varphi(\x)+\E\left[\int_0^t\mathcal{N}_0\varphi(\X(s,\x))\d s\right]. 
	\end{align}
	
	\vskip 0.1 cm
	\noindent
	\textbf{Step 2:} \emph{$\mathcal{N}_2$ is an extension of $\mathcal{N}_0$.} 
	We start by proving that $\mathcal{N}_2\varphi=\mathcal{N}_0\varphi$ for all $\varphi\in\mathscr{E}_{\A}(\H)$. In fact, for any $\varphi\in\mathscr{E}_{\A}(\H)$, we have from \eqref{451}
	\begin{align}\label{448}
		\E\left[\varphi(\X(t,\x))\right]&=\varphi(\x)+\E\left[\int_0^t\mathcal{N}_0\varphi(\X(s,\x))\d s\right],\ t\geq 0,\ \x\in\H,
	\end{align}
where $\mathcal{N}_0$ is defined in \eqref{4p5}. 
Since $\mathcal{N}_2$  is the	infinitesimal generator of $\P_t$, for all $\varphi\in\mathscr{E}_{\A}(\H)$, we infer
\begin{align}
	\mathcal{N}_2\varphi(\x)=\lim\limits_{t\to 0}\frac{1}{t}\left(\P_t\varphi(\x)-\varphi(\x)\right)=\lim\limits_{t\to 0}\frac{1}{t}\E\left[\int_0^t \mathcal{N}_0\varphi(\X(s,\x))\d s\right]= \mathcal{N}_0\varphi(\x),\ \x\in\H,
\end{align}
pointwise. We will now demonstrate that $\varphi\in\D(\mathcal{N}_2)$. It is enough to show that
\begin{align}
	\frac{1}{t}(\mathrm{P}_t\varphi-\varphi)
\end{align}
is equibounded in $\L^2(\H;\eta)$ for $t\in(0,1]$. It is clear that for any $\varphi\in\mathscr{E}_{\A}(\H)$ there exists constants $C_1$ and $C_2$ depending on $\varphi$ such that 
\begin{align}\label{kolm}
	\left|\frac{1}{2}\Tr\left[\Q\D_{\x}^2\varphi(\x)\right]-(\alpha\x,\D_{\x}\varphi(\x))\right|\leq C_1+C_2\|\x\|_{\H}, \ \x\in\H,
\end{align}
see, for instance, \cite[Chapter 2, pp. 38]{gdp1}. In view of \eqref{451} and \eqref{kolm}, we have in fact 
\begin{align}
&	|\mathrm{P}_t\varphi(\x)-\varphi(\x)|\nonumber\\&=\left|\E\left[\int_0^t \mathcal{N}_0\varphi(\X(s,\x))\d s\right]\right|\nonumber\\&\leq\int_0^t\E\left[C_1+C_2\|\X(s,\x)\|_{\H}+\|\varphi\|_1\left(\mu\|\A\X(s,\x)\|_{\H}+\|\B(\X(s,\x))\|_{\H}+\beta\|\mathcal{C}(\X(s,\x))\|_{\H}\right)\right]\d s,
\end{align}
for any $\x\in\H$. An application of H\"older's inequality yields 
{\small\begin{align}
&	|\mathrm{P}_t\varphi(\x)-\varphi(\x)|^2\nonumber\\&\leq t\int_0^t\left\{\E\left[C_1+C_2\|\X(s,\x)\|_{\H}+\|\varphi\|_1\left(\mu\|\A\X(s,\x)\|_{\H}+\|\B(\X(s,\x))\|_{\H}+\beta\|\mathcal{C}(\X(s,\x))\|_{\H}\right)\right]\right\}^2\d s \nonumber\\&\leq  Ct\int_0^t\left[C_1^2+C_2^2\E\|\X(s,\x)\|_{\H}^2+\|\varphi\|_{1}^2\left(\mu^2\E\|\A\X(s,\x)\|_{\H}^2+\E\|\B(\X(s,\x))\|_{\H}^2+\beta^2\E\|\mathcal{C}(\X(s,\x))\|_{\H}^2\right)\right]\d s\nonumber\\&=  Ct\int_0^t\left[C_1^2+C_2^2\mathrm{P}_s(\|\cdot\|_{\H}^2(\x))+\|\varphi\|_{1}^2\left(\mu^2\mathrm{P}_s(\|\A\cdot\|_{\H}^2(\x))+\mathrm{P}_s(\|\B(\cdot)\|_{\H}^2(\x))+\beta^2\mathrm{P}_s(\|\mathcal{C}(\cdot)\|_{\H}^2(\x))\right)\right]\d s.
\end{align}}
Integrating with respect to $\eta$ over $\H$ and taking into account the invariance of $\eta$, we derive 
\begin{align}\label{traq}
\|\mathrm{P}_t\varphi-\varphi\|_{\L^2(\H;\eta)}^2\leq Ct^2\int_{\H}\left[C_1^2+C_2^2\|\x\|_{\H}^2+\mu^2\|\A\x\|_{\H}^2+\|\B(\x)\|_{\H}^2+\beta^2\|\mathcal{C}(\x)\|_{\H}^2\right]\d\eta<+\infty, 
\end{align}
by using \eqref{4p44}. Therefore, from Lemma \ref{lem5.2}, it is clear that $\mathcal{N}_2\varphi\in\L^2(\H;\eta)$, and  by \eqref{448} and the dominated convergence theorem, one can deduce that 
\begin{align}
	\mathcal{N}_2\varphi=\lim\limits_{t\to 0}\frac{1}{t}\left(\P_t\varphi-\varphi\right)=\mathcal{N}_0\varphi \ \text{ exists in }\ \L^2(\H;\eta). 
\end{align}
	Therefore, for any  $\varphi\in\mathscr{E}_{\A}(\H)$, we have $\varphi\in\D(\mathcal{N}_2)$ and $	\mathcal{N}_2\varphi=\mathcal{N}_0\varphi ,$ hence $\mathcal{N}_2$ extends $\mathcal{N}_0$. 
	\vskip 0.1 cm
	\noindent
	\textbf{Step 3:} \emph{$\mathcal{N}_0$ is closable.} 
    Since $\mathcal{N}_2$ is dissipative (see \eqref{453} below), therefore, by definition \eqref{dissp}, we have
	\begin{align*}
		\|\varphi\|_{\H}\leq\frac{1}{\lambda}\|\lambda\varphi-\mathcal{N}_2\varphi\|_{\H} \ \text{ for all }\  \varphi\in\D(\mathcal{N}_2).
	\end{align*}
Using the fact that $\mathcal{N}_2$ is an extension $\mathcal{N}_0$, the above inequality implies that
\begin{align*}
		\|\varphi\|_{\H}\leq\frac{1}{\lambda}\|\lambda\varphi-\mathcal{N}_0\varphi\|_{\H}\  \text{ for all }\  \varphi\in\mathscr{E}_{\A}(\H),
\end{align*}
	and hence $\mathcal{N}_0$ is dissipative. Moreover, since every disspative operator is closable, so is $\mathcal{N}_0$. 
	\vskip 0.1 cm
	\noindent
	\textbf{Step 4:} \emph{Closure of $\mathcal{N}_0$ coincides with $\mathcal{N}_2$.} 
    Since $\mathcal{N}_0$ is closable, therefore from \cite[Chapter 1, Lemma 1.8]{EBD}, $\mathcal{N}_0$ has a \emph{least closed extension} called its \emph{closure}. Let us denote it by $\overline{\mathcal{N}_0}$.
    Our main aim is to show that $\overline{\mathcal{N}_0}=\mathcal{N}_2$. In order to do this, we consider the approximating system \eqref{4.33} and the invariant measure $\eta$  for $\P_t$.

    Let $f\in\C_b^1(\H)$ and write
	\begin{align*}
	(\mathcal{N}_{\e}\varphi)(\x)&=\frac{1}{2}\Tr\left[\Q\D_{\x}^2\varphi(\x)\right]-(\mu\A\x+\alpha\x+\B_{\e}(\x)+\beta\mathcal{C}_{\e}(\x),\D_{\x}\varphi(\x))\nonumber\\&=(\mathcal{N}_0\varphi)(\x)-(\B_{\e}(\x)-\B(\x),\D_{\x}\varphi(\x))-\beta(\mathcal{C}_{\e}(\x)-\mathcal{C}(\x),\D_{\x}\varphi(\x)),
\end{align*}
for $\varphi\in\mathscr{E}_{\A}(\H).$  Since, both $\B_{\e}(\cdot)$ and $\mathcal{C}_\e(\cdot)$ are bounded and regular, there exists a unique solution  $\varphi_{\e}\in\D(\mathcal{N}_{\e})\cap\C_b^1(\H)$ of the approximating equation
	\begin{align*}
		\lambda\varphi_{\e}-\mathcal{N}_{\e}\varphi_{\e}=f 
	\end{align*}
	given by the resolvent formula (\cite[Theorem 3.21, Step 1]{gdp1}), 
	\begin{align*}
		\varphi_{\e}(\x)=\int_0^{\infty}e^{-\lambda t}\E\left[f(\X_{\e}(t,\x))\right]\d t,
	\end{align*}
where $\X_\e(\cdot)$ is a solution to the problem \eqref{4.33}. Therefore, $\varphi_{\e}$ is differentiable in the direction of $\boldsymbol{h}\in\H$ and it is immediate that 
	\begin{align*}
		(\D_{\x}\varphi_{\e}(\x),\boldsymbol{h})=\int_0^{\infty}e^{-\lambda t}\E\left[(\D_{\x}f(\X_{\e}(t,\x)),\boldsymbol{\xi}_{\e}^{\boldsymbol{h}}(t,\x))\right]\d t,
	\end{align*}
	where $\boldsymbol{\xi}_{\e}^h(t,\x)=\D_{\x}\X_{\e}(t,\x)\boldsymbol{h}$. By Lemma \ref{lem4.4} (see \eqref{4p20}), it follows that there exist $\kappa,\delta>0$  such that
	\begin{align*}
		|(\D_{\x}\varphi_{\e}(\x),\boldsymbol{h})|&\leq \int_0^{\infty}e^{-\lambda t}\left\{\E\left[\|\D_{\x}f(\X_{\e}(t,\x))\|_{\H}^2\right]\right\}^{1/2}\E\left\{\left[\|\boldsymbol{\xi}_{\e}^{\boldsymbol{h}}(t,\x)\|_{\H}^2\right]\right\}^{1/2}\d t\nonumber\\&\leq \frac{1}{\lambda}\|\boldsymbol{h}\|_{\H}e^{\kappa\|\x\|_{\H}^2}e^{-\delta t}\|f\|_1, \ \text{ for all }\  \boldsymbol{h},\x\in\H. 
	\end{align*}
	where we have used the Cauchy-Schwarz and H\"older's inequalities. Using the arbitrariness of $\boldsymbol{h}\in\H$, we also have 
	\begin{align}\label{bism}
		\|\D_{\x}\varphi_{\e}(\x)\|_{\H}\leq \frac{1}{\lambda}e^{\kappa\|\x\|_{\H}^2}e^{-\delta t}\|f\|_1, \  \text{ for all }\ \x\in\H. 
	\end{align}
	Then arguing similarly as in  \cite[Claim 1, pg. 171]{gdp1}, we have $\varphi_{\e}\in\D(\overline{\mathcal{N}_0})$ and 
	\begin{align}\label{bece}
		\lambda\varphi_{\e}-\overline{\mathcal{N}_0}\varphi_{\e}=(\B_{\e}(\x)-\B(\x),\D_{\x}\varphi_{\e}(\x))+
		\beta(\mathcal{C}_{\e}(\x)-\mathcal{C}(\x),\D_{\x}\varphi_{\e}(\x))+f. 
	\end{align}
	Now, our aim is to show that 
	\begin{align}\label{457}
		\lim\limits_{\e\to 0}(\B_{\e}(\x)-\B(\x),\D_{\x}\varphi_{\e}(\x))=0\ \text{ in }\ \L^2(\H;\eta)
	\end{align}
and 
\begin{align}\label{4p57}
	\lim\limits_{\e\to 0}(\mathcal{C}_{\e}(\x)-\mathcal{C}(\x),\D_{\x}\varphi_{\e}(\x))=0\ \text{ in }\ \L^2(\H;\eta). 
\end{align}
 By using the Cauchy-Schwarz inequality, \eqref{modiB} and \eqref{bism}, we note that
	\begin{align*}
	&\int_{\H}\left|(\B_{\e}(\x)-\B(\x),\D_{\x}\varphi_{\e}(\x))\right|^2\eta(\d\x)\nonumber\\&=\int_{\{\|\x\|_{\V}\geq\e^{-1}\}}\left|(\B_{\e}(\x)-\B(\x),\D_{\x}\varphi_{\e}(\x))\right|^2\eta(\d\x)\nonumber\\&\leq\int_{\{\|\x\|_{\V}\geq\e^{-1}\}}\left|\frac{1-\e^2\|\x\|_{\V}^2}{\e^2\|\x\|_{\V}^2}\right|\|\B(\x)\|_{\H}^2\|\D_{\x}\varphi_{\e}(\x)\|_{\H}^2\eta(\d\x)\nonumber\\&\leq C e^{-\delta t}\|f\|_1^2\int_{\{\|\x\|_{\V}\geq\e^{-1}\}}e^{\kappa\|\x\|_{\H}^2}\|\B(\x)\|_{\H}^2\eta(\d\x),
	\end{align*}
	so that \eqref{457} follows if 
	\begin{align}\label{4p59}
		\lim\limits_{\e\to 0}\int_{\{\|\x\|_{\V}\geq\e^{-1}\}}e^{\kappa\|\x\|_{\H}^2}\|\B(\x)\|_{\H}^2\eta(\d\x)=0. 
	\end{align}
	Using the estimates \eqref{435} and \eqref{436}, we find 
	\begin{align}\label{beb1}
		\|\B(\x)\|_{\H}\leq\|\x\|_{\wi\L^{\infty}}\|\x\|_{\V}\leq C\|\A^{\delta}\x\|_{\H}^{4\delta}\|\A^{\delta+\frac{1}{2}}\x\|_{\H}^{2(1-2\delta)},
	\end{align}
	for all $\x\in\D(\A^{\delta+\frac{1}{2}})$. 
 Along with Lemma \ref{lem4.5} and the estimate \eqref{fracA}, the expression \eqref{beb1} yields
\begin{align}\label{beb2}
		&\int_{\{\|\x\|_{\V}\geq\e^{-1}\}}e^{\kappa\|\x\|_{\H}^2}\|\B(\x)\|_{\H}^2\eta(\d\x)\nonumber\\&\leq
	C\int_{\{\|\x\|_{\V}\geq\e^{-1}\}}e^{\kappa\|\x\|_{\H}^2}\|\A^{\delta}\x\|_{\H}^{8\delta}\|\A^{\delta+\frac{1}{2}}\x\|_{\H}^{4(1-2\delta)}\eta(\d\x)\nonumber\\&
	\nonumber\\&\leq C\uplambda_1^{\delta(2-8\delta)}\int_{\{\|\x\|_{\V}\geq\e^{-1}\}}e^{\kappa\|\x\|_{\H}^2}\|\A^{\delta}\x\|_{\H}^{8\delta}\frac{\|\A^{\delta+\frac{1}{2}}\x\|_{\H}^2}{\|\x\|_{\V}^{8\delta-2}}\eta(\d\x)
	\nonumber\\&\leq C\uplambda_1^{\delta(2-8\delta)}\e^{8\delta-2}\int_{\H}e^{\kappa\|\x\|_{\H}^2}\|\A^{\delta}\x\|_{\H}^{8\delta}\|\A^{\delta+\frac{1}{2}}\x\|_{\H}^{2}\eta(\d\x)\nonumber\\&\leq C\e^{8\delta-2},
\end{align}
provided $\frac{1}{4}<\delta<\frac{1}{2}$, and thus \eqref{4p59} follows by an application of the Dominated convergence theorem. Let us now calculate by using \eqref{modiC} and \eqref{bism}
	\begin{align*}
&	\int_{\H}\left|(\mathcal{C}_{\e}(\x)-\mathcal{C}(\x),\D_{\x}\varphi_{\e}(\x))\right|^2\eta(\d\x)\nonumber\\&=\int_{\{\|\x\|_{\V}\geq\e^{-1}\}}\left|(\mathcal{C}_{\e}(\x)-\mathcal{C}(\x),\D_{\x}\varphi_{\e}(\x))\right|^2\eta(\d\x)\nonumber\\&\leq\int_{\{\|\x\|_{\V}\geq\e^{-1}\}}\left|\frac{1-\e^{r+1}\|\x\|_{\V}^{r+1}}{\e^{r+1}\|\x\|_{\V}^{r+1}}\right|\|\mathcal{C}(\x)\|_{\H}^2\|\D_{\x}\varphi_{\e}(\x)\|_{\H}^2\eta(\d\x)\nonumber\\&\leq C e^{-\delta t}\|f\|_1^2\int_{\{\|\x\|_{\V}\geq\e^{-1}\}} e^{\kappa\|\x\|_{\H}^2}\|\mathcal{C}(\x)\|_{\H}^2\eta(\d\x),
\end{align*}
so that \eqref{4p57} follows if 
\begin{align}\label{4pp57}
	\lim\limits_{\e\to 0}\int_{\{\|\x\|_{\V}\geq\e^{-1}\}}e^{\kappa\|\x\|_{\H}^2} \|\mathcal{C}(\x)\|_{\H}^2 \eta(\d\x)=0. 
\end{align}
We calculate for $r=2$ by using \eqref{435}, Ladyzhenskaya's and Poincar\'e's inequlaities
\begin{align*}
	\|\mathcal{C}(\x)\|_{\H}\leq\|\x\|_{\wi\L^4}^2\leq\sqrt{2}\|\x\|_{\H}\|\x\|_{\V}\leq \frac{\sqrt{2}}{\uplambda_1}\|\x\|_{\V}^2\leq\frac{\sqrt{2}}{\uplambda_1}\|\A^{\delta}\x\|_{\H}^{4\delta}\|\A^{\delta+\frac{1}{2}}\x\|_{\H}^{2(1-2\delta)}.
\end{align*}
Then proceeding in a similar way as we did in \eqref{beb1}, one can conclude \eqref{4pp57}  provided $\frac{1}{4}<\delta<\frac{1}{2}$.
Along with Lemma \ref{lem4.5}, estimates \eqref{wtm3}, \eqref{beb2} and \eqref{fracA1}, it follows  for $r=3$ that
\begin{align*}
	&\int_{\{\|\x\|_{\V}\geq\e^{-1}\}}e^{\kappa\|\x\|_{\H}^2}\|\mathcal{C}(\x)\|_{\H}^2\eta(\d\x)\nonumber\\&\leq
	C\int_{\{\|\x\|_{\V}\geq\e^{-1}\}}e^{\kappa\|\x\|_{\H}^2}\|\x\|_{\V}^4\|\x\|_{\H}^2\eta(\d\x)\nonumber\\&\leq \frac{C}{\uplambda_1^{2\delta}} \int_{\{\|\x\|_{\V}\geq\e^{-1}\}}e^{\kappa\|\x\|_{\H}^2}\|\A^{\delta}\x\|_{\H}^{8\delta} \|\A^{\delta+\frac{1}{2}}\x\|_{\H}^{4-8\delta}\|\A^{\delta}\x\|_{\H}^2\eta(\d\x)\nonumber\\&\leq\frac{C}{\uplambda_1^{2\delta}}
	\int_{\{\|\x\|_{\V}\geq\e^{-1}\}}e^{\kappa\|\x\|_{\H}^2}\|\A^{\delta}\x\|_{\H}^{8\delta+2}\frac{\|\A^{\delta+\frac{1}{2}}\x\|_{\H}^2}{\|\A^{\delta+\frac{1}{2}}\x\|_{\H}^{8\delta-2}}\eta(\d\x)\nonumber\\&\leq C\uplambda_1^{\delta(2-8\delta)}\int_{\{\|\x\|_{\V}\geq\e^{-1}\}}e^{\kappa\|\x\|_{\H}^2}\|\A^{\delta}\x\|_{\H}^{8\delta+2}\frac{\|\A^{\delta+\frac{1}{2}}\x\|_{\H}^2}{\|\x\|_{\V}^{8\delta-2}}\eta(\d\x)
	\nonumber\\&\leq C\uplambda_1^{\delta(2-8\delta)}\e^{8\delta-2}\int_{\H}e^{\kappa\|\x\|_{\H}^2}\|\A^{\delta}\x\|_{\H}^{8\delta+2}\|\A^{\delta+\frac{1}{2}}\x\|_{\H}^{2}\eta(\d\x)\nonumber\\&\leq C\e^{8\delta-2},
\end{align*}
provided $\frac{1}{4}<\delta<\frac{1}{2}$. Therefore, finally, from \eqref{bece}, we have 
	\begin{align*}
		\lambda\varphi_{\e}-\overline{\mathcal{N}_0}\varphi_{\e}\to\f \ \text{ as }\ \e\to 0\ \text{ in }\ \L^2(\H;\eta).
	\end{align*}
	Hence we deduce that the closure of the range of $\lambda-\overline{\mathcal{N}_0}$ is dense in $\L^2(\H;\eta)$ and so, in view of the Lumer-Phillips Theorem (see \cite[Theorem 3.20]{gdp1}), $\overline{\mathcal{N}_0}$ is $m$-dissipative. Since $\mathcal{N}_2$ is $m$-dissipative (as it is the infinitesimal generator of a strongly continuous semigroup of contractions) and it extends $\mathcal{N}_0$ also, thus $\mathcal{N}_2$ must coincides with $\overline{\mathcal{N}_0}$ as claimed. 
\end{proof}

\subsection{The ``Carre du Champ's" identity} Let us now address some of the implications of the main result Theorem \ref{thm4.7}, particularly the integration by parts formula, also called \emph{identit\'e du carr\'e du champ}. The following identity is straightforward:
\begin{align*}
	\mathcal{N}_0(\varphi^2)=2\varphi \mathcal{N}_0\varphi+\|\sqrt{\Q}\D_{\x}\varphi\|_{\H}^2 \ \text{ for all } \varphi\in\mathscr{E}_{\A}(\H). 
\end{align*} 
By exploiting the invariance of $\eta$ and integrating the aforementioned identity with respect to $\eta$ over $\H$, we obtain
\begin{align}\label{carre}
		\int_{\H}\mathcal{N}_0\varphi(\x)\varphi(\x)\eta(\d \x)=-\frac{1}{2}\int_{\H}\|\sqrt{\Q}\D_{\x}\varphi(\x)\|_{\H}^2\eta(\d\x).
\end{align}
Therefore, $\mathcal{N}_0$ is dissipative, and hence it is closable in $\L^2(\H;\eta)$, since $\mathscr{E}_{\A}(\H)$ is dense in $\L^2(\H;\eta)$. Let us now discuss the infinitesimal generator $\mathcal{N}_2$ of the semigroup $\{\mathrm{P}_t\}_{t\geq 0}$. We endow the domain $\D(\mathcal{N}_2)$ of $\mathcal{N}_2$ with the following graph norm:
\begin{align}
	\|\varphi\|_{\D(\mathcal{N}_2)}^2=\|\varphi\|_{\L^2(\H;\eta)}^2
	+\|\mathcal{N}_2\varphi\|_{\L^2(\H;\eta)}^2, \ \varphi\in\D(\mathcal{N}_2). 
\end{align}
Since $\mathscr{E}_{\A}(\H)$ is a core of $\mathcal{N}_2$ (Theorem \ref{thm4.7}), one can extend the identity \eqref{carre} to $\D(\mathcal{N}_2)$, which is stated below. 
\begin{lemma}\label{carredu}
	The operator $\Q^{\frac{1}{2}}\D_{\x}\varphi$ defined in $\mathscr{E}_{\A}(\H)$, is uniquely extendible	to a linear bounded operator from $\D(\mathcal{N}_2)$ into $\mathbb{L}^2(\H,\eta;\H)$. The extension is still denoted by
	$\Q^{\frac{1}{2}}\D_{\x}\varphi$. Moreover, we have the following ``Carre du Champ's" identity:
	\begin{align}\label{453}
		\int_{\H}\mathcal{N}_2\varphi(\x)\varphi(\x)\eta(\d \x)=-\frac{1}{2}\int_{\H}\|\sqrt{\Q}\D_{\x}\varphi(\x)\|_{\H}^2\eta(\d\x)\ \text{ for all }\ \varphi\in\D(\mathcal{N}_2)
	\end{align}
	and 
	\begin{align}
		\|\Q^{\frac{1}{2}}\D_{\x}\varphi\|_{\mathbb{L}^2(\H,\eta;\H)}\leq \|\varphi\|_{\D(\mathcal{N}_2)}\ \text{ for all }\ \varphi\in\D(\mathcal{N}_2). 
	\end{align}
\end{lemma}
\begin{proof}
	See \cite[Proposition 4.1]{VBGD}. 
\end{proof}
Note that we are not able to prove that the operator $\Q^{\frac{1}{2}}\D_{\x}\varphi$  defined in $\mathscr{E}_{\A}(\H)$
is closable in $\mathbb{L}^2(\H,\eta)$.  Let us state prove  some estimates for the resolvent of $\mathcal{N}_2$.
\begin{lemma}\label{lem4.10}
	Let $f\in\mathbb{L}^2(\H,\eta)$, $\lambda>0$ and set $\varphi=(\lambda\I-\mathcal{N}_2)^{-1}f$. Then we have 
	\begin{align}
		\|\varphi\|_{\mathbb{L}^2(\H,\eta)}\leq \frac{1}{\lambda}\|f\|_{\mathbb{L}^2(\H,\eta)}\ \text{ and }\  \|\Q^{\frac{1}{2}}\D_{\x}\varphi\|_{\mathbb{L}^2(\H,\eta;\H)}\leq \sqrt{\frac{2}{\lambda}}\|f\|_{\mathbb{L}^2(\H,\eta)}.
	\end{align}
\end{lemma}
\begin{proof}
	See \cite[Proposition 4.2]{VBGD}. 
\end{proof}
The following is a result on some perturbations of $\mathcal{N}_2$. 
\begin{lemma}\label{pertubkol}
Let $F\in \mathscr{B}_b(\H;\H)$. Consider the linear operator 
\begin{align}
	\mathcal{N}_1\varphi=\mathcal{N}_2\varphi+\langle F(\x), \Q^{\frac{1}{2}}\D_{\x}\varphi\rangle,\ \varphi\in\D(\mathcal{N}_2). 
\end{align}	
Then the resolvent set $\rho(\mathcal{N}_1)$ of $\mathcal{N}_1$ contains the half-right  $(2\|F\|_0^2,\infty)$ and
\begin{align}
	(\lambda\I-\mathcal{N}_1)^{-1}f=(\lambda\I-\mathcal{N}_2)^{-1}(\I-T_{\lambda})^{-1}f, 
\end{align}
where 
\begin{align*}
	T_{\lambda}f(\x)=\langle F(\x), \Q^{\frac{1}{2}}\D_{\x}(\lambda\I-\mathcal{N}_2)^{-1}f(\x)\rangle. 
\end{align*}
\end{lemma}
\begin{proof}
	See \cite[Proposition 4.3]{VBGD}. 
\end{proof}

\section{Applications in Control Theory: Infinite Horizon Problem} \label{sec5}\setcounter{equation}{0} 

We consider an infinite horizon problem described by  the state equation for incompressible 2D stochastic convective Brinkman-Forchheimer fluids:
\begin{equation}\label{51}
	\left\{
	\begin{aligned}
		\d\X(t)+[\mu \A\X(t)+\alpha\X(t)+\B(\X(t))+\beta\mathcal{C}(\X(t))
		]\d t&=\sqrt{\Q}\mathrm{U}(t)\d t+\sqrt{\Q}\d\W(t), \ t>0,\\
		\X(0)&=\x,
	\end{aligned}
	\right.
\end{equation}
with a cost functional of the form
\begin{align}\label{cost}
	J_{\infty}(\mathrm{U})=\E\left\{\int_0^{\infty}e^{-\lambda t}\left[f(\X(t,\x;\mathrm{U}))+h(\mathrm{U}(t))\right]\d t\right\},
\end{align}
over all adapted square integrable controls $\mathrm{U}$, where $f$ and $h$ are real valued  funtions on $\H$ and $\lambda>0$ is a discount factor. We define the admissible class of control processes as
\begin{align}\label{admis}
	\mathcal{U}_{\mathrm{ad}}:=\left\{\mathrm{U}(\cdot)\in\mathrm{L}^2(\Omega,\mathrm{L}^2(0,\infty;\H)):  \|\mathrm{U}(t)\|_{\H}\leq R, \ \mathbb{P}\text{-a.s. and } \mathrm{U}(\cdot)  \text{ is adapted to } \{\mathscr{F}_t\}_{t\geq 0} \right\},
\end{align}
where $R>0$ is fixed, corresponding to a fixed reference probability space $(\Omega,\mathscr{F},\mathbb{P})$.
We again call $(\X(\cdot),\mathrm{U}(\cdot))$ an admissible control pair if $\mathrm{U}(\cdot)$  is an $\mathscr{F}_t$-adapted process with values in $\H$ and $\X(\cdot)$ is the weak solution to \eqref{51} corresponding to $\mathrm{U}(\cdot)$. To every admissible control pair, we associate a cost \eqref{cost}. The optimal control problem is to find an admissible control $\mathrm{U}(\cdot)$ which minimizes the cost functional \eqref{cost}. We define the value function $\mathcal{V}:\H\to\mathbb{R}$ corresponding to the cost functional \eqref{cost}, as
\begin{align}\label{56}
	\mathcal{V}(\x):=\inf_{\mathrm{U}(\cdot)\in\mathcal{U}_{\mathrm{ad}}}J_{\infty}(\mathrm{U})=\inf_{\mathrm{U}\in\mathcal{U}_{\mathrm{ad}}}\E\left\{\int_0^{\infty}e^{-\lambda t}\left[f(\X(t,\x;\mathrm{U}(\cdot)))+h(\mathrm{U}(t))\right]\d t\right\}.
\end{align}
We consider the following  infinite-dimensional second order stationary HJB equation (or \emph{semilinear Kolomogorov equations}) related to the stochastic optimal control problem \eqref{51}:
\begin{align}\label{55}
\lambda\varphi(\x)-\frac{1}{2}\Tr\left[\Q\D_{\x}^2\varphi(\x)\right]+ (\mu\A\x+\alpha\x+\B(\x)+\beta\mathcal{C}(\x),\D_{\x}\varphi(\x))+g(\Q^{1/2}\D_{\x}\varphi(\x))=f(\x),
\end{align}
where $\lambda>0$, $f\in\L^2(\H;\eta)$ and the Hamiltonian $g:\H\to\R$  is Lipschitz continuous. 
Moreover, $g$ is defined as the Legendre transform of the convex function $h : \H\to\R$:
\begin{align}\label{leg}
	g(\x)=\sup\limits_{\y\in\H}\left\{(\x,\y)-h(\y)\right\},\ \x\in\H.
\end{align}

\begin{example}\label{epe}
	 For instance (cf. \cite[pp. 294]{gdp7}), one can take 
	\begin{enumerate}
		\item $h(\x)=\frac{1}{2}\|\x\|_{\H}^2$ for $\x\in\H$, so that the Hamiltonian is given by $g(\x)=\frac{1}{2}\|\x\|_{\H}^2$.
		\item for a given $R>0$, 
		\begin{align*}
			h(\x)=\left\{\begin{array}{cc}\frac{1}{2}\|\x\|_{\H}^2, &\text{ if } \|\x\|_{\H}\leq R,\\
				+\infty,  &\text{ if } \|\x\|_{\H}>R,\end{array}\right.
		\end{align*}
		so that the Hamiltonian $g(\cdot)$ is explicitly given by 
		\begin{align}\label{how1}
			g(\x)=\left\{\begin{array}{cc}\frac{1}{2}\|\x\|_{\H}^2, &\text{ if } \|\x\|_{\H}\leq R,\\
				R\|\x\|_{\H}-\frac{R^2}{2}, &\text{ if } \|\x\|_{\H}>R.\end{array}\right.
		\end{align}
	\end{enumerate}
	Moreover, for the second case, the optimal feedback control is given formally as  (see \cite{gozzi})
	\begin{align*}
		\wi{\mathrm{U}}(t)=-\mathscr{G}(\Q^{1/2}\D_{\x}\varphi(\x)),
	\end{align*}
	where $\mathscr{G}(\cdot)$ is given by
	\begin{align}\label{how2}
		\mathscr{G}(\x):=
		\begin{cases}
			\x,  &\text{ when } \|\x\|_{\H}\leq R,\\
			\frac{R}{\|\x\|_{\H}}\x, &\text{ when } \|\x\|_{\H}>R.
		\end{cases}
	\end{align} 
\end{example}

\begin{remark}\label{mildsoln}
By a solution of \eqref{55}, we mean a solution in the mild form. The mild solution is defined as the solution of the following integral equation (cf. \cite{gozzi3}):
	\begin{align*}
		\varphi(\x)=\int_0^{\infty} e^{-\lambda t} \E[f(\X(t,\x))+g(\Q^{1/2}\D_{\x}\varphi(\X(t,\x)))]\d t,
	\end{align*}
	for all $\lambda>0$ and all Lipschitz continuous functions $g:\H\to\R$.
%
\end{remark}

\subsection{Solution to Hamilton-Jacobi-Bellman (HJB) equations}\label{HJBeqn}
 We need to solve the HJB equation \eqref{55} associated with \eqref{51} to do this. 
More precisely, the solution $\varphi$ of \eqref{55} is the value function $\mathcal{V}$ defined in \eqref{56} of the optimal control problem associated with \eqref{51}.


We write \eqref{55} in following form:
\begin{align}\label{55*}
	\lambda\varphi(\x)-\mathcal{N}_2\varphi(\x)+g(\Q^{1/2}\D_{\x}\varphi(\x))=f(\x),
\end{align}
and set $\lambda\varphi-\mathcal{N}_2\varphi=\psi$ so that $\varphi=(\lambda-\mathcal{N}_2)^{-1}\psi$. Then, we have 
\begin{align}\label{contr}
	\psi=f-g(\Q^{1/2}\D_{\x}(\lambda-\mathcal{N}_2)^{-1}\psi).
\end{align}
We multiply the equation $\lambda\varphi-\mathcal{N}_2\varphi=\psi$ by $\varphi$ and integrate over $\H$ with respect to $\eta$ and use  Carre du Champ's identity to obtain
\begin{align*}
	    \lambda\int_{\H}\varphi^2\d\eta+\frac{1}{2}\int_{\H}\|\Q^{1/2}\D_{\x}\varphi\|_{\H}^2\d\eta =\int_{\H}\varphi\psi\d\eta\leq\frac{\lambda}{2}\int_{\H}\varphi^2\d\eta+\frac{1}{2\lambda}
	    \int_{\H} \psi^2\d\eta,
\end{align*} 
which in turn implies that
\begin{align}\label{contr1}
\int_{\H}\varphi^2\d\eta\leq\frac{1}{\lambda^2}\int_{\H}\psi^2\d\eta  \ \text{ and } \
	\int_{\H}\|\Q^{1/2}\D_{\x}\varphi\|_{\H}^2\d\eta\leq\frac{1}{\lambda} \int_{\H}\psi^2\d\eta .
\end{align}
Now, we take $\gamma(\psi):=f-g(\Q^{1/2}\D_{\x}(\lambda-\mathcal{N}_2)^{-1}\psi).$ Then, for any $\psi,\wi\psi\in\L^2(\H,\eta)$, we calculate 
\begin{align*}
	|\gamma(\psi)-\gamma(\wi\psi)|&\leq|g(\Q^{1/2}\D_{\x}(\lambda-\mathcal{N}_2)^{-1}\psi)-g(\Q^{1/2}\D_{\x}(\lambda-\mathcal{N}_2)^{-1}\wi\psi)|\no\\&\leq\|g\|_{\mathrm{Lip}}\|\Q^{1/2}\D_{\x}(\lambda-\mathcal{N}_2)^{-1}\psi-\Q^{1/2}\D_{\x}(\lambda-\mathcal{N}_2)^{-1}\wi\psi\|_{\H}.
\end{align*}
By using \eqref{contr1}, we deduce
\begin{align*}
	\int_{\H} |\gamma(\psi)-\gamma(\wi\psi)|^2\d\eta&\leq\|g\|_{\mathrm{Lip}}^2\int_{\H} 
	\|\Q^{1/2}\D_{\x}(\lambda-\mathcal{N}_2)^{-1}\psi-\Q^{1/2}\D_{\x}(\lambda-\mathcal{N}_2)^{-1}\wi\psi\|_{\H}^2\d\eta
	\no\\&\leq\frac{\|g\|_{\mathrm{Lip}}^2}{\lambda}\int_{\H} |\psi-\wi\psi|^2\d\eta.
\end{align*}
Therefore, by choosing  $\lambda>\|g\|_{\mathrm{Lip}}^2$, it follows that $\gamma(\cdot):\L^2(\H,\eta)\to\L^2(\H,\eta)$ is a contraction map and thus by the Banach fixed point theorem, \eqref{55} or \eqref{55*} has a unique solution $\varphi\in\L^2(\H,\eta)$.  Moreover, from \eqref{contr1}, we have $\Q^{1/2}\D_{\x}\varphi\in\L^2(\H,\eta;\H)$. 
\begin{remark}
Since $\varphi(\cdot)$ is resolvent of a (unique) $m$-dissipative operator (see \eqref{55*}-\eqref{contr}) whose resolvent set contains $(0,+\infty)$. Thus, from Remark \ref{mildsoln}, it implies that \eqref{55} has a unique mild solution for all $\lambda>0$ (cf. \cite{gozzi3}).
\end{remark}

\subsection{Existence of an optimal control}\label{existopt} Let us fix the Hamiltonian  (see example \eqref{epe})
 \begin{align*}
 	g(\x)=\left\{\begin{array}{cc}\frac{1}{2}\|\x\|_{\H}^2, &\text{ if } \ \|\x\|_{\H}\leq R,\\
 		R\|\x\|_{\H}-\frac{R^2}{2}, &\text{ if }\  \|\x\|_{\H}>R.\end{array}\right.
 \end{align*}
 We need the following lemma in the sequel:
\begin{lemma}\label{verif}
 Let $\varphi(\cdot)$ be the mild solution of the HJB equation \eqref{55}. Then, the following identity holds:
 \begin{align}\label{713}
 	\varphi(\x)+\E\left(\int_0^{\infty}\frac{e^{-\lambda t}}{2} \left[\|\mathrm{U}(t)+\Q^{1/2}\D_{\x}\varphi(\x)\|_{\H}^2 -\Psi(\|\Q^{1/2}\D_{\x}\varphi(\x)\|_{\H}-R)\right] \d t\right)=J_{\infty}(\mathrm{U}),
 \end{align}
 where 
 \begin{align*}
 	\Psi(\upxi):=
 	\begin{cases}
 		0, \ &\text{ if } \ \upxi\leq0,\\
 		\upxi^2, \ &\text{ if } \ \upxi>0.
 	\end{cases}
 \end{align*}
\end{lemma}
\begin{proof}
 Without loss of generality, we assume that $\varphi\in\C_b^2(\H)$ (otheriwse one can proceed via Galerkin approximations, \cite[Chapter 13]{gdp7}). Since, $\X(\cdot)$ is the unique strong solution of \eqref{51}, therefore on applying the It\^o formula to the mapping $t\mapsto e^{-\lambda t}\varphi(\X(t,\x))$, we obtain for a.e. $t\in[0,T]$, $\mathbb{P}$-a.s. 
 \begin{align}\label{stpr}
&\d(e^{-\lambda t}\varphi(\X(t,\x)))\nonumber\\&=-e^{-\lambda t}(\mu\A\X(t,\x)+\alpha\X(t,\x)+\B(\X(t,\x))+ \beta\mathcal{C}(\X(t,\x)),\D_{\x}\varphi(\X(t,\x))) \no\\&\quad + (\Q^{1/2}\mathrm{U}(t),\D_{\x}\varphi(\X(t,\x)))+e^{-\lambda t}(\Q^{1/2}\d\W,\D_{\x}\varphi) \no\\&\quad+\frac{e^{-\lambda t}}{2} \Tr\left[\Q\D_{\x}^2\varphi(\X(t,\x))\right]-\lambda e^{-\lambda t}\varphi(\X(t,\x)).
\end{align}
Let us first consider the case when $\|\x\|_{\H}\leq R$. Then, from \eqref{55}, the above equality implies 
 \begin{align}\label{statio}
\d(e^{-\lambda t}\varphi(\X(t,\x)))&=e^{-\lambda t}(\Q^{1/2}\mathrm{U}(t),\D_{\x}\varphi(\x)) +e^{-\lambda t} (\Q^{1/2}\d\W,\D_{\x}\varphi)\no\\&\quad+\frac{e^{-\lambda t}}{2} \|\Q^{1/2}\D_{\x}\varphi(\x)\|_{\H}^2-e^{-\lambda t}  f(\X(t,\x))\no\\&=
\frac{e^{-\lambda t}}{2}\|\mathrm{U}(t)+\Q^{1/2}\D_{\x}\varphi(\x)\|_{\H}^2+e^{-\lambda t} (\Q^{1/2}\d\W,\D_{\x}\varphi)\no\\&\quad-\frac{e^{-\lambda t}}{2}\|\mathrm{U}(t)\|_{\H}^2-
e^{-\lambda t}  f(\X(t,\x)).
 \end{align}
By integrating with respect to $t\in[0,T]$ on each side of \eqref{statio}, taking expectation, and then  using the fact that $\int_0^{t} e^{-\lambda s}(\Q^{1/2} \d\W(s),\D_{\x}\varphi)$ is an $\mathscr{F}_t$-adapted martingale, we obtain
 \begin{align*}
 	e^{-\lambda t}\varphi(\X(t,\x))-\varphi(\x)&=-\E\bigg( \int_0^{t}e^{-\lambda s} \left[ f(\X(s,\x))+\frac{1}{2} \|\mathrm{U}(s)\|_{\H}^2\right]\d s\no\\&\quad-\int_0^{t} \frac{e^{-\lambda t}}{2} \|\mathrm{U}(s)+\Q^{1/2}\D_{\x}\varphi(\x) \|_{\H}^2\d s \bigg).
 \end{align*}
On taking the limit as $t\to\infty$, the above inequality reduces to
   \begin{align}\label{cost1}
   \varphi(\x)&=\E\bigg( \int_0^{\infty} e^{-\lambda t}\left[ f(\X(t,\x))+\frac{1}{2} \|\mathrm{U}(t)\|_{\H}^2\right]\d t-\int_0^{\infty} \frac{e^{-\lambda t}}{2}  \|\mathrm{U}(t)+\Q^{1/2}\D_{\x}\varphi(\x) \|_{\H}^2\d t \bigg)\no\\&=J_{\infty}(\mathrm{U})-
   \E\left(\int_0^{\infty} \frac{e^{-\lambda t}}{2}\|\mathrm{U}(t)+\Q^{1/2}\D_{\x}\varphi(\x)\|_{\H}^ 2 \d t\right),
	\end{align}
where we have used the definition of the cost functional $J_{\infty}(\mathrm{U})$ given in \eqref{cost}. 	

Now, when $\|\x\|_{\H}>R$, then from the definition of $g$ and using \eqref{55}, we obtain
	\begin{align*}
		\d(e^{-\lambda t}\varphi(\X(t,\x)))&=e^{-\lambda t}(\Q^{1/2}\mathrm{U}(t),\D_{\x}\varphi(\x)) +e^{-\lambda t} (\Q^{1/2}\d\W,\D_{\x}\varphi)\no\\&\quad+e^{-\lambda t} \left(R\|\Q^{1/2}\D_{\x}\varphi(\x)\|_{\H}-\frac{R^2}{2}\right)-e^{-\lambda t}  f(\X(t,\x))\no\\&=
		\frac{e^{-\lambda t}}{2}\|\mathrm{U}(t)+\Q^{1/2}\D_{\x}\varphi(\x)\|_{\H}^2 -\frac{e^{-\lambda t}}{2}(\|\Q^{1/2}\D_{\x}\varphi(\x)\|_{\H}-R)^2\no\\& \quad+e^{-\lambda t} (\Q^{1/2}\d\W,\D_{\x}\varphi)-\frac{e^{-\lambda t}}{2}\|\mathrm{U}(t)\|_{\H}^2-e^{-\lambda t}  f(\X(t,\x)).
	\end{align*}
	Subsequently, continuing as before, we ultimately have
	\begin{align}\label{cost2}
		\varphi(\x)=J_{\infty}(\mathrm{U})-\E\left(\int_0^{\infty}\frac{e^{-\lambda t}}{2} \left[\|\mathrm{U}(t)+\Q^{1/2}\D_{\x}\varphi(\x)\|_{\H}^2 -(\|\Q^{1/2}\D_{\x}\varphi(\x)\|_{\H}-R)^2\right] \d t\right).
	\end{align}
  On combining, \eqref{cost1}-\eqref{cost2}, we can arrive at \eqref{713}. 
\end{proof}  
The following result for the existence of an optimal control for the problem \eqref{56} can be proved by using  Lemma \ref{verif} and following in the similar lines of \cite[Section 13.4.2]{gdp7} (see \cite{gozzi1,gozzi2} also):
\begin{lemma}
	There exists an optimal pair $(\mathrm{U}^*(\cdot),\X^*(\cdot))$ for the problem \eqref{56} such that the following optimal feedback formula holds:
	\begin{equation*}
		\mathrm{U}^*(t)=\mathcal{G}(\Q^{1/2}\D_{\x}\varphi(\X^*(\x,t))), \ \text{ for all } \ t\geq 0 \ \text{ and } \ \x\in\H,
	\end{equation*}
	where 
		\begin{equation*}
	    \mathcal{G}(\mathfrak{p})=\D_{\mathfrak{p}} g(\mathfrak{p})=
	    \left\{
		\begin{aligned}
			-\mathfrak{p}, \ \text{ when } \|\mathfrak{p}\|_{\H}\leq R,\\
			-R\frac{\mathfrak{p}}{\|\mathfrak{p}\|_{\H}}, \ \text{ when } \|\mathfrak{p}\|_{\H}> R.
		\end{aligned}
		\right.
	\end{equation*}
	Furthermore, the optimal cost is given by $J^*_{\infty}(\x)=\varphi(\x)$.
\end{lemma}

\section{Applications in Control Theory: Optimal stopping problem}\label{sec6}\setcounter{equation}{0} 
	Let $\X(\cdot)$ be the process associated with the following SCBF equations: 
	\begin{equation}\label{4p2}
		\left\{
		\begin{aligned}
			\d\X(r)+[\mu \A\X(r)+\alpha\X(r)+ \B(\X(r))+\beta\mathcal{C}(\X(r))]\d r&=\sqrt{\Q}\d\W(r), \ r> t,\\
			\X(t)&=\x.
		\end{aligned}
		\right.
	\end{equation}
Let us denote by $\mathscr{F}^t_r$, a filtration generated by $\W(\cdot)$, that is, $\mathscr{F}^t_r:=\sigma\{\W(s):t\leq s\leq r \}$. For an $\mathscr{F}^t_r$-adapted stopping time $\tau$, let us define a cost functional associated with \eqref{4p2} as
\begin{align}\label{4.1.1}
	J_{\tau}(\x)=\E\left[\int_t^{\tau}\F(s,\X(s))\d s\right]+ \E[\G(\X(\tau))],
\end{align}
where $\F:(0,\infty)\times\H\to\R$ and $\G:\H\to\R$ are given functions. The integral cost in \eqref{4.1.1} corresponds to what is paid while the process is not going to be stopped. The final cost, that is, $\E[\G(\X(\tau))]$, corresponds to what is paid when we decide to stop the process at time $\tau$. We say $\tau$ is \emph{admissible} whenever $J_\tau(\x)$ is well-defined. Then, the optimal stopping problem is to find a stopping time $\tau^*$ which minimizes $J_\tau(\x)$. 

Let us define the value function of the above optimal stopping problem associated with \eqref{4p2} as
\begin{align}\label{4.1}
	\varphi(t,\x):=\inf_{\tau\in\mathfrak{M}}\left\{\E\left[\int_t^{\tau}\F(s,\X(s))\d s\right]+\E[\G(\X(\tau))]\right\},
\end{align}
where $\mathfrak{M}$ is the family of all $\{\mathscr{F}_r^t\}_{r\geq t}$-adapted stopping times defined as 
\begin{align*}
	\mathfrak{M}:=\{\tau\in\mathscr{F}^t_r: t\leq\tau\leq T, \  \mathbb{P}\text{-a.s.} \}.
\end{align*}
The optimal stopping problem can be reduced to a free boundary problem or an obstacle problem, which can be solved explicitly. One can show that the value function $\varphi$ defined by \eqref{4.1} (after a suitable change of time variable) is formally the solution to the  following variational inequality: 
\begin{equation}\label{4.3}
	\left\{
	\begin{aligned}
		&	\frac{\partial\varphi}{\partial t}(t,\x)- \frac{1}{2}\Tr\left[\Q\D_{\x}^2\varphi(t,\x)\right]+(\mu\A\x+\alpha\x+\B(\x)+\beta\mathcal{C}(\x),\D_{\x}\varphi(t,\x))\leq\F(t,\x),\\ &\quad\text{ for all }\ t\geq 0, \ \x\in\D(\A), \ \varphi(t,\x)\leq\G(\x), \ \text{ for all }\ t\geq 0, \ \x\in\H, \\
		&	\frac{\partial\varphi}{\partial t}(t,\x)-\frac{1}{2}\Tr\left[\Q\D_{\x}^2\varphi(t,\x)\right]+(\mu\A\x+\alpha\x+\B(\x)+\beta\mathcal{C}(\x),\D_{\x}\varphi(t,\x))=\F(t,\x),\\
		&\quad\text{ in }\ \{\x:\varphi(t,\x)<\G(\x)\}, \ \varphi(0,\x)=\varphi_0(\x),\ \x\in\H. 
	\end{aligned}
	\right. 
\end{equation}

\begin{remark}\cite[Chapter I]{gpas}
	 Note that the unique solution of \eqref{4p2} is an $\H$-valued continuous Markov process. Being Markovian means the process $\X(r)$ always starts afresh. It implies that at time $r$, the decision `to stop' or `to continue' with the observation should depend only on the present state $\X(r)$ of the process and not on its past states $\X(s)$ for $t\leq s<r$. Thus following the sample path $r\mapsto\X(r)(\omega),$ for a given and fixed $\omega\in\Omega$, and evaluating the integral and final costs, it is naturally expected that at each time $r,$ we shall be able  to decide optimally either to continue with the observation or to stop it. In this way, the state space $\H$ naturally splits into two regions: the continuation (or noncoincidence) set $C$ and the stopping (or coincidence) set $S$.  Then, it follows that as soon as the observed value $\X(r)(\omega)$ enters the stopping  region $S$, observation should be stopped, and  optimal stopping time is obtained. However, the central question here is how to determine the sets $C$ (and $S$).
	
Formally, let us introduce the continuation set	\begin{align}\label{noncon}
		C=\{\x\in\H:\varphi(\cdot,\x)<G(\x)\},
	\end{align}
	and the stopping set
	\begin{align}\label{con}
		S=\{\x\in\H:\varphi(\cdot,\x)=G(\x)\}.
	\end{align}
	Note that if $\varphi(\cdot,\x)$ and $G(\x)$ is continuous, then $C$ is open, and the set $S$ is closed. Further, let us introduce the first exit time $\tau_C$ of the process $\X(r)$ from $C$ (or first entry time of the process $\X(r)$ into $S$) by setting
	\begin{align*}
		\tau_C:=\{r\geq t:\X(r)\notin C\}.
	\end{align*}
	Note that $\tau_C$ is an $\mathscr{F}_t$-adapted stopping time when $C$ is open, since, both the process $\X(t)$ and $\mathscr{F}_t$ are right continuous. 
\end{remark}

\subsection{A ``heuristic'' explanation of the variational inequality \eqref{4.3}}\label{derivation}

Let $\tau$ be an admissible stopping rule and let $J_\tau(\x)$ be the corresponding cost functional. 
Suppose that there is a stopping rule $\tau^*$, with continuation region $C$, which minimizes the cost functional $J_{\tau}(\x)$ for every $\x\in\H$ and therefore from \eqref{4.1}, the value function is $\varphi(t,\x)=J_{\tau^*}(\x)$.
We aim to find the equation for $\varphi(\cdot,\x)$ (which we will see later that it turns out to be variational inequality). To this end, let us define a new stopping rule $\tau_1$ as
\begin{align*}
	\tau_1:=\{r\geq\tau:\X(r)\notin C\}. 
\end{align*}
It means $\tau_1$ is the rule under which we do not stop the process, until the time $\tau$, and continue optimally afterwards. Then by assumption, we have
\begin{align}\label{tau1}
	J_{\tau^*}(\x)\leq J_{\tau_1}(\x) \ \text{ for all } \ \x\in\H.
\end{align}
Now, by using the Markovian proiperty of the process $\X(\cdot)$ and the tower property of the conditional expectation,we calculate
\begin{align*}
J_{\tau_1}(\x)&=\E\left[\int_t^{\tau_1}\F(s,\X(s))\d s+ \G(\X(\tau_1))\right]\nonumber\\&=
\E\left[\int_t^{\tau}\F(s,\X(s))\d s+\int_{\tau}^{\tau_1}\F(s,\X(s))\d s+ \G(\X(\tau_1))\right]\nonumber\\&=
\E\bigg\{\E\left[\int_t^{\tau}\F(s,\X(s))\d s+\int_{\tau}^{\tau_1}\F(s,\X(s))\d s+ \G(\X(\tau_1))\right]\bigg\vert\mathscr{F}^t_{\tau}\bigg\}
\nonumber\\&=\E\bigg\{\int_t^{\tau}\F(s,\X(s))\d s+ \E\left[\int_{\tau}^{\tau_1}\F(s,\X(s))\d s+ \G(\X(\tau_1))\big\vert\mathscr{F}^t_{\tau}\right]\bigg\}\nonumber\\&
=\E\bigg\{\int_t^{\tau}\F(s,\X(s))\d s+ \E\left[\int_{\tau}^{\tau_1}\F(s,\X(s))\d s+ \G(\X(\tau_1))\big\vert\X(\tau)\right]\bigg\}\nonumber\\&=
\E\bigg\{\int_t^{\tau}\F(s,\X(s))\d s+ J_{\tau_1}(\X(\tau))\bigg\}.
\end{align*}
Thus, from \eqref{tau1}, we conclude that
\begin{align}\label{tau1.1}
	\varphi(t,\x)\leq\E\bigg\{\int_t^{\tau}\F(s,\X(s))\d s+ \varphi(\tau,\X(\tau))\bigg\}.
\end{align}
Note that if we choose $\tau\leq\tau^*$, then by repeating the above procedure, we will get equality in \eqref{tau1.1}, that is,
\begin{align}\label{tau1.2}
	\varphi(t,\x)=\E\bigg\{\int_t^{\tau}\F(s,\X(s))\d s+ \varphi(\tau,\X(\tau))\bigg\}.
\end{align}
On applying the It\^o formula to $\varphi(\cdot,\X(\cdot))$ by assuming that it is sufficiently smooth, and then taking expectation, we get 
\begin{align}\label{itof}
	&\E[\varphi(\tau,\X(\tau))]\nonumber\\&=\varphi(t,\x)+\E\bigg[-\int_t^{\tau} \left(\frac{\partial\varphi}{\partial t}(s,\X(s))+(\mu\A\X+\alpha\X+ \B(\X)+\beta\mathcal{C}(\X),\D_{\x}\varphi(s,\X(s)))\right)
	\d s\bigg]\no\\&\quad+\E\bigg[\int_t^{\tau} \frac{1}{2}\Tr\left[\Q\D_{\x}^2\varphi(s,\X(s))\right]\d s\bigg].
\end{align}
Using \eqref{tau1.2}, we obtain 
\begin{align*}
	&\E\bigg[\int_t^{\tau} \left(\frac{\partial\varphi}{\partial t}(s,\X(s))+(\mu\A\X+\alpha\x+ \B(\X)+\beta\mathcal{C}(\X),\D_{\x}\varphi(s,\X(s)))\right)
	\d s\bigg]\no\\&\quad-\E\bigg[\int_t^{\tau} \bigg(\frac{1}{2}\Tr\left[\Q\D_{\x}^2\varphi(s,\X(s))\right]-\F(s,\X(s))\bigg)\d s\bigg]=0.
\end{align*}
Since the above equality holds for any $\tau\leq\tau^*$, we then find
\begin{align*}
	\frac{\partial\varphi}{\partial t}(t,\x)-\frac{1}{2}\Tr\left[\Q\D_{\x}^2\varphi(t,\x)\right]+ (\mu\A\x+\alpha\x+\B(\x)+\beta\mathcal{C}(\x),\D_{\x}\varphi(t,\x))=\F(t,\x),
\end{align*}
provided $\tau^*>0$, that is, for $\x\in C$. Moreover, for $\x\notin C$, we have $\varphi(t,\x)=G(\x)$. Finally, we deduce
\begin{equation*}
	\left\{
	\begin{aligned}
			\frac{\partial\varphi}{\partial t}(t,\x)&-\frac{1}{2}\Tr\left[\Q\D_{\x}^2\varphi(t,\x)\right]\nonumber\\&+ (\mu\A\x+\alpha\x+\B(\x)+\beta\mathcal{C}(\x),\D_{\x}\varphi(t,\x))=\F(t,\x), \  &&\text{ for } \  \x\in C,\\
			\varphi(t,\x)&=G(\x),  \  &&\text{ for } \  \x\notin C.
	\end{aligned}
	\right.
\end{equation*}

Now, assume that $\tau$ is any arbitrary admissible stopping time. Then, we have the inequality \eqref{tau1.1}. By utilizing this inequality  in the It\^o formula \eqref{itof} and proceeding as above, we arrive at
\begin{align*}
		\frac{\partial\varphi}{\partial t}(t,\x)-\frac{1}{2}\Tr\left[\Q\D_{\x}^2\varphi(t,\x)\right]+ (\mu\A\x+\alpha\x+\B(\x)+\beta\mathcal{C}(\x),\D_{\x}\varphi(t,\x))\leq \F(t,\x), 
\end{align*}
and $\varphi(t,\x)\leq G(\x)$ for all $\x\in\H$. It completes formally the derivation of \eqref{4.3}.

\begin{remark}
	
The idea is to solve the free-boundary problems \eqref{4.3}, and then we show that it coincides with the solution of the optimal stopping problem \eqref{4.1}. This is called \emph{verification theorem}.
\end{remark}

\subsection{Formulation of abstract problem}
Let us define 
\begin{align}\label{45}
	K=\left\{\varphi\in\L^2(\H;\eta):\varphi\leq \G, \ \eta \text{-a.e.} \right\},
\end{align} 
which is a closed convex subset of $\mathbb{L}^2(\H;\eta),$  where $\eta$ is the invariant measure for $\mathrm{P}_t$.
We define the normal cone $N_K:\mathbb{L}^2(\H;\eta)\to 2^{\mathbb{L}^2(\H;\eta)}$ at $K$ in $\varphi$ (cf. \cite{VB2}) by 
\begin{align}\label{4p54}
	N_K(\varphi)=\left\{\zeta\in\mathbb{L}^2(\H;\eta):\int_{\H}\zeta(\varphi-\psi)\eta(\d x)\geq 0,\ \text{ for all }\ \psi\in K \right\},\ \varphi\in K,
\end{align}
	or equivalently, we can write
	\begin{align*}
		N_K(\varphi)=\left\{\zeta\in\mathbb{L}^2(\H;\eta): \zeta(\x)=0 \ \text{ if } \ \varphi(\x)<\G(\x) \text{ and } \zeta(\x)\geq0 \  \text{ if } \ \varphi(\x)=\G(\x), \ \eta \text{-a.e.}\right\}.
		\end{align*}
We are going to study the existence and uniqueness result for the problem \eqref{4.3} which can be viewed as a nonlinear equation of the form: 
\begin{equation}\label{44}
	\left\{
	\begin{aligned}
		\frac{\d\varphi(t)}{\d t}-\mathcal{N}_2\varphi(t)+N_K\varphi(t) &\ni\F(t), \ t\in(0,T),\\
		\varphi(0)&=\varphi_0,
	\end{aligned}
	\right. 
\end{equation}
where $\varphi_0\in\mathbb{L}^2(\H;\eta)$ and $\F\in\mathrm{L}^2(0,T;\mathbb{L}^2(\H;\eta))$ are given. Here $\mathcal{N}_2$ is the infinitesimal generator (the Kolmogorov operator) of the transition semigroup $\P_t$ associated with the stochastic SCBF system \eqref{4p2}. 


	 The operator $\mathcal{N}_2$ appearing in the variational inequality \eqref{44} is just the operator $\mathcal{N}_2$ defined in Theorem \ref{thm4.7}.  Let us denote by 
	 \begin{align*}
	 	W^{1,p}([0,T];\mathbb{L}^2(\H;\eta)), \ \ 1\leq p\leq \infty,
	 \end{align*}
	 the space of all absolutely continuous functions 
	 $$\varphi:[0,T]\to \mathbb{L}^2(\H;\eta)\ \text{ such that }\ \frac{\d\varphi}{\d t}\in \mathrm{L}^{p}(0,T;\mathbb{L}^2(\H;\eta)).$$
	 \begin{definition}\cite[Definition 3.1, Chapter III, pp. 64]{OPHB}
	 	A \emph{strong solution} of \eqref{44} is a function $\varphi\in\mathrm{W}^{1,1}([0,T];\mathbb{L}^2(\H;\eta))$ which satisfies \eqref{44} a.e. on $[0,T]$.
%
	 \end{definition}
     Let us now formulate the main existence and uniqueness result for the problem \eqref{44}. 
\begin{theorem}\label{main-thm}
	Assume that  $\mu$ and $\alpha$ satisfy \eqref{439}  and  the condition \eqref{Asmp} holds. Suppose further that $G\in\mathbb{L}^2(\H;\eta)$ and 
	\begin{align}\label{456}
		(P_tG)(\x)\leq G(\x)\ \text{ for all }\ t\geq 0,\ \x\in\H. 
	\end{align}
	Then for each $\varphi_0\in\D(\mathcal{N}_2)\cap K$ and $F\in W^{1,1}([0,T];\mathbb{L}^2(\H;\eta)),$ there exists a unique function $\varphi\in W^{1,\infty}([0,T];\mathbb{L}^2(\H;\eta)),$ such that $\mathcal{N}_2\varphi\in\mathrm{L}^{\infty}(0,T;\mathbb{L}^2(\H;\eta))$ and 
	\begin{equation}\label{454}
		\left\{
	\begin{aligned}
		\frac{\d}{\d t}\varphi(t)-\mathcal{N}_2\varphi(t)+\zeta(t)&=F(t)\ \text{ for a.e.}\ t\in[0,T],\\
		\zeta(t)&\in N_K(\varphi(t))\ \text{ for a.e.}\ t\in[0,T],\\
		\varphi(0)&=\varphi_0.
	\end{aligned}
	\right. 
	\end{equation}
	Furthermore, $\varphi:[0,T]\to \mathbb{L}^2(\H;\eta)$ is differentiable from the right and
	\begin{align}
		\frac{\d^+}{\d t}\varphi(t)&=F(t)+\mathcal{N}_2 \varphi(t)-P_{N_K(\varphi(t))}(F(t)+\mathcal{N}_2\varphi(t)), \ \text{ for all }\ t\in[0,T),
	\end{align}
	where $P_{N_K(\varphi)}$  is the projection on the cone $N_{K}(\varphi)$. 
\end{theorem}

\begin{remark}
	It is important to note that a solution to \eqref{454} can be viewed as a solution to the variational inequality \eqref{4.3}. Indeed, if $\G,\varphi\in\C_b(\H)$ and 
choose $\mathfrak{h}\in\L^2(\H;\eta)$ with $\mathfrak{h}\leq\boldsymbol{0}$ such that $\psi=\varphi+\mathfrak{h}\in K$, then, for all $\zeta\in N_K(\varphi)$, we have 
	\begin{align}\label{oi1}
	\int_{\H}\zeta(\x)\mathfrak{h}(\x)\eta(\d\x)\leq0,
\end{align}
for all $\mathfrak{h}\in\L^2(\H;\eta)$ with $\mathfrak{h}\leq\boldsymbol{0}$. Let us take $\mathfrak{h}:=-\zeta^{-}$, where $\zeta^{-}$ is the negative part of $\zeta$ defined as
\begin{align*}
	\zeta^{-}:=\begin{cases}
		-\zeta, \ &\mbox{ if } \  \zeta<0,\\
		0, \  &\mbox{ if } \  \zeta\geq0.
	\end{cases}
\end{align*}
Then, clearly $\zeta(\x)\zeta^{-}(\x)=-|\zeta^{-}(\x)|^2$ for any $\x\in\H$. Thus, from \eqref{oi1}, we get
\begin{align}\label{qa1}
	\int_{\H}|\zeta^{-}(\x)|^2\eta(\d\x)\leq0.
\end{align}
Now, consider the following set
\begin{align*}
	\mathcal{H}_0:=\{\x\in\H:\varphi(\x)<\G(\x)\}.
\end{align*}
Let $\vartheta\in\C_b(\H)$ be such that $\mathrm{supp}(\vartheta)\subset\mathcal{H}_0$. Then, for sufficiently small $|\e|>0$, one can easily verify that $\psi=\varphi\pm \e\vartheta\leq\G$  and therefore $\psi\in K$. Then, from the definition of $N_K(\varphi)$, we have
\begin{align*}
	\mp\e\int_{\H}\zeta(\x)\vartheta(\x)\eta(\d\x)\leq0.
\end{align*}
Since $\e$ is both positive and negative, we conclude that
\begin{align}\label{oi2}
  \int_{\H}\zeta(\x)\vartheta(\x)\eta(\d\x)=0,
\end{align}
for all $\vartheta\in\C_b(\H)$ with $\mathrm{supp}(\vartheta)\subset\mathcal{H}_0$. 
 
Let us recall that the support of the measure $\eta$ (see \cite[Theorem 2.1, Chapter II]{KRP}) is defined as
 \begin{align*}
 	\mathrm{supp}(\eta)=\{\x\in\H: \eta(U_{\x})>0 \ \text{ for every open neighborhood } \ U_{\x} \text{ of } \x\}.
 \end{align*}
 Since the transition semigroup $\P_t$ associated with \eqref{4p2} is irreducible (\cite[Subsection 3.3]{akmtm}), it follows that the invariant measure $\eta$ is full (\cite{gdp2}). By definition, it says that $\eta(U)>0$ for any open set $U$. In particular, it means $\mathrm{supp}(\eta)=\H$. Now, for any $r>0$, let us consider a ball $B_r(\x_0):=\{\x\in\H:\|\x-\x_0\|_{\H}<r\}$ centered at $\x_0\in\H$. Then, from \eqref{qa1}, we write
 \begin{align*}
 	\int_{B_r(\x_0)}|\zeta^{-}(\x)|^2\eta(\d\x)+\int_{\H\setminus B_r(\x_0)}|\zeta^{-}(\x)|^2\eta(\d\x)\leq0.
 \end{align*}
 On leaving the second integral in the left hand side of above inequality, we obtain
 \begin{align*}
 	\int_{B_r(\x_0)}|\zeta^{-}(\x)|^2\eta(\d\x)\leq0,
 \end{align*}
which in turn implies $\zeta^{-}(\x)\leq0$ for all $\x\in\H$, by continuity of $\zeta$. Thus, $\zeta^{-}(\x)=0$ for all $\x\in\H$, and hence it follows that
\begin{align}\label{oi4}
	\zeta(\x)=\zeta^{+}(\x)\geq0 \ \text{ for all } \ \x\in\H.
\end{align}

 Now, from \eqref{oi2}, it follows that (\cite{VBSSO})
 \begin{align}\label{oi3}
 \int_{\mathcal{H}_0}|\zeta(\x)|^2\eta(\d\x)=0.
 \end{align}
Again by the continuity of $\zeta$ and the fact that invariant measure $\eta$ is full, we deduce that
 \begin{align*}
 	\zeta(\x)=0 \ \text{ for all } \ \x\in\mathcal{H}_0.
 \end{align*}
\end{remark}

\begin{proof}[Proof of Theorem \ref{main-thm}]
	Under assumptions of Theorem \ref{main-thm}, we prove that the following multi-valued operator
	\begin{align}
		\mathscr{A}(\varphi)&:=-\mathcal{N}_2\varphi+N_K(\varphi)\ \text{ for all }\ \varphi\in\D(\mathscr{A}):=\D(N)\cap K
	\end{align}
	is $m$-accretive in $\L^2(\H;\eta)$. That is, 
	\begin{align}\label{459}
		\int_{\H}(\mathscr{A}(\varphi)-\mathscr{A}(\psi))(\varphi-\psi)\eta(\d \x)\geq 0\ \text{ for all }\ \varphi,\psi\in\D(\mathscr{A})
	\end{align}
	and 
	\begin{align}\label{460}
		\mathscr{R}(\I+\lambda\mathscr{A})=\L^2(\H;\eta)\ \text{ for all }\ \lambda>0,
	\end{align}
	where $\mathscr{R}$ represents  the range of the operator and $\I$ is the identity operator.
	
	The ``Carre du Champ's" identity \eqref{453} and the linearity of the operator $\mathcal{N}_2$ imply 
	\begin{align}\label{461}
		\int_{\H}(\mathcal{N}_2\varphi-\mathcal{N}_2\psi)(\varphi-\psi)\eta(\d \x)=-\frac{1}{2}\int_{\H}\|\sqrt{\Q}\D_{\x}(\varphi(\x)-\psi(\x))\|_{\H}^2\eta(\d\x)\leq 0. 
	\end{align}
It has been shown in \cite[Theorem 2.1, Chapter 2, pp. 62]{VB2}, that when $X$ is a Banach space and $f : X \to\mathbb{R}$ is proper convex
lower semicontinuous, then its convex subdifferential $\partial f : X \to 2^{X^*}$, defined by $x^*\in \partial f(x)$ if $f(x)$ is finite and for every $y \in X$, $f(y)\geq f(x) + \langle y-x, x^*\rangle$, is maximal monotone. In particular, the normal cone to $C$ which is given by $N_C = \partial I_K $ is maximal monotone (\cite[Section 2.2, Chapter 2, pp. 60]{VB2}), whenever $C\subset X$ is closed and  convex. Here $I_C (x) = 0,$ for $x \in C$; $I_C(x) = +\infty$,  for $x \in X \backslash C$ denotes the indicator function of $C$. As $\D(\partial I_K)=K$ is closed and convex, from the definition of normal cone given in \eqref{4p54}, we have $N_K$ is maximal monotone and 
\begin{align}\label{462}
	\int_{\H}(N_K(\varphi)-N_K(\psi))(\varphi-\psi)\eta(\d \x)\geq 0. 
\end{align}
Combining \eqref{461} and \eqref{462}, one can easily obtain \eqref{459}. 

Now it is left to show \eqref{460} only. In order to prove \eqref{460}, it is sufficient to establish (\cite[Theorem 1.8, Chapter IV, pp. 186]{VB1})
\begin{align}\label{464}
	(\I-\lambda\mathcal{N}_2)^{-1}K\subset K
\end{align}
Let $g\in K $ be arbitrary, but fixed.  Using the resolvent representation formula, we observe 
\begin{align*}
	(\I-\lambda\mathcal{N}_2)^{-1}g(\x)&= \frac{1}{\lambda}\int_0^{\infty}e^{-\frac{t}{\lambda}}P_tg(\x)\d t=\frac{1}{\lambda}\int_0^{\infty}e^{-\frac{t}{\lambda}}\E\left[g(\X(t,\x))\right]\d t\ \text{ for all }\ \x\in\H,
\end{align*}
where $\X(t,\x)$ is the solution to the stochastic problem \eqref{4p2}. Since $g\in K$ and using the definition of $K$ in \eqref{45}, we immediately have 
\begin{align}\label{466}
	(\I-\lambda\mathcal{N}_2)^{-1}g(\x)\leq \frac{1}{\lambda}\int_0^{\infty}e^{-\frac{t}{\lambda}}\E\left[G(\X(t,\x))\right]\d t\ \text{ for all }\ \lambda>0. 
\end{align}
But by the assumption \eqref{456}, we know that 
\begin{align*}
\E\left[G(\X(t,\x))\right]\leq G(\x)\ \text{ for all }\ t\in[0,T]\ \text{ and }\ \x\in\H. 
\end{align*}
Therefore, from \eqref{466}, it is immediate that 
\begin{align*}
		(\I-\lambda\mathcal{N}_2)^{-1}g(\x)\leq G(\x)\ \text{ for all }\ \x\in\H, \ \lambda>0.
\end{align*}
so that  $(\I-\lambda\mathcal{N}_2)^{-1}g\in K$ and hence \eqref{464} follows. Since $\mathcal{N}_2$ is $m$-dissipative and single valued, therefore from \cite[Theorem 1.6 (ii), Chapter III, pp. 118]{VB1}, for every $\varphi_0\in\D(\mathcal{N}_2)\cap K$, there exists a unique function $\varphi(t)\in\D(\mathcal{N}_2)$ for every $t\in[0,T]$ which is Lipschitz continuous on $[0,T]$ and everywhere differentiable from the right satisfying the following:
\begin{equation}\label{essnorm1}
	\left\{
	\begin{aligned}
	   \frac{\d^+}{\d t}\varphi(t)&=\mathcal{N}_2\varphi(t), \ t\in(0,T),\\
	   \varphi(0)&=\varphi_0.
	\end{aligned}
	\right.
\end{equation}
and 
\begin{align}\label{essnorm}
	\left\|\frac{\d^+}{\d t}\varphi(t)\right\|_{\mathbb{L}^2(\H;\eta)} \leq\|\mathcal{N}_2\varphi_0\|_{\mathbb{L}^2(\H;\eta)}.
\end{align}
Now, since the condition \eqref{464} holds, therefore from \cite[Theorem 1.8(c), Chapter IV, pp. 186]{VB1}, we have $\mathscr{A}^0(\varphi_0)=(\mathcal{N}_2+N_K)^0(\varphi_0)=\mathcal{N}_2^0\varphi_0$,  
where $\mathscr{A}^0$ is the minimal section of the operator $\mathscr{A}$, that is, 
\begin{align*}
	\|\mathscr{A}^0(\varphi)\|_{\L^2(\H;\eta)}=\inf_{\Psi\in\mathscr{A}(\varphi)}\left\{\|\Psi\|_{\L^2(\H;\eta)}\right\}. 
\end{align*}
Thus \eqref{essnorm1} and \eqref{essnorm} imply that
\begin{align*}
	\|\mathcal{N}_2\varphi(t)\|_{\L^2(\H;\eta)}\leq\|\mathscr{A}^0(\varphi_0)\|_{\L^2(\H;\eta)}\ \text{ for all }\ \varphi_0\in \D(\mathcal{N}_2)\cap K,
\end{align*}
for all $t\in[0,T]$.
The proof of Theorem \ref{main-thm} can be completed by applying  the standard existence and uniqueness theorem for nonlinear Cauchy problems in Hilbert spaces of accretive type (cf. \cite[Theorem 1.6, Chapter III, pp. 118]{VB1})
\end{proof}

An immediate consequence of  Theorem \ref{main-thm} and \eqref{453} is the following corollary:
\begin{corollary}
	Under assumptions of Theorem \ref{main-thm}, the solution $\varphi$ to \eqref{454} satisfies 
	\begin{align*}
		\sqrt{\Q}\D_{\x}\varphi\in\mathrm{L}^{\infty}(0,T;\L^2(\H;\eta)),
	\end{align*}
	where $\sqrt{\Q}\D_{\x}\varphi(\x)=\sum\limits_{k=1}^{\infty}\D_k\varphi(\x)\sqrt{\lambda_k}\boldsymbol{e}_k(\x)$, $\eta$ a.e. in $\H$, $\Q\boldsymbol{e}_k=\lambda_k\boldsymbol{e}_k$. 
\end{corollary}

\begin{remark}
	It should be noted that if $G\in\C_b^2(\H)$, then by the infinite-dimensional It\^o formula (cf. \eqref{450}), we have 
	\begin{align*}
		\d G(\X(t,\x))&= -(\mu \A\X(t,\x)+\alpha\X(t,\x)+\B(\X(t,\x))+\beta\mathcal{C}(\X(t,\x)),\D_{\x}\varphi(\X(t,\x)))\d t\nonumber\\&\quad+\frac{1}{2}\Tr(\Q\D^2_{\x}\varphi(\X(t,\x)))\d t+(\sqrt{\Q}\d\W(t),\D_{\x}G(\X(t,\x))),
	\end{align*}
	so that 
	\begin{align*}
	&	(P_tG)(\x)-G(\x)\nonumber\\&=-\E\left[\int_0^t(\mu \A\X(s,\x)+\alpha\X(s,\x)+\B(\X(s,\x))+\beta\mathcal{C}(\X(s,\x)),\D_{\x}\varphi(\X(s,\x)))\d s\right]\nonumber\\&\quad+\frac{1}{2}\int_0^t\Tr(\Q\D^2_{\x}\varphi(\X(s,\x)))\d s,
	\end{align*}
for all $t\in[0,T]$.	Therefore, the condition \eqref{456}  is implied in this case by the following one:
	\begin{align*}
		\frac{1}{2}\Tr\left[\Q\D_{\x}^2G(\x)\right]-(\mu\A\x+\alpha\x+\B(\x)+\beta\mathcal{C}(\x),\D_{\x}G(\x))\leq 0\ \text{ for all }\ \x\in\D(\A). 
	\end{align*}
	More generally, the condition \eqref{456} holds if $G \in\D(\mathcal{N}_2)$  and
	$$\mathcal{N}_2 G\leq 0 \ \text{ on }\ \H,$$
	which means that $G$ is \emph{super-harmonic} with respect to $\mathcal{N}_2$. 
\end{remark}

\begin{example}{\cite[Remark 1]{VBSSO}}
	One can consider $F$ and $G$ of the form
	\begin{align*}
		F(t,\x)\equiv\|\nabla\x\|_{\H}^2\ \text{ and }\ G(t,\x)=k(\|\x\|_{\H}^2)\|\x\|_{\H}^2,
	\end{align*}
	where $k\in\C^2(\R^+)$ is such that 
	\begin{align*}
		0\leq k(r)\leq \kappa r, \ k''\leq 0,\ k'>0, \ k(0)=0.
	\end{align*}
	A simple calculation then demonstrates that, provided $\mu$  and $\kappa$ are both suitably large and small, condition \eqref{456} is satisfied. Note that $F$ is  the enstrophy (total vorticity) in the flow field since 	$$\|\nabla\x\|_{\H}^2=\|\mathrm{curl}\ \x\|_{\H}^2\ \text{ for all }\ \x\in\V.$$ We infer from the definition of invariant measure $\eta$ that 
	\begin{align*}
		\int_{\H}(\mathcal{N}_2\psi)(\x)\eta(\d\x)=0\ \text{ where }\ \psi(\x)=\|\x\|_{\H}^2.
	\end{align*}
	Therefore, it is immediate from \eqref{4p5}, that 
	\begin{align*}
		2\mu\int_{\H}\|\nabla\x\|_{\H}^2\eta(\d\x)+2\alpha\int_{\H}\|\x\|_{\H}^2\eta(\d\x)+2\beta\int_{\H}\|\x\|_{\wi\L^{r+1}}^{r+1}\eta(\d\x)=\Tr(\Q),
	\end{align*}
	which implies that $\|\nabla\x\|_{\H}\in \L^2(\H;\eta)$ and  the enstropy is integrable with respect to the invariant measure.
\end{example}

 \medskip\noindent
{\bf Acknowledgments:} The first author would like to thank Ministry of Education, Government of India-MHRD for financial assistance. 

\medskip\noindent	\textbf{Declarations:} 

\noindent 	\textbf{Ethical Approval:}   Not applicable 


\noindent  \textbf{Conflict of interest: }On behalf of all authors, the corresponding author states that there is no conflict of interest.

\noindent 	\textbf{Authors' contributions: } All authors have contributed equally. 

\noindent 	\textbf{Funding: } DST-SERB, India, MTR/2021/000066 (M. T. Mohan). 

\noindent 	\textbf{Availability of data and materials: } Not applicable.


\begin{thebibliography}{99}
	
	
	  \bibitem{Aand} A. Andersson, M. Hefter, A. Jentzen and R. Kurniawan, Regularity properties for solutions of infinite-dimensional Kolmogorov equations in Hilbert spaces, \emph{Potential Anal.}, \textbf{50}(3) (2019), 347--379.
	
	  \bibitem{ware} W. Arendt, J. K. C. Batty, M. Hieber and F. Neubrander, \emph{Vector-Valued Laplace Transforms and Cauchy Problems}, Second edition, Monographs in Mathematics, \textbf{96}, Birkh\"auser/Springer Basel AG, 2011.
	
	   
        
      \bibitem{VB1} V. Barbu, \emph{Nonlinear Semigroups and Differential Equations in Banach Spaces}, Noordhoff International Publishing, 1976.
        
      \bibitem{VB2} V. Barbu, \emph{Analysis and Control of Nonlinear Infinite Dimensional Systems}, Academic Press, 1993.
	
	  \bibitem{VBGD} V. Barbu, G. Da Prato and A. Debussche, The Kolmogorov equation associated to the stochastic Navier-Stokes equations in 2D, \emph{Infinite Dimensional Analysis, Quantum Probability and Related Topics}, 
		{\bf 7}(2) (2004), 163--182. 
		
		\bibitem{VBGA} V. Barbu, G. D. Prato and A. Debussche,
		Essential m-dissipativity of Kolmogorov operators corresponding to periodic 2D-Navier Stokes equations, 
		\emph{Atti Accad. Naz. Lincei Cl. Sci. Fis. Mat. Natur. Rend. Lincei (9) Mat. Appl.}, \textbf{15}(1) (2004), 29--38.
		
		\bibitem{VbGdp} V. Barbu and G. D. Prato, The Kolmogorov equation for a 2D-Navier-Stokes stochastic flow in a channel, \emph{Nonlinear Anal.}, \textbf{69} (3) (2008), 940--949.
		
		\bibitem{VBCM} V. Barbu and C. Marinelli, Variational inequalities in Hilbert spaces with measures and optimal stopping problems, \emph{Appl. Math. Optim.}, \textbf{57}(2) (2008), 237--262.
		
		\bibitem{VBSSO} V. Barbu and S. S. Sritharan, Optimal stopping-time problem for stochastic Navier-Stokes equations and infinite-dimensional variational inequalities,
		\emph{Nonlinear Anal.}, \textbf{64}(5) (2006), 1018--1024.
		
		
		\bibitem{OPHB} H. Brezis, \emph{Operateurs Maximaux et Semi-groupes de Contractions das les Espaces de Hilbert,} North Holland, New York, 1973.
	
	  
	  
	  \bibitem{ABJL2} A. Bensoussan and J. L.  Lions, \emph{Applications of Variational Inequalities in Stochastic Control}, Studies in Mathematics and its Applications, \textbf{12}, North-Holland Publishing Co., Amsterdam-New York, 1982.
	  
%
	  
	  \bibitem{ceb} C. E. Br\'ehier and A. Debussche,
	  Kolmogorov equations and weak order analysis for SPDEs with nonlinear diffusion coefficient, \emph{J. Math. Pures Appl. (9)},
	  \textbf{119} (2018), 193--254.
	  
	
      
	  \bibitem{pcgd} P. Cannarsa and G. D. Prato, Second order Hamilton-Jacobi equations in infinite dimensions, \emph{SIAM J. Control Optim.}, \textbf{29}(2) (1991), 474--492.
	
	 \bibitem{pcgdp1} P. Cannarsa and G. Da Prato, Direct solution of a second order Hamilton-Jacobi equation in Hilbert spaces, \emph{Stochastic Partial Differential Equations and Applications}, Pitman Research Notes in Mathematics, Vol. \textbf{268}, pp. 72--85, Pitman, London, 1992.
	
	
	\bibitem{SC1} S. Cerrai, \emph{Second order PDE's in finite and infinite dimension: A probabilistic approach}, Lecture Notes in Mathematics, 1762, Springer-Verlag, Berlin, 2001.
	
	\bibitem{SC2} S. Cerrai, A Hille-Yosida theorem for weakly continuous semigroups, \emph{Semigroup Forum}, \textbf{49}(3) (1994), pp. 349-367.
    
    \bibitem{EANK} E. Charpentier, A. Lesne and N. K. Nikolski, \emph{Kolmogorov's Heritage in Mathematics}, translated from the 2004 French original, Springer, Berlin, 2007. 
    
    \bibitem{MBCTD} M. B. Chiarolla and T. D. Angelis, Optimal stopping of a Hilbert space valued diffusion: An infinite-dimensional variational inequality, \emph{Appl. Math. Optim.}, \textbf{73}(2) (2016),  271--312.
    
	\bibitem{AC} A. Chorin, \emph{A Mathematical Introduction to Fluid Mechanics}, Springer-Verlag, 1992. 
	
	
	
	
%
%
	
   \bibitem{gdp} G. Da Prato and J. Zabczyk, \emph{Stochastic Equations in Infinite Dimensions},
   Second edition, Encyclopedia of mathematics and its applications, 152, Cambridge University Press, Cambridge, 2014. 
   
   
   \bibitem{gdp1} G. Da Prato, \emph{Kolmogorov Equations for Stochastic PDEs}, 
   Advanced Courses in Mathematics, CRM Barcelona, Birkhäuser Verlag, Basel, 2004.
   
   
   
   \bibitem{gdp7} G. Da Prato and J. Zabczyk, \emph{Second order partial differential equations in Hilbert spaces}, London Mathematical Society Lecture Note Series, \textbf{293}, Cambridge University Press, Cambridge, 2002.
   
   \bibitem{gdp2} G. Da Prato and J. Zabczyk, \emph{Ergodicity for Infinite-Dimensional Systems}, London Mathematical Society Lecture Note Series, Cambridge University Press, Cambridge, 1996.
   
   
   
       \bibitem{gPaD} G. Da Prato and A. Debussche, Dynamic programming for the stochastic Navier-Stokes equations,
     Special issue for R. Temam's 60th birthday, \emph{M2AN Math. Model. Numer. Anal.}, \textbf{34} (2) (2000), 459--475.
   
   

	\bibitem{EBD} E. B. Davies, \emph{One-Parameter Semigroups},
	London Mathematical Society Monographs, Academic Press, London-New York, 1980.
	
%
	
	 \bibitem{AD} A. Debussche, Ergodicity results for the stochastic Navier-Stokes equations: An introduction, \emph{Topics in Mathematical Fluid Mechanics}, Volume 2073 of the series Lecture Notes in Mathematics, Springer, 23--108, 2013.
	
	
	
	
	
	\bibitem{FGS} G. Fabri, F. Gozzi, and A. Swiech, \emph{Stochastic Optimal Control in Infinite Dimension: Dynamic Programming and HJB Equations}, Springer-Verlag, New York, 2018.

	
	\bibitem{WHF} W. H. Fleming, Optimal continuous-parameter stochastic control, \emph{SIAM Rev.}, \textbf{11} (1969), 470--509.
	
	\bibitem{AvF} A. Friedman, Optimal stopping problems in stochastic control, \emph{SIAM Rev.}, \textbf{21}(1) (1979), 71--80.
	
	\bibitem{gozzi} F. Gozzi, S. S. Sritharan and A. Swiech, Bellman equations associated to the optimal feedback control of stochastic Navier-Stokes equations, \emph{Comm. Pure Appl. Math.}, \textbf{58}(5) (2005), 671--700.
	
	\bibitem{gozzi2} F. Gozzi, Regularity of solutions of a second order Hamilton-Jacobi equation and application to a control problem, \emph{Comm. Partial Differential Equations}, \textbf{20} (5-6) (1995), 775--826.
	
	\bibitem{gozzi1} F. Gozzi, Global regular solutions of second order Hamilton-Jacobi equations in Hilbert spaces with locally Lipschitz nonlinearities, \emph{J. Math. Anal. Appl.}, \textbf{198}(2) (1996), 399--443.
	
	\bibitem{gozzi3} F. Gozzi and E. Rouy, Regular solutions of second-order stationary Hamilton-Jacobi equations, \emph{J. Differential Equations}, \textbf{130}(1) (1996), 201--234.
	
	
	
	
	
	 \bibitem{KWH} K. W. Hajduk and J. C. Robinson, Energy equality for the 3D critical convective Brinkman-Forchheimer equations, \emph{J. Differential Equations}, \textbf{263}(11) (2017), 7141--7161.
	 
	\bibitem{MHJC}   M. Hairer, J.C. Mattingly, Ergodicity of the 2D Navier-Stokes equations with degenerate stochastic forcing, \emph{Annals of Mathematics}, {\bf 164} (2006), 993--1032.
	
	\bibitem{MH1} M. Hairer, M. Hutzenthaler and A. Jentzen, Loss of regularity for Kolmogorov equations, \emph{Ann. Probab.}, \textbf{43}(2) (2015), 468--527.
	 
	 \bibitem{KT2} V. K. Kalantarov and S. Zelik, Smooth attractors for the Brinkman-Forchheimer equations with fast growing nonlinearities, \emph{Commun. Pure Appl. Anal.}, \textbf{11}(5) (2012), 2037--2054.
	
	\bibitem{kkmtm} K. Kinra and M. T. Mohan, Random attractors and invariant measures for stochastic convective Brinkman-Forchheimer equations on 2D and 3D unbounded domains, \emph{Discrete Contin. Dyn. Syst. Ser. B}, \textbf{29}(1) (2024), 377--425.
	
	\bibitem{ANK1} A. Kolmogoroff, 
	\"Uber die analytischen Methoden in der Wahrscheinlichkeitsrechnung, \emph{Math. Ann.}, \textbf{104}(1) (1931), 415--458.
    
	\bibitem{akmtm} A. Kumar and M. T. Mohan, Large deviation principle for occupation measures of two dimensional stochastic convective Brinkman-Forchheimer equations, \emph{Stoch. Anal. Appl.}, \textbf{41}(2) (2023), 214--256. 
	
	\bibitem{OAL}	O. A. Ladyzhenskaya, \emph{The Mathematical Theory of Viscous Incompressible Flow}, Gordon and Breach, New York, 1969.
	
	
	
	
%
	
	
	
	
	
		
	\bibitem{MTM8} M. T. Mohan, Stochastic convective Brinkman-Forchheimer equations, \url{https://arxiv.org/pdf/2007.09376}. 
	
	   \bibitem{MT2}  M. T. Mohan, Well-posedness and asymptotic behavior of stochastic convective Brinkman-Forchheimer equations perturbed by pure jump noise, \emph{Stoch PDE: Anal. Comp.}, {\bf 10}(2) (2022), 614--690.
	
	\bibitem{MoSS2} {M. T. Mohan, K. Sakhtivel and S. S. Sritharan,} Dynamic programming for the stochastic 2D-Navier-Stokes equations forced by L\'{e}vy noise, 2024, \emph{Math. Control Relat. Fields}. \url{10.3934/mcrf.2024016}
	
%
	 
	
    
	
	\bibitem{oks} B. Øksendal, Stochastic Differential Equations:
	An Introduction with Applications, Sixth edition, Universitext. Springer-Verlag, Berlin, 2003.
	
	\bibitem{MO} M. Ondrej\'at, Brownian representations of cylindrical local martingales, martingale problem and
	strong Markov property of weak solutions of SPDEs in Banach spaces, \emph{Czechoslovak Math. J.}, \textbf{55}(4) (2005), 1003--1039.
	
	\bibitem{KRP} K. R. Parthasarathy, \emph{Probability Measures on Metric Spaces}, vol. \textbf{3}, Academic Press, New York-London, 1967.
	
	\bibitem{gpas} G. Peskir and A. Shiryaev, Optimal stopping and free-boundary problems, Lectures in Mathematics ETH Zürich. Birkh\"user Verlag, Basel, 2006.
	
	
	 
	 \bibitem{mr13} M. R\"ockner, $L^p$-analysis of finite and infinite-dimensional diffusion operators, \emph{Stochastic PDE's and Kolmogorov equations in infinite dimensions}, Lecture Notes in Math, vol. 1715, pp. 65--116. Springer, Berlin (1999).
	 
	 \bibitem{mr14} M. R\"ockner and Z. Sobol,
	 A new approach to Kolmogorov equations in infinite dimensions and applications to the stochastic 2D Navier-Stokes equation,
	 \emph{C. R. Math. Acad. Sci. Paris}, \textbf{345}(5) (2007), 289--292.
	 
	
	\bibitem{WSt} W. Stannat, $L^p$-uniqueness of Kolmogorov operators associated with 2D-stochastic Navier-Stokes-Coriolis equations, \emph{Math. Nachr.}, \textbf{284}(17-18) (2011), 2287--2296.
	
	
	\bibitem{TT2} T. Tobias, Optimal stopping of diffusion processes and parabolic variational inequalities, \emph{Differencial'nye Uravnenija}, \textbf{9} (1973), 702--708.
	
	
	
	\bibitem{jz13} J. Zabczyk, Parabolic equations on Hilbert spaces, \emph{Stochastic PDE's and Kolmogorov equations in infinite dimensions}, Lecture Notes in Math, vol. 1715, pp. 65--116. Springer, Berlin (1999).
	
\end{thebibliography}
\end{document}